\documentclass[11pt, a4paper]{article}
\usepackage[utf8]{inputenc}
\usepackage{mathtools}
\usepackage{amsmath}
\usepackage{amsthm}
\usepackage{algorithm2e}
\usepackage{algorithmic}
\RestyleAlgo{ruled}
\SetKwInput{KwInput}{Input}
\SetKwInput{KwOutput}{Output}
\usepackage{algorithmic}
\usepackage{amssymb}
\usepackage{parskip}
\usepackage{epsfig}
\usepackage[colorlinks, citecolor=hanblue,linkcolor=mordantred19,urlcolor=darkgreen]{hyperref}

\usepackage{stackengine}
\mathtoolsset{showonlyrefs}
\usepackage{xcolor}
\usepackage[a4paper, margin=2.5cm]{geometry}
\usepackage{comment}
\usepackage{bbm}
\usepackage{mathrsfs}
\usepackage{subfigure}
\usepackage{float}
\usepackage{tikz, pgfplots}
\usetikzlibrary{decorations.pathreplacing,calligraphy}
\pgfplotsset{compat=1.18}
\usepackage[export]{adjustbox}
\DeclareMathOperator*{\argmax}{arg\,max}
\DeclareMathOperator*{\argmin}{arg\,min}
\usepackage{multirow}
\usepackage{setspace}

\newcommand{\mres}{\mathbin{\vrule height 1.4ex depth 0pt width
0.13ex\vrule height 0.13ex depth 0pt width 1.0ex}}

\usepackage{color}
\definecolor{hanblue}{rgb}{0.27, 0.42, 0.81}
\definecolor{mordantred19}{rgb}{0.68, 0.05, 0.0}
\definecolor{darkgreen}{rgb}{0.0, 0.38, 0.12}
\definecolor{red}{rgb}{0.8, 0.0, 0.0}
\definecolor{green}{rgb}{0.0, 0.5, 0.0}

\newcommand{\B}{\mathcal{B}}

\newcommand{\ur}{\mathring{u}}

\DeclareMathOperator{\diam}{diam}

\DeclareMathOperator{\supp}{supp}
\DeclareMathOperator{\Id}{Id}

\usepackage{mathtools}
\usepackage{hyperref}
\mathtoolsset{showonlyrefs}

\renewcommand{\div}{{\rm div}\,}
\newcommand{\R}{\mathbb{R}}
\newcommand{\1}{\mathbbm{1}}
\newcommand{\Z}{\mathbb{Z}}

\newcommand{\N}{\mathbb{N}}
\renewcommand{\L}{\mathcal{L}}
\newcommand{\cC}{\mathcal{C}}
\renewcommand{\S}{\mathcal{S}}
\newcommand{\T}{\mathcal{T}}

\renewcommand{\H}{\mathcal{H}}

\newcommand{\BV}{\operatorname{BV}}
\newcommand{\BVO}{\operatorname{BV}(\Omega)}
\newcommand{\LO}{L^q(\Omega)}
\newcommand{\LrO}{\mathring{L}^q(\Omega)}
\newcommand{\Vr}{\mathring{V}}
\newcommand{\Prc}{\mathring{P}_0(\mathcal{T})}
\newcommand{\Prl}{\mathring{P}_1(\mathcal{T})}
\newcommand{\TV}{\operatorname{TV}}
\newcommand{\Per}{\operatorname{Per}}
\newcommand{\ce}{\operatorname{ce}}
\newcommand{\w}{\widetilde}
\newcommand{\wkto}{\rightharpoonup}
\newcommand{\wksto}{\stackrel{\ast}{\rightharpoonup}}

\newcommand{\Ext}{\operatorname{Ext}}
\newcommand{\Span}{\operatorname{span}}
\newcommand{\dd}{\, \mathrm{d}}

\newcommand\restr[2]{{% we make the whole thing an ordinary symbol
  \left.\kern-\nulldelimiterspace % automatically resize the bar with \right
  #1 % the function
  \vphantom{\big|} % pretend it's a little taller at normal size
  \right|_{#2} % this is the delimiter
  }}

\numberwithin{equation}{section}

\theoremstyle{plain}

\newtheorem{theorem}{Theorem}[section]

\newtheorem{lemma}[theorem]{Lemma}
\newtheorem{proposition}[theorem]{Proposition}

\newtheorem{assumption}{Assumption}
\newtheorem{assumptionalt}{Assumption}[assumption]
\newenvironment{assumptionp}[1]{
  \renewcommand\theassumptionalt{#1}
  \assumptionalt
}{\endassumptionalt}
\newtheorem{remark}[theorem]{Remark}

\theoremstyle{definition}
\newtheorem{definition}[theorem]{Definition}

\definecolor{hanblue}{rgb}{0.27, 0.42, 0.81}

\usepackage{scalerel}

\title{Conditional gradients for total variation regularization with PDE constraints: a graph cuts approach}
\author{Giacomo Cristinelli$^\ast$, Jos\'e A. Iglesias\thanks{Department of Applied Mathematics, University of Twente, 7500AE Enschede, The Netherlands \newline (\texttt{g.cristinelli@utwente.nl, jose.iglesias{@}utwente.nl})} , Daniel Walter\thanks{Institut f\"ur Mathematik, Humboldt-Universit\"at zu Berlin, 10117 Berlin, Germany \newline (\texttt{daniel.walter@hu-berlin.de})}}
\date{}

\begin{document} 
\maketitle

\begin{abstract}
Total variation regularization has proven to be a valuable tool in the context of optimal control of differential equations. This is particularly attributed to the observation that TV-penalties often favor piecewise constant minimizers with well-behaved jumpsets. On the downside, their intricate properties significantly complicate every aspect of their analysis, from the derivation of first-order optimality conditions to their discrete approximation and the choice of a suitable solution algorithm. In this paper, we investigate a general class of minimization problems with TV-regularization, comprising both continuous and discretized control spaces, from a convex geometry perspective. This leads to a variety of novel theoretical insights on minimization problems with total variation regularization as well as tools for their practical realization. First, by studying the extremal points of the respective total variation unit balls, we enable their efficient solution by geometry exploiting algorithms, e.g. fully-corrective generalized conditional gradient methods. We give a detailed account on the practical realization of such a method for piecewise constant finite element approximations of the control on triangulations of the spatial domain. Second, in the same setting and for suitable sequences of uniformly refined meshes, it is shown that minimizers to discretized PDE-constrained optimal control problems approximate solutions to a continuous limit problem involving an anisotropic total variation reflecting the fine-scale geometry of the mesh.
\end{abstract}
 

\section{Introduction}
In this paper, for $\Omega \subset \R^d$ a bounded domain with strongly Lipschitz boundary, we consider minimization problems of the form
\begin{align} \label{def:BVprob}
    \min_{u \in V} J(u):= \left\lbrack  F(K u)+ \TV(u,\Omega)\right \rbrack \tag{$\mathcal{P}$},
\end{align}
where $V \subset L^q(\Omega)$ is a weakly closed linear subspace of $L^q(\Omega)$, $q = d/(d-1)$, and the objective functional consists of two parts: First, the composition of a smooth functional~$F$ with a linear continuous operator $K : V \to Y$  mapping to a Hilbert space of observations $Y$. Second, the total variation~$\TV(u,\Omega)$ of a function~$u \in V$. While this general setting encompasses denoising problems,~\cite{RudOshFat92}, by setting $K=\operatorname{Id}$, our particular focus lies on instances in which the forward mapping $K$ is compact and potentially costly to evaluate. As a guiding example, we point out optimal control problems of the form 
\begin{align} \label{eq:PDEconstraint}
    \min_{u,y}\left\lbrack \frac{1}{2} \|y-y_o\|^2_{L^2(\Omega)}+  \alpha\|D u\|_{M(\Omega)}\right \rbrack \tag{$\mathcal{D}$}
\end{align}
where~$\alpha>0$,~$\|D u\|_{M(\Omega)} = \TV(u,\Omega)$ denotes the norm of the distributional gradient $D u$ seen as a vector-valued Radon measure and the control and state variables $u$ and~$y$ are coupled by the PDE-constraint 
\begin{align} \label{eq:PDE} 
    -\div (a \nabla y)+cy=u \text{ in }\Omega, \quad  y=0  \text{ on }\partial \Omega. 
\end{align}
Imposing suitable regularity assumptions on the coefficient functions~$a$ and~$c$, respectively, Problem~\eqref{eq:PDEconstraint} fits into the general setting~\eqref{def:BVprob} by choosing $F(\cdot)=(1/2\alpha)\|\cdot - y_o\|^2_{L^2(\Omega)}$ and introducing the control-to-state operator~$y=Ku$ which maps the right hand side of a Poisson-type problem to its unique solution. Problems of this or similar form have received considerable interest since the total variation regularization favors minimizers which only attain a finite number of values on~$\Omega$. We refer e.g. to~\cite{Berg14,Clason18,Casas17,Pola98,Engel19} without pretence of completeness. However, this desirable structure of optimal solutions comes at the cost of additional difficulties such as the non-reflexivity of the space~$\BV(\Omega)$, the nonsmoothness of the total variation as well as the possible lack of strong convexity in Problem~\eqref{def:BVprob} due to the compactness of the observation operator~$K$. These complicate both the theoretical analysis of the problem as well as its algorithmic solution.

A possible approach to alleviating some of these difficulties can be found by interpreting the TV-term as the Minkowski functional of the set
\begin{align*}
B_{\Vr}=  \left\{\,u \in V\;\middle\vert\;\TV(u,\Omega)\leq 1,~\int_\Omega u(x)~\mathrm{d}x=0\,\right\}  
\end{align*}
and studying the convex geometry of the latter, in particular its set of extremal points~$\Ext(B_{\Vr})$. From this abstract perspective it turns out that Problem~\eqref{def:BVprob} encourages minimizers that are given, up to a constant shift, by conic combinations of few elements in~$\Ext(B_{\Vr})$. Moreover, existence of such sparse solutions is ensured by convex representer theorems if the range of~$K$ is finite dimensional,~\cite{Boyer19,BreCar20}. For~$V=L^q(\Omega)$, an explicit characterization of~$\Ext(B_{\Vr})$ is available which identifies them as scaled characteristic functions of simple sets. This closes the circle to the aforementioned practical observation of piecewise constant solutions attaining few values. Moreover, this geometric interpretation also offers the opportunity for the application of efficient solution algorithms which can capitalize on these observations and produce fast-converging iterates comprising few extremal points. For example, the fully-corrective generalized conditional gradient method (FC-GCG) from~\cite{BreCarFanWal23} relies on alternating between the greedy update of finite sets~$\mathcal{A}_k$ of extremal points by the solution of linear minimization problems over~$\Ext(B_{\Vr})$, and computing new iterates in the conic hull of~$\mathcal{A}_k$ by the solution of convex finite dimensional minimization problems with an~$\ell_1$-type penalty term.

However, in many cases, e.g. for Problems of the form~\eqref{eq:PDEconstraint}, the action of the observation operator~$K$ on a function~$u \in \BV$ cannot be computed analytically but requires the assistance of numerical methods and thus, ultimately, a discretization of the variable space. In the simplest case, this can be achieved by considering piecewise constant,~$ V=P_0(\T_h)$, or piecewise affine and continuous,~$V=P_1(\T_h)$, finite element spaces subordinate to a triangulation $\T_h$ with grid width~$h$ of $\Omega$ for the discretization of~$u$ and by replacing~$K$ by a finite dimensional aproximation.          
While this yields a finite dimensional minimization problem, many of the original challenges, e.g. the nonsmoothness of the TV-term, still remain. Moreover, neglecting the infinite dimensional character of the continuous Problem~\eqref{def:BVprob} and applying black box methods for the arising nonsmooth minimization problems harbors the danger of mesh-dependence, i.e., the convergence behavior of the chosen algorithm may quickly deteriorate with decreasing grid width. Finally, discretizing the problem also raises the question of the limiting behavior of its solutions and their convergence towards minimizers of the continuous problem as the grid width tends to zero. While the answer to this is affirmative for~$V=P_1 (\T_h)$, e.g. based on~\cite[Theorem 4.1]{BelLus03} or~\cite{Bar12}, it is well known  that the same does not hold true for~$V=P_0 (\T_h)$. More in detail, given a generic function~$u\in\BVO$ as well as a family of triangulations~$\{\T_h\}_{h>0}$ of~$\Omega$, there may be no~$\{u_h\}_{h>0}$ with~$u_h \in P_0(\T_h)$ and
\begin{align*}
    \lim_{h \rightarrow 0} \left \lbrack\|u-u_h\|_{L^1}+ \big|\TV(u,\Omega)-\TV(u_h,\Omega)\big| \right \rbrack=0,
\end{align*}
see e.g.~\cite{Bar12} or the counterexample~\cite[Example 3.6]{CasKunPol99}. As a consequence, accumulation points of sequences of minimizers~$\{\Bar{u}_h\}_{h>0}$ to~\eqref{def:BVprob} for~$V=P_0(\T_h)$ might not constitute minimizers of the corresponding continuous problem.   

\subsection{Contribution}
The present paper aims at contributing towards the efficient solution of discretized minimization problems with total variation regularization. Moreover, we further investigate the behavior of solutions to Problem~\eqref{def:BVprob} for piecewise constant approximations,~$V=P_0(\T_h)$, and vanishing grid width~$h>0$. In detail, the contributions of the present work are threefold: 
\vskip 0.1in
\textit{Convex geometry of total variation balls for~$P_0(\T)$ and~$P_1(\T)$}\\
Motivated by the recent interest in the interplay between nonsmooth minimization and convex geometry, we investigate the extremal points of the set~$B_{\Vr}$ for the particular choice of $V=P_0(\T)$ and~$V=P_1(\T)$, respectively. This yields two key observations: On the one hand, in the piecewise constant case, we show that~$\Ext(B_{\mathring{P}_0})=\Ext(B_{\mathring{L}^q}) \cap P_0(\T)$, i.e., piecewise constant discretizations are, loosely speaking, conforming with the convex geometry of the continuous ball~$\Ext(B_{\mathring{L}^q})$. In particular, this implies that~$B_{\mathring{P}_0}$ is a convex polytope. On the other hand, we also show that piecewise affine and globally continuous discretizations,~$V=P_1(\T_h)$, fail to capture the geometry of the limiting problem, since one clearly has~$\Ext(B_{\mathring{P}_1})\cap\Ext(B_{\mathring{L}^q}) =\emptyset$, and, more surprisingly, the sublevel set~$B_{\mathring{P}_1}$ is curved, thus admitting infinitely many extremal points.
\vskip 0.1in
\textit{A FC-GCG method for~$V=P_0(\T)$}\\
We propose a fully-corrective generalized conditional gradient method for the solution of minimization problems with total variation regularization. For this purpose, an equivalent formulation of Problem~\eqref{def:BVprob} is considered which introduces the mean value of~$u$ as a separable variable. The resulting problem fits into the abstract framework of~\cite{BreCarFanWal23} and thus allows for the immediate application of FC-GCG methods as well as of the presented worst-case convergence results therein. For the case of piecewise constant discretizations, we give a detailed description of its practical realization. Using the derived characterization of~$\Ext(B_{\mathring{P}_0})$, this requires the repeated solution of problems of the form
\begin{align} \label{def:fracintro}
    \max_{v \in P_0(\T_h)  } \frac{\int_{\Omega} p v~\mathrm{d}x}{\TV(v, \Omega)} \quad \text{where} \quad p \in P_0(\T_h), \quad \int_\Omega p ~\mathrm{d}x =0
\end{align}
is a given dual variable. While finite, the size of its admissible set might grow exponentially with~$h$, thus requiring iterative methods for its solution. Inspired by classical approaches from fractional minimization, in particular the so-called Dinkelbach method,~\cite{schaible}, we argue that~\eqref{def:fracintro} is equivalent to solving a finite number of minimal cut problems on the dual graph of the mesh. The practical efficiency of the resulting algorithm is demonstrated on two instances of the PDE-constrained minimization problem~\eqref{eq:PDEconstraint}, each of them in two and three dimensions.
\vskip 0.1in
\textit{Discrete-to-continuum limits for~$V=P_0(\T_h)$}\\
We show that for a sequence of periodically refined triangulations~$\{\T_k\}_{k \in\N}$ on the unit cube~$\Omega=[0,1]^d$, see Definition~\ref{def:periodictriang}, the restriction of the total variation onto~$P_0(\mathcal{T}_k)$, denoted by~$\TV_{\T_k}$, approximates a limiting energy of the form   
\begin{equation}
\TV_{\varphi_{\T_1}}(u, \Omega) := \sup \left\{ \int_\Omega u \,\div \psi \dd x \,\middle\vert\, \varphi^\ast_{\T_1}(\psi) \leq 1 \text{ for all } x \in \Omega\right\} = \int_\Omega \varphi_{\T_1} \left(\frac{Du}{|Du|}\right)\dd |Du|.
\end{equation}
in the sense of~$\Gamma$-convergence w.r.t the topology on~$L^1(\Omega)$. Here, the functional $\varphi_{\T_1}$ captures the anisotropy of the triangulation and is given implicitly by the optimal value function of a nonsmooth minimization problem. We emphasize, that in PDE-constrained control applications such as the ones we focus on, the main goal is to obtain solutions that attain few values. Consequently, we believe that the appearance of this anisotropy is not a serious concern. This stands in contrast to applications where isotropy is indeed desirable, such as in image processing and plasticity-related continuum mechanics problems (see for example \cite{IglMerSch20}). For the particular case of the usual ``double-diagonal'' triangulation, see Figure \ref{fig:latticegraph}, we explicitly characterize~$\varphi_{\mathcal{T}_1}$ by proving that the level set~$\varphi^{-1}_{\mathcal{T}_1}(1)$ is given by a regular octagon.
Finally, we point out that the convergence of the total variation immediately translates to minimizers of PDE-constrained problems of the form~\eqref{eq:PDEconstraint}. More in detail, we show that finite element discretizations of Problem~\eqref{eq:PDEconstraint} with~$P_0$ elements for the control and~$P_1$ elements for the state variable asymptotically approximate
\begin{align*}
       \min_{u,y}\left\lbrack \frac{1}{2} \|y-y_o\|^2_{L^2(\Omega)}+ \alpha \TV_{\varphi_{\T_1}}(u,\Omega)\right \rbrack \quad\text{ s.t. }\quad \eqref{eq:PDE}.
\end{align*}

\subsection{Related work}
\textit{Nonsmooth optimization \& convex geometry} \\
The empirical observation that the proper choice of a convex nonsmooth regularization function enhances desired structural features in the solutions of associated minimization problems has solidified their role as a cornerstone of modern approaches in optimal control, variational regularization in inverse problems as well as applications in data science and machine learning. In the case of one-homogeneous regularizers, this inherent property of the regularizer has recently, see e.g.~\cite{Boyer19, BreCar20, BreCarFanWal23}, been linked to the geometry of its, in a suitable sense compact, sublevel sets, in particular the corresponding set of extremal points. More in detail, nonsmooth regularization tends to favor sparse solutions, i.e., elements that are given by a conic combination of finitely many extremals. This fundamental insight has led to an increased interest in the explicit characterizations of the extremals points for practically relevant regularizers as well as well as numerical solution methods which capitalize on this sparsity by using the precise structure of the extremal points in their formulation. For an incomplete list of examples, we refer, e.g., to \cite{BrePik13, BreCarFanRom21, BreCarFanRom23, IglWal22, CarIglWal22, AmbAziBreUns24, ParNow21, Uns21}.
\vskip 0.1in
\textit{Solution algorithms for total variation regularized problems}\\
Due to the significance of total variation regularization in the context of mathematical imaging, there is a wide variety of different numerical algorithms for the solution of (discretized) problems of the form~\eqref{def:BVprob} with quadratic~$F$. Without pretence of completeness, we mention dual and primal-dual approaches, e.g. semismooth Newton and primal-dual active set methods,~\cite{Ng07}, as well as split Bregman iteration,~\cite{GoldOsh09}, and general splitting-type methods, e.g.~\cite{ChaPock11}. In \cite{BentBou18}, the authors propose a conditional gradient method for the solution of the dual problem associated to~\eqref{def:BVprob}. For more details and further reference, we point to the literature reviews~\cite{Caselles2015,ChanChan15}. While many of these approaches easily generalize to larger classes of functionals~$F$ and forward operators~$K$, others, like minimization of graph cut energies to obtain level sets separately, \cite{DarSig06, ChaDar09}, are inherently tied to the structure of image denoising problems. Moreover, we again point out that, in the context of mathematical imaging, the operator~$K$ is usually either the identity or an integral operator with given kernel and thus relatively cheap to evaluate. This constitutes a conceptual difference to the current work in which~$K$ is thought to be connected to the solution of a differential equation. Hence, extensive evaluation of~$K$ or its adjoint~$K^*$, as e.g. required in splitting type methods, can potentially create a computational bottleneck, see also the discussion in \cite{JenVal22}. Split Bregman methods for discretized PDE-constrained optimal control problems have been considered in~\cite{ElvNie16}, while, e.g.~\cite{ClaKun11,HafMan22}, put the focus on the analysis of (generalized) Newton methods for infinite dimensional instances of Problem~\eqref{eq:PDEconstraint} or the associated dual problem. Proving the well-posedness and local fast convergence convergence of the latter, however, requires further regularization and/or smoothing of the problem. This introduces new hyperparameters which require additional attention, both from a theoretical as well as from a practical perspective, due to the need of a suitable path-following strategy.  

In contrast, the method presented in this manuscript is both simple and practically efficient, achieving satisfactory results in a moderate number of steps at the cost of two PDE solves per iteration. It is closest related to the modified Frank-Wolfe algorithm in the recent work~\cite{CasDuvPet22} for off-the-grid image reconstruction on unbounded domains as well as the fully corrective conditional gradient in~\cite{HarJudNem15} for finite dimensional problems with TV-regularization. However, it differs from these earlier works in several key aspects, e.g. the specific form of the arising subproblems and/or their numerical treatment . We give a more detailed account of these differences in Remark~\ref{rem:nonmonotone} and~\ref{rem:comparison}, respectively. Finally, let us point out that FC-GCG methods for total variation regularization have already been considered in~\cite{TraWal23} for the one dimensional setting, i.e. $d=1$ and~$\Omega=(0,T)$. These results, however, rely on the identification of~$\BV(\Omega) \simeq M(\Omega; \R^d) \times \R $ which does not hold for~$d>1$.       
\vskip 0.1in
\textit{FE-approximation of optimal control problems \& discrete-to-continuum limits}\\
While there is a considerable amount of literature on finite element approximations of problems in imaging, see e.g.~\cite{Bar12,HerHer19,Hilb23,ChaPo21} as well as their asymptotic analysis, previous works in the context of PDE-constrained optimal control with total variation regularization have mainly focused on the one-dimensional case, i.e.~$d=1$,~\cite{Casas17,Engel21,Herberg23,Neitzel20}, or the finite element analysis of modified problems, e.g. by smoothing the total variation term, see e.g.~\cite{HafMan22} or the Master's thesis~\cite{Haf17}.

In the present work, rather than relying on approximation theory, the derivation of the limiting energy~$\TV_{\varphi_{\T_1}}$ is inspired by results on discrete-to-continuum limits for variational problems on lattice systems, in particular the integral representation of \cite[Thm.~2.4]{BraPia13} as well as the cell formula of \cite[Prop.~2.6]{ChaKre23}. Moreover, the explicit characterization of~$\varphi_{\T_1}$ for the particular case of the double diagonal triangulation relies on a discrete analogue of slicing results used to fix boundary conditions in relaxation, lower semincontinuity and variational nonlinear homogenization results. For completeness, let us mention that the anisotropy induced by regular meshes in the continuous piecewise linear approximation of the Mumford-Shah functional has been characterized in \cite{Neg99}.

\subsection{Structure of the paper} \label{sec:organization}
We collect notation and basic definitions in Subsection \ref{sec:notation} below. In Section \ref{sec:existence} we review the details of the existence of minimizers for \eqref{def:BVprob}, optimality conditions, and introduce the abstract generalized conditional gradient strategy to be used. In Section \ref{sec:finitedimensional}, in view of Galerkin formulations of \eqref{def:BVprob}, we focus on the convex geometry of the total variation balls of piecewise constant and piecewise affine functions on polygonal meshes. Having identified the extremal points in the piecewise constant case, in Section \ref{sec:realization} we introduce our realization of the generalized conditional gradient for this problem, in which the insertion step is treated through a Dinkelbach method with inner iterations solved with graph cuts. Section \ref{sec:disccont} is dedicated to discrete to continuum convergence of the minimization problem for piecewise constant functions on periodic triangulations, and to the explicit computation of the resulting anisotropy for a simple 2D case. In Section \ref{sec:computations}, we numerically apply the proposed algorithm to a few model situations in 2D and 3D, and present some comparisons with other approaches to \eqref{def:BVprob}. Appendix \ref{app:convproofs} contains additional proofs skipped in Section \ref{sec:existence}, in particular basic convergence results applicable to the method. Finally, in Appendix \ref{app:discretegeodesics} we provide detailed proofs relating to the minimal path interpretation of the anisotropy used in Section \ref{sec:disccont}.

\subsection{Notation} \label{sec:notation}
We first recall basic facts (the proof of all of which can be found in \cite[Ch.~3]{AmbFusPal00}) and fix notation for functions of bounded variation which feature repeatedly throughout later sections. For $u \in L^1(\Omega)$ one defines its isotropic (c.f. the general case in \eqref{eq:deftvrho}) total variation as
\begin{align*}
     \TV(u,\Omega)=\sup\left\{\int_\Omega u\,\div\psi \dd x \,\middle\vert\, \psi\in \cC^1_c(\Omega; \R^d), \|\psi\|_{\cC}\leq 1\right\}
\end{align*}
where
\begin{align*}
    \|\psi\|_\cC := \sup_{x\in \Omega } |\psi(x)| \quad \text{for all}~\psi \in \cC_c(\Omega; \R^d),
\end{align*}
with $|\cdot|$ denoting the Euclidean norm. We call $u \in L^1(\Omega)$ a function of bounded variation if $\TV(u,\Omega) < \infty$. Note that $u$ is of bounded variation if and only if there is a regular Borel measure $D u =(\partial_{x_1}u,\dots, \partial_{x_d}u)\in M(\Omega; \R^d)$ with
\begin{align*}
    \int_\Omega \div \psi (x) u(x)~\mathrm{d}x=-\sum^d_{j=1}\int_\Omega \psi_j (x) ~\mathrm{d} \partial_{x_j} u(x) \quad \text{for all }\psi\in \cC^1_c(\Omega; \R^d),
\end{align*}
and whenever $u \in W^{1,1}(\Omega)$, then we have that the weak gradient $\nabla u$ equals $Du$. The linear space of functions of bounded variation
\begin{equation*}
    \BVO := \left\{\,u \in L^1(\Omega)\,\middle\vert\,\TV(u,\Omega)<\infty\,\right\}
\end{equation*}
is a Banach space with respect to the norm
\begin{align*}
    \|u\|_{\BV}= \|u\|_{L^1(\Omega)}+\TV(u,\Omega).
\end{align*}
The space $ \BVO$ continuously embeds into $L^p(\Omega)$ for every $p \in [1, d/(d-1)]$, the embedding being compact for $p<d/(d-1)$. Every bounded sequence $\{u_k\}_k \subset \BV(\Omega)$ admits a subsequence, denoted by the same subscript, which satisfies
\begin{align} \label{def:weak*}
    \lim_{k \rightarrow \infty  } \|u_k-\bar{u}\|_{L^p(\Omega)} = 0 \quad \text{as well as} \quad  D u_k \wksto D \bar{u}
\end{align}
for every $p \in [1,d/(d-1))$ and some $\Bar{u} \in\BVO$. By~``$\wksto$'' we refer to weak* convergence on the space $M(\Omega;\R^d)$. If we additionally have
\begin{align*}
    \lim_{k \rightarrow \infty } \left \lbrack \TV(u_k, \Omega) - \TV(u, \Omega) \right \rbrack = 0,
\end{align*}
the subsequence is called strictly convergent.

Throughout this paper, we interpret $\BVO$ as a subspace of $\LO$, $q=d/(d-1)$, and canonically extend $\TV(u,\Omega)$ to $+\infty$ outside of $\BVO$. Note that $\TV(u,\Omega)$ is convex and weakly lower semicontinuous on $L^q(\Omega)$. Weak convergence in $L^p(\Omega)$ is denoted by~``$\rightharpoonup$''.

Furthermore, denoting the Lebesgue measure on $\R^d$ as $\L^d$, we set
\begin{align*}
    a_u= \frac{1}{\L^d(\Omega)} \int_\Omega u(x)~\mathrm{d}x \quad \text{as well as} \quad \ur=u-a_u
\end{align*}
for every $u \in L^1(\Omega)$ as well as
\begin{align*}
    \Vr:= \left\{\,u \in V \,\middle\vert\,a_u=0\,\right\}=\left\{\,\ur \,\middle\vert\,u \in V\,\right\}
\end{align*}
and its particular instances
\begin{align*}
    \LrO&:= \left\{\,u \in \LO\,\middle\vert\,a_u=0\,\right\}=\left\{\,\ur \,\middle\vert\,u \in \LO\,\right\},\\
    \Prc&:= \left\{\,u \in P_0(\T) \,\middle\vert\,a_u=0\,\right\}=\left\{\,\ur \,\middle\vert\,u \in P_0(\T)\,\right\},\\
    \Prl&:= \left\{\,u \in P_1(\T) \,\middle\vert\,a_u=0\,\right\}=\left\{\,\ur \,\middle\vert\,u \in P_1(\T)\,\right\}.
\end{align*}

Due to \cite[Thm.~3.23]{AmbFusPal00} and the weak closedness of $V$, the unit ball
\begin{align*}
B_{\Vr}=  \left\{\,u \in V\;\middle\vert\;\TV(u,\Omega)\leq 1,~\int_\Omega u(x)~\mathrm{d}x=0\,\right\} \end{align*}
is compact in the weak $\LO$ topology. 

Finally, a subset $E \subset \Omega$ is called a set of finite perimeter if the associated characteristic function $\1_{E}$ is a function of bounded variation. In this case, we call $\Per(E,\Omega):=\TV(\1_{E},\Omega) $ its perimeter. For such sets one can define the reduced boundary $\partial^\ast E$ as the set of points in $x \in \supp D\1_E$ for which
\[\nu_E(x):=\lim_{r \to 0^+} \frac{D\1_E(B(x,r))}{|D\1_E(B(x,r))|}\]
exists and satisfies $|\nu_E(x)|=1$. Using this notion, the De Giorgi structure theorem \cite[Thm.~3.59]{AmbFusPal00} provides the expression
\[D\1_E = \nu_E \,\mathcal{H}^{d-1} \mres \partial^\ast E,\]
where the symbol $\mres$ denotes restriction of a measure, that is $(\mu \mres A)(B) = \mu(A \cap B)$, and $\mathcal{H}^{d-1}$ is the $(d-1)$-dimensional Hausdorff measure \cite[Sec.~2.8]{AmbFusPal00}.
A set of finite perimeter $E\subseteq \Omega$ is \textit{decomposable} if there exist two nonempty subsets $E_1,E_2\subseteq \Omega$ such that $E = E_1 \cup E_2$ and $\Per(E,\Omega)=\Per(E_1,\Omega)+\Per(E_2,\Omega)$, and \textit{indecomposable} otherwise. Moreover, we say that $E\subseteq \Omega$ is \textit{simple} if both $E$ and $\Omega \setminus E$ are indecomposable. These notions play the role of connectedness and simple connectedness in the setting of sets of finite perimeter, for which the topological notions are not meaningful. In particular, one can \cite[Thm.~1]{AmbCasMasMor01} express such sets as a union of their maximal indecomposable subsets, also referred to as \textit{indecomposable components}.

\section{Existence, optimality, and conditional gradient approach}\label{sec:existence}
The following properties of the misfit functional $F$ and of the forward operator $K$ are assumed throughout the paper.
\begin{assumptionp}{A}\label{Ass:basics}
Let $Y$ be a separable Hilbert space with inner product $(\cdot,\cdot)_Y$ and denote by $\|\cdot\|_Y$ the induced norm. We assume that:  
\begin{itemize}
\item[$\textbf{A}1$] The operator $K : V \to Y $ is continuous from the weak topology of $L^q(\Omega)$ to the strong topology of $Y$.
    \item[$\textbf{A}2$] $F$ is thrice continuously Fr\'echet differentiable as well as strictly convex. Moreover, we have~$F(y) \geq 0$ for all~$y\in Y $
    \item[$\textbf{A}3$] For every $y \in Y$ there is a unique $\bar{c}(y)\in \R$ with
    \begin{align*} 
        F(y+\bar c(y)K_{\1})=\min_{c \in \R} F(y+c(y)K_{\1})
    \end{align*}
    where we abbreviate $K_{\1} := K \1_\Omega$.
\end{itemize}
\end{assumptionp}
These requirements are, e.g., fulfilled for the PDE-constrained optimal control problem~\eqref{eq:PDEconstraint}, assuming $d\leq 4$ and setting
\begin{align*}
    F(\cdot)=\frac{1}{2\alpha} \|\cdot-y_o\|^2_{L^2(\Omega)} \quad \text{and} \quad Y=L^2(\Omega).
\end{align*}
Moreover, under suitable assumptions on the coefficient functions $a$ and $b$ as well as the domain $\Omega$,~\eqref{eq:PDE} admits a unique solution $y=Ku$ for $u \in L^2(\Omega)$ and the control-to-state operator $K$ is weak-to-strong continuous (see for example \cite{Tro10}, and the related considerations in Proposition \ref{prop:gamamconvJ} below).

In this section, we prove the existence of at least one minimizer to~\eqref{def:BVprob} and state a particular set of necessary and sufficient optimality conditions which are tailored to the algorithmic approach considered in the remainder of the paper. For this purpose, similar to the one dimensional setting in~\cite{TraWal23}, we rely on an equivalent reformulation of Problem~\eqref{def:BVprob}. More in detail, define
\begin{align}\label{eq:defofcalF}
    \mathcal{F} : Y \to \R \quad \text{where} \quad \mathcal{F}(y)= \min_{c \in \R} F(y+cK_{\1}).
\end{align}
and consider the auxiliary problem
\begin{align} \label{def:BVpropmean}
\inf_{v \in \Vr} \mathcal{J}(v) := \lbrack \mathcal{F}(K v)+ \TV(v,\Omega) \rbrack \tag{$\mathfrak{P}$}. 
\end{align}
The following proposition collects some of its properties. Its proof is relegated to Appendix \ref{app:convproofs}.
\begin{proposition}\label{prop:propofparF}
There holds:
    \begin{itemize}
    \item The restricted operator $K : \Vr \to Y $ is weak-to-strong continuous.
    \item The mapping $\bar{c} : Y \to \R$ from Assumption~\ref{Ass:basics}, $\textbf{A}3$, is two times continuously differentiable.
    \item The functional $\mathcal{F}$ is convex and twice continuously Fr\'echet differentiable. Its gradient $\nabla \mathcal{F}$ is Lipschitz continuous on compact sets and there holds 
    \begin{align}
        \nabla \mathcal{F}(y)=\nabla F(y+\bar{c}(y)K_{\1}) \quad \text{for all} \quad y\in Y.
    \end{align}
\end{itemize}
\end{proposition}

\begin{remark} \label{rem:firstorderforcoeff}
    Note that the minimization problem in~\eqref{eq:defofcalF} is convex. Calculating its first order necessary optimality conditions reveals 
    \begin{align*}
        \bar c(y)= \argmin_{c \in \R} F(y+c K_{\1})  \Leftrightarrow \int_\Omega K^* \nabla F(y+\bar c(y) K_{\1})=0. 
    \end{align*}
\end{remark}
\subsection{Existence of minimizers} \label{subsec:existence}
In order to show the existence of a solution of~\eqref{def:BVprob} we rely on the following equivalence result as well as Proposition~\ref{prop:propofparF}.
\begin{proposition} \label{prop:equivalence}
The following equivalence holds:
\begin{itemize}
    \item The function $\Bar{u}$ is a minimizer of~\eqref{def:BVprob}.
    \item The function $\mathring{\bar u}$ is a minimizer of~\eqref{def:BVpropmean} and
    \begin{align} \label{eq:optofaverage}
        a_{\bar u}=\argmin_{c \in \R} F(K \mathring{\Bar{u}}+c K_{\1}),
    \end{align}
    that is, $a_{\bar u} = \Bar{c}(K\mathring{\bar u})$.
\end{itemize}
\end{proposition}
\begin{proof}
    Let $\bar{u} \in V $ and $\Bar{v} \in \Vr$ be arbitrary but fixed. On the one hand, note that we have
    \begin{align*}
 J(\bar{u})= F(K\mathring{\bar{u}}+a_{\bar{u}} K_{\1})+ \TV(\mathring{\bar u}) \geq \mathcal{F}(K \mathring{\bar u})+ \TV(\mathring{\bar{u}}) =\mathcal{J}(\mathring{\bar{u}})    
    \end{align*}
    as well as
    \begin{align*}
        \mathcal{J}(\mathring{\bar{u}}) \geq \inf_{u \in V} \mathcal{J}(\ur) =\inf_{v \in \Vr} \mathcal{J}(v).
    \end{align*}
    On the other hand, there holds
    \begin{align*}
        \inf_{u \in V} J(u) \leq \inf_{c \in \R} J( \bar v+c \1_\Omega)=\mathcal{J}(\bar v) \quad \text{and thus} \quad \inf_{u \in V} J(u)\leq \inf_{v \in \Vr} \mathcal{J}(v).  
    \end{align*}
    Now, if $\bar{u}$ is a minimizer of~\eqref{def:BVprob}, these observations imply
    \begin{align*}
        \min_{u \in V} J(u) = J(\bar u)=\mathcal{J}(\mathring{\bar u})=\min_{v \in \Vr} \mathcal{J}(v),
    \end{align*}
    i.e., $\mathring{\bar u}$ is a minimizer of $\eqref{def:BVpropmean}$. Finally, from $J(\bar u)=\mathcal{J}(\mathring{\bar u})$ we conclude
    \begin{align*}
       F(K\Bar{u})= F(K\mathring{\Bar{u}}+ a_{\Bar{u}}K_{\1})= \min_{c \in \R} F(K\mathring{\Bar{u}}+c K_{\1}).
    \end{align*}
Thus, we conclude~\eqref{eq:optofaverage} from Assumption $\mathbf{A}1$.
    The converse direction follows by the same argument, finishing the proof.
\end{proof}

\begin{theorem}\label{thm:existence}
Let Assumption~\ref{Ass:basics} hold. Then Problem~\eqref{def:BVprob} admits at least one minimizer $\Bar{u} \in \BVO$.  
\end{theorem}
\begin{proof}
Using Proposition~\ref{prop:propofparF}, Problem~\eqref{def:BVpropmean} can be embedded in the abstract setting of~\cite{BreCarFanWal23}. Thus, existence of at least one minimizer $\mathring{\bar u}$ is concluded from~\cite[Prop.~2.3]{BreCarFanWal23}. Finally, applying Proposition~\ref{prop:equivalence} yields the desired statement.
\end{proof}
\subsection{First order optimality conditions} \label{subsec:}
Next, we state a particular set of necessary and sufficient first order optimality conditions for Problem~\eqref{def:BVprob} which will serve as a foundation for its algorithmic solution. In this context, an important role is played by the extremal points of~$B_{\Vr}$.
\begin{definition} \label{def:extremals}
An element $u \in B_{\Vr}$ is called an extremal point of $B_{\Vr}$ (or extremal) if there is no $s \in (0,1) $ and $u_1,u_2 \in B_{\Vr}$, $u_1 \neq u_2$ with $u=(1-s)u_1+su_2$. The set of all extremal points of $B_{\Vr}$ is denoted by $\Ext(B_{\Vr})$.
\end{definition}
Since $B_{\Vr} \neq \emptyset$ is weakly compact, the set $\Ext(B_{\Vr})$ is nonempty due to Krein-Milman Theorem~\cite[Thm.~3.23]{Rud91}. Utilizing this concept as well as the one-homogeneity of~$\TV(\cdot, \Omega)$, we arrive at the following sufficient first order optimality condition.
\begin{theorem} \label{thm:firstorder}
Let $\Bar{u} \in V$ be given and define
\begin{align*}
\Bar{p}=-K^* \nabla F(K\Bar{u})\in V'.     
\end{align*}
The following are equivalent:
\begin{itemize}
    \item[i)] The function $\Bar{u}$ is a solution to~\eqref{def:BVprob}.
    \item[ii)] The functions  $\Bar{u} \in V$ and $\Bar{p} \in V'$ satisfy
    \begin{align} \label{eq:overallvar}
     \int_\Omega \Bar{p}~\mathrm{d}x=0, \quad    \int_\Omega \Bar{p} \Bar{u}~\mathrm{d}x= \TV(\Bar{u},\Omega), \quad \int_\Omega \Bar{p} u~\mathrm{d}x \leq 1.
    \end{align}
    for all $u \in B_{\Vr}$.
   \item[iii)] The functions  $\mathring{\Bar{u}} \in \mathring V$ and $\Bar{p} \in V'$ satisfy
    \begin{align} \label{eq:overallper}
     \int_\Omega \Bar{p}~\mathrm{d}x=0, \quad    \int_\Omega \Bar{p} \mathring{\Bar{u}}~\mathrm{d}x= \TV(\mathring{\Bar{u}},\Omega), \quad \int_\Omega \Bar{p}u~\mathrm{d}x \leq 1.
    \end{align}
    for all $u \in \Ext(B_{\Vr})$.
\end{itemize}
\end{theorem}
\begin{proof}
We proceed similarly to Theorem~\ref{thm:existence} and exploit the equivalence of Problem~\eqref{def:BVprob} and Problem~\eqref{def:BVpropmean} in order to apply the abstract theory of~\cite{BreCarFanWal23}.
According to Proposition~\ref{prop:equivalence}, $\Bar{u}$ is a solution to~\eqref{def:BVprob} if and only if
\begin{align*}
    \mathring{\Bar{u}} \in \argmin \eqref{def:BVpropmean}, \quad a_{\bar u}=\argmin_{c \in \R} F(K \mathring{\Bar{u}}+c K_{\1}) .
\end{align*}
Thus, it suffices to derive first-order necessary and sufficient optimality conditions for~$\mathring{\bar{u}}$. Set~$\bar{P}=- K^* \nabla \mathcal{F}(K \mathring{\bar{u}})$. 
Recalling Remark~\ref{rem:firstorderforcoeff} and applying~\cite[Prop.~2.3]{BreCarFanWal23} yields that~$\mathring{\bar{u}}$ is a minimizer of Problem~\eqref{def:BVpropmean} if and only if 
\begin{align*}
\int_{\Omega} \bar{P} \mathring{\Bar{u}}~\mathrm{d}x = \TV(\mathring{\bar{u}},\Omega), \quad \max_{u \in \Ext(B_{\Vr})} \int_{\Omega} \bar{P} \mathring{\bar{u}}~\mathrm{d}x \leq 1, \quad \int_{\Omega} \bar{P}~\mathrm{d}x =0.
\end{align*}
Finally, from~$a_{\bar{u}}=\bar{c}(K\mathring{\bar{u}})$, see Proposition~\ref{prop:equivalence}, as well as the  characterization of~$\nabla \mathcal{F}$, see Proposition~\ref{prop:propofparF}, we conclude
\begin{align*}
    \nabla \mathcal{F}(K\mathring{\bar{u}})=\nabla {F}(K\mathring{\bar{u}}+ a_{\bar{u}}K_{\1})=\nabla {F}(K\bar{u})
\end{align*}
and thus~$\bar{p}=\bar{P}$. Together with~$\TV(\mathring{\bar{u}},\Omega)=\TV(\bar{u},\Omega)$, this yields
$i) \Leftrightarrow ii)$.
   Finally, the equivalence between $ii)$ and $iii)$ follows from
    \begin{align*}
       \max_{v\in B_{\Vr} } \int_{\Omega} \Bar{p} v~\mathrm{d}x= \max_{v \in\operatorname{Ext}(B_{\Vr})} \int_{\Omega} \Bar{p}v~\mathrm{d}x \quad \text{as well as} \quad \int_{\Omega} \Bar{p} \mathring{u}~\mathrm{d}x=\int_{\Omega} \Bar{p} u~\mathrm{d}x 
    \end{align*}
    for all~$u \in \BVO$, see~\cite[Lem.~A.1]{BreCarFanWal23}.
\end{proof}
\subsection{Abstract fully-corrective generalized conditional gradient approach} \label{subsec:contalg}

For the algorithmic computation of a solution to~\eqref{def:BVprob}, we can again use the equivalence result in Proposition~\ref{prop:equivalence} and apply the fully-corrective generalized conditional gradient method (FC-GCG) from~\cite{BreCarFanWal23} to the surrogate Problem~\eqref{def:BVpropmean}. In this section, we briefly summarize the abstract algorithm for the problem at hand. For this purpose, and given a finite, ordered set $\mathcal{A}=\{v_j\}^N_{j=1} \subset \Ext(B_{\Vr})$ of extremal points, define the parametrization  
\begin{align*}
    \mathcal{U}_{\mathcal{A}} : \R^{N}  \to  \Vr, \quad \mathcal{U}_{\mathcal{A}}(\gamma)= \sum^{N}_{j=1} \gamma_j v_j
\end{align*}
as well as the finite-dimensional minimization problems
\begin{align} \label{eq:separatsubs}
    \bar{\gamma} \in \argmin_{\gamma \in \R^N, \gamma \geq 0} \left\lbrack \mathcal{F}(K \mathcal{U}_{\mathcal{A}}(\gamma))+ \sum^N_{j=1} \gamma_j \right \rbrack, 
\end{align}
where the inequality constraint is understood componentwise. By definition of $\mathcal{F}$, the pair $\bar{\gamma}$ satisfies~\eqref{eq:separatsubs} if and only if
\begin{align} \label{def:subprobPDAP}
    (\bar{\gamma},\bar c) \in  \argmin_{\gamma \in \R^N,c \in \R, \gamma \geq 0} \left\lbrack F(K \mathcal{U}_{\mathcal{A}}(\gamma)+cK_{\1})+ \sum^N_{j=1} \gamma_j \right \rbrack. \tag{$\mathcal{P}(\mathcal{A})$}
\end{align}
where~$\bar{c}=\Bar{c}(K\mathcal{U}_{\mathcal{A}}(\bar \gamma))$.
The FC-GCG method according to~\cite{BreCarFanWal23} relies on the alternating update of a finite active set $\mathcal{A}_k \subset \Ext(B_{\Vr})$ as well as of iterates $ \mathring u_k \in \mathring V$ and~$u_k\in V$, respectively, given by
\begin{align*}
 \mathring u_k= \mathcal{U}_{\mathcal{A}_k}(\gamma^k), \quad   u_k=\mathcal{U}_{\mathcal{A}_k}(\gamma^k)+c^k 
\end{align*}
where~$(\gamma^k,c^{k})$ is obtained from \noeqref{def:subprobPDAP}
\begin{align*}
    (\gamma^k,c^{k}) \in \argmin (\mathcal{P}(\mathcal{A}_k)).
\end{align*}
More in detail, the algorithm proceeds as follows: Given the current active set $\mathcal{A}_k$ and iterates $ \mathring{u_k}$ and~$u_k$ as defined above, we first calculate~$p_k=- K^* \nabla \mathcal{F}(K\mathring u_k) \in V'$. By Proposition~\ref{prop:propofparF} and since from the definitions of~$c^k$ and~$u_k$ we have $c^k=\bar{c}(K\mathring{u}_k)$, we get
\begin{align*}
    p_k=- K^* \nabla \mathcal{F}(K\mathring u_k)=- K^* \nabla F(K\mathring u_k+c^k K_{\1})=-K^*\nabla F(K u_k).
\end{align*}
Moreover, from Remark \ref{rem:firstorderforcoeff}, we also conclude that
\begin{align*}
-\int_\Omega K^* \nabla F(K\mathring u_k+c^k K_{\1} )~\mathrm{d}x =   \int_\Omega p_k~\mathrm{d}x=0.
\end{align*}
Now, we compute a new candidate extremal point~$\widehat{v}_k \in \Ext(B_{\Vr})$ satisfying
\begin{align*}
    \int_{\Omega} p_k \widehat{v}_k~\mathrm{d}x=\max_{v \in\operatorname{Ext}(B_{\Vr})} \int_{\Omega} p_k v~\mathrm{d}x.
\end{align*}
At this point we can check for convergence, since if we have
\begin{align*}
    \int_{\Omega} p_k \widehat{v}_k~\mathrm{d}x \leq 1,
    %\max_{v \in\operatorname{Ext}(B_{\Vr})} \int_{\Omega} p_k v~\mathrm{d}x,
\end{align*}
then there also holds
\begin{align*}
    \int_\Omega p_k \mathring{u}_k~\mathrm{d}x= \TV(\mathring{u},\Omega)
\end{align*}
see~\cite[Proposition 3.1]{BreCarFanWal23}, and the method stops, because in this case, by Theorem~\ref{thm:firstorder} we have that $\mathring{u}_k$ and $u_k$ are minimizers of respectively~\eqref{def:BVpropmean} and~\eqref{def:BVprob}. Otherwise, we update the active set and iterates according to
\begin{align*}
 \mathcal{A}^+_k= \mathcal{A}_k \cup \{\widehat{v}_k\}, \quad    \mathring u_{k+1}=\mathcal{U}_{\mathcal{A}^+_k}(\gamma^{k,+})+c^{k+1}, \quad u_{k+1}=\mathring u_k+c^{k+1},
\end{align*}
where
\begin{align*}
    (\gamma^{k,+},c^{k+1}) \in \argmin (\mathcal{P}(\mathcal{A}^+_k)).
\end{align*}
Finally, the new active set $\mathcal{A}_{k+1}$ is obtained from $\mathcal{A}^+_k$ is pruned by removing extremal points which are assigned a zero weight in $(\mathcal{P}(\mathcal{A}_k))$. Recalling Theorem~\ref{thm:firstorder}, this procedure can be interpreted as a greedy method for the first order condition presented therein: In every iteration, it checks whether the current dual variable~$p_k$ satisfies
\begin{align*}
    \int_\Omega p_k u~\mathrm{d}x \leq 1 \quad \text{for all} \quad u \in \Ext(B_{\Vr}). 
\end{align*}
If this is not the case, an element~$\widehat{v}_k \in \Ext(B_{\Vr})$ which maximizes the violation of this constraint is added to~$\mathcal{A}_k$ and the subproblem is solved anew.
The abstract method is summarized in Algorithm~\ref{alg:abstractgcg}. Since~\eqref{def:BVpropmean} is merely an auxiliary problem and our main interest lies in computing a solution to Problem~\eqref{def:BVprob}, both the pseudocode as well as the following theoretical results are only describing the update as well as the converge of the shifted sequence~$u_k$.  

\begin{algorithm}[ht]\label{alg:abstractgcg}
\setstretch{1.15}
\caption{Abstract FC-GCG for Problem~\eqref{def:BVprob}}
    \KwInput{$u_0=0$, $\mathcal A_0=\emptyset$, $N_0=0$}
     calculate $c^0 \in \argmin_{c \in \R} F(cK_{\1})$\\
     $u_{0}\leftarrow c^0 \1_{\Omega}$\\
    \For{$k=0,1,2,...$}{
        $p_k\leftarrow -K^*\nabla F(K u_k)$\\
        calculate: \[\displaystyle \widehat{v}_k \in \argmax_{v \in \Ext(B_{\Vr})} \int_\Omega p_k v \dd x\]\\
        \If{$\int_\Omega p_k \widehat{v}_k \dd x \leq 1$}
        {terminate with a solution $\bar{u}=u_k$ to~\eqref{def:BVprob}}
        $N_k^+\leftarrow N_k+1$, $ u^k_{N_k^+}\leftarrow \widehat{v}_k$\\
        $\mathcal A_k^+\leftarrow \mathcal A_k\cup \{u^k_{N_k^+}\}$\\
        calculate:
       \begin{equation}
            (\gamma^{k+1}, c^{k+1})\in \argmin_{(\gamma, c)\in \R_+^{N_k^+} \times \R} \!\! F\left(K\mathcal{U}_{\mathcal{A}^+_k}(\gamma) + cK_{\1}\right)+\sum_{i=1}^{N_k^+}\gamma_i
        \end{equation}
        $\!u_{k+1}\leftarrow \mathcal{U}_{\mathcal{A}^+_k}(\gamma^{k+1}) + c^{k+1}\1_\Omega = \sum_{i=1}^{N_k^+}\gamma^{k+1}_i u_i^k + c^{k+1}\1_\Omega$\\
        $\mathcal{A}_{k+1}\leftarrow \mathcal{A}_{k}^+\setminus \{u_i^k:\gamma_i^{k+1}=0\}$\\        
        $N_{k+1}\leftarrow \#\mathcal{A}_{k+1}$
    }
\setstretch{1}
\end{algorithm}
We finish this section by summarizing some of its properties. Their proof is again relegated to Appendix \ref{app:convproofs}.
\begin{proposition} \label{prop:sublinPDAP}
Let $u_k$, $k\in \N$, be generated by Algorithm \ref{alg:abstractgcg}. Then the method either stops after $\bar{k}$ steps with a minimizer $u_{\bar{k}}$ to Problem~\eqref{def:BVprob} or there holds
\begin{align*}
    J(u_k)- \min_{u \in \LO} J(u) \leq  \frac{c}{1+k}.
\end{align*}
for some $c>0$ and all $k\in\N$.
Moreover, in this case, the sequence $u_k$ admits at least one strictly convergent subsequence and every strict accumulation point is a minimizer of~\eqref{def:BVprob}.
\end{proposition}

\begin{lemma} \label{lem:upperbound}
Let $u_k=\mathcal{U}_{A_k}(\gamma^k)+ c^k \1_\Omega$, $k\in\N$, be generated by Algorithm~\ref{alg:abstractgcg} and set
\begin{align*}
    M_k := \left(F\left(K\mathcal{U}_{\mathcal{A}_k}(\gamma^k) + c^k K_{\1}\right)+\sum_{j=1}^{N_k} \gamma^k_j \right).
\end{align*}
Then there holds
\begin{align}\label{eq:indicator}
    J(u_k)- \min_{u \in V} J(u) \leq \zeta_k \coloneqq M_k \left( \int_\Omega p_k \widehat{v}_k~\mathrm{d}x -1\right).
\end{align}
\end{lemma}
\begin{remark} \label{rem:1k}
By inspecting the respective proofs in~\cite{BreCarFanWal23}, the results of Proposition~\ref{prop:sublinPDAP} and Lemma~\ref{lem:upperbound} remain valid if~$\widehat{v}_k$ is not an extremal point but satisfies~$\widehat{v}_k \in B_{\Vr} $ as well as
\begin{align}
   \int_{\Omega} p_k \widehat{v}_k~\mathrm{d}x =\max_{v \in\operatorname{Ext}(B_{\Vr})} \int_{\Omega} p_k v~\mathrm{d}x.
\end{align}
The incentive behind minimizing over~$\operatorname{Ext}(B_{\Vr})$ instead of the whole ball~$B_{\Vr}$ is that, once the extremal points are characterized, additional structural properties of this ``simpler'' set can be utilized to facilitate the algorithmic solution of the linear subproblems. 
\end{remark}
\begin{remark} \label{rem:nonmonotone}
It is worth pointing out that the convexity and one-homogeneity of~$\TV(\cdot, \Omega )$ also imply
\begin{align}\label{eq:degeneracymaybe}
    J(\mathcal{U}_{\mathcal{A}_k}(\gamma)) \leq F(\mathcal{U}_{\mathcal{A}_k}(\gamma)) + \sum^{N_k}_{j=1} \gamma_j \ \quad \text{for all} \quad \gamma \in \R^{N_k},
\end{align}
but the inequality can be strict for some~$\gamma$.
In particular, Algorithm~\ref{alg:abstractgcg} is not a descent method, i.e., it does not ensure~$J(u_{k+1})\leq J(u_k)$ in general. In the recent paper~\cite{CasDuvPet22}, the authors consider a similar method based on the greedy insertion of extremal points for functions of bounded variation on unbounded domains, i.e.,~$V= \BV(\R^2)$. Different to the current work, they rely on updating the iterate following
\begin{align*}
    u_{k+1} \in \argmin_{u \in \Span(\mathcal{A}_k)} J(u)= F(Ku)+ \TV(u,\Omega).
\end{align*}
This naturally ensures motonicity as well as a sublinear rate of convergence. However, ensuring the full resolution of the $\TV$-regularized subproblem is significantly more challenging than just using the~$\ell_1$ problem considered in the current work. Moreover, in our numerical examples presented in Section \ref{sec:computations} the computed iterates always satisfied \eqref{eq:degeneracymaybe} with equality. % For this reason, in Appendix \ref{app:FCGCunbound}, we briefly describe the application of FC-GCG in the setting of~\cite{CasDuvPet22} which again yields a sublinearly converging algorithm with~$\ell_1$-like subproblems albeit at the cost of monotonicity.
\end{remark}
\section{Extremal points of total variation balls}\label{sec:finitedimensional}
From the discussion in the previous section, it becomes apparent that a practical realization of the abstract Algorithm~\ref{alg:abstractgcg}, first and foremost, requires a characterization of the set of extremal points~$\Ext(B_{\Vr})$. While this can be directly concluded from known results for~$\Vr=\mathring{L}^q(\Omega)$, see Section~\ref{subsec:conextremals}, to the best of our knowledge, there are no explicitly stated results concerning the finite element spaces $\Vr=\Prc$ and $\Vr=\Prl$, respectively. Since many interesting instances of Problem~\eqref{def:BVprob} inherently require a discretization of the variable space, we take a step at closing this gap in the following section. On the one hand, we give a precise characterization of~$\Ext(B_{\mathring{P}_0})$ and show that piecewise constant discretizations are conforming with the convex geometry of~$B_{\mathring{L}^q}$ in the sense that $\Ext(B_{\mathring{P}_0})$ is a finite set and~$\Ext(B_{\mathring{P}_0}) \subset \Ext\left(B_{\mathring{L}^q}\right)$. On the other hand, we provide a ``negative'' result on~$\Ext(B_{\mathring{P}_1})$ showing that it contains infinitely many functions, albeit we do not provide of a full characterization of this set.       
\subsection{Continuous case with Neumann boundary} \label{subsec:conextremals}
For~$\Vr=\mathring{L}^q(\Omega)$, the characterization of $\Ext(B_{\Vr})$ follows from the results in~\cite{BreCar20} by exploiting the invariance of extremal points under embeddings. 
\begin{proposition} \label{prop:charofconextremals}
There holds
\begin{align*}
    \Ext\left(B_{\mathring{L}^q}\right)= \left\{\,\frac{1}{\Per(E,\Omega)}\mathring{\1}_E\,\middle\vert\,E ~ \operatorname{  simple}\,\right\}.
\end{align*}
\end{proposition}
\begin{proof}
For $u \in \BVO$, define the equivalence class
\begin{align*}
    \lbrack u \rbrack= \left \{\,u+ c \1_{\Omega}\,\middle\vert\,c \in \R\,\right\} \in \BVO / \mathcal{C} \quad \text{where} \quad \mathcal{C}=\left\{\,c\1_\Omega\,\middle\vert\,c \in \R\,\right\}.
\end{align*}
We equip the quotient space $\BVO / \mathcal{C}$ with the canonical quotient norm
\begin{align*}
    \|\lbrack u\rbrack\|_{\BVO / \mathcal{C}}= \inf_{v\in \mathcal{C}} \|u+v\|_{\BV}= \inf_{v\in \mathcal{C}} \left \lbrack \|u+v\|_{L^1(\Omega)}+ \TV(u,\Omega)\right \rbrack.
\end{align*}
Mutatis mutandis, we define the extremal points $\operatorname{Ext}(\mathcal{B})$ of the set
\begin{align*}
    \mathcal{B}=\left\{\,\lbrack u \rbrack\,\middle\vert\,\TV(u,\Omega)\leq 1\,\right\} \subset \BVO / \mathcal{C}.
\end{align*}
According to~\cite[Thm.~4.7]{BreCar20}, there holds 
\begin{align*}
    \operatorname{Ext}(\mathcal{B})= \left\{\, \left\lbrack \frac{1}{ \Per(E, \Omega)} \1_E \right \rbrack \, \middle\vert \, E ~ \text{simple}\,\right\}.
\end{align*}
Consider the mapping $L : \BVO / \mathcal{C} \to \LrO$ with $L(\lbrack u\rbrack)= \mathring{u}$. By definition, $L$ is linear and there holds $L(\mathcal{B})=L(B_{\mathring{L}^q})$. Moreover, $L$ is continuous due to
to
\begin{align*}
    \|\mathring{u}\|_{L^q(\Omega)}=\|L(\lbrack u \rbrack)\|_{L^q(\Omega)} \leq C \TV(u, \Omega) \leq C \|\lbrack u\rbrack\|_{\BVO / \mathcal{C}},
\end{align*}
see, e.g.,~\cite[Thm.~3.44]{AmbFusPal00}, as well as injective. Summarizing these observations,\cite[Lem.~3.2]{BreCar20} yields
\begin{align*}
    \Ext\left(B_{\mathring{L}^q}\right)=\Ext(L(\mathcal{B}))=L(\Ext(\mathcal{B}))
\end{align*}
and thus the desired statement.
\end{proof}

\subsection{Discretization with piecewise constant functions} \label{subsec:discpiece}
In the following section, let $\Omega\subseteq \R^d$ be an open and bounded polyhedron and let~$\T$ be a triangulation thereof, that is, a collection of open simplices 
\[\T = \{T_1,\ldots,T_{n_{\T}}\}\quad\text{such that}\quad\overline{\Omega} = \bigcup_{i=1}^{n_{\T}} \overline{T_i}.\]
\begin{definition}\label{def:triangulated}
We define the class of triangulated subsets of $\Omega$ with respect to the triangulation $\T$, denoted as $\S_\T(\Omega)$, to be the open sets obtained as the interior of a union of closures of simplices belonging to $\T$. For such sets $E_1, E_2 \in \S_\T(\Omega)$, we implicitly redefine unions and set differences so that $E_1 \cup E_2,\,E_1 \setminus E_2,\,E_1 \triangle E_2:= (E_1\setminus E_2)\cup (E_2\setminus E_1) \in \S_\T(\Omega)$.
\end{definition}
If not otherwise specified, any subset $E\subseteq \Omega$ considered in this section belongs to $\S_\T(\Omega)$. Moreover, for a subset~$E \in \S_\T(\Omega)$ we have that $\Per(E,\Omega) = \H^{d-1}(\partial E\cap \Omega)$. Let $P_0(\T)$ be the set of piecewise constant functions on $\Omega$ with respect to the triangulation $\T$, and $\Prc$ its subspace of zero average. Inside $\Prc$, define the unit ball with respect to the total variation:
\begin{equation}\label{eq:regular extreme}
    B_{\mathring{P}_0}:=\big\{u\in \Prc : \TV(u, \Omega)\leq 1\big\}.
\end{equation}
%Recall that the total variation is usually defined as follows
%\begin{equation}
 %   \TV(f,\Omega)=\sup\left\{\int_\Omega f\,\div\psi \dd x: \psi\in \cC^1_c(\Omega), \|\psi\|_\infty\leq 1\right\}.
%\end{equation}
There is a convenient description of the total variation in this discrete setting.
\begin{lemma}
    For any $u\in P_0(\T)$, the total variation corresponds to 
    \[\TV(u,\Omega)=\frac{1}{2}\sum_{i,j}|u(T_i)-u(T_j)|\H^{d-1}(\partial T_i\cap \partial T_j)\]
    \begin{proof}
        This result follows by standard facts about functions of bounded variation, see for example Section 4.2 and in particular Theorem 4.23 in \cite{AmbFusPal00}. It is also explicitly stated in \cite[Thm.~3.1]{CasKunPol99} and \cite[Lem.~4.1]{Bar12}.        
    \end{proof}
\end{lemma}

\begin{proposition}\label{prop:pwcextremals}
There holds
\begin{align*}
    \Ext\left(B_{\mathring{P}_0}\right)= \left\{\,\frac{1}{\Per(E,\Omega)}\mathring{\1}_E\,\middle\vert\,E\in S_\T(\Omega) ~ \operatorname{simple}\,\right\}.
\end{align*}
\end{proposition}

\begin{proof}
The proof of this lemma is structured in 3 cases. First, we prove that piecewise constant functions with at least three values in the image are not extremal. Then, in case 2, we do the same for characteristic functions on indecomposable but not simple domains. Finally, in case 3, we prove that, among the remaining functions, the extremal points of $B_{\mathring{P}_0}$ are characteristic functions on simple domains.

The proof of each case follows a standard procedure using perturbations that can be found, for example, in \cite{AmbAziBreUns24}. 
Namely, for some $u\in \partial B_{\mathring{P}_0}$, set 
\[\delta_u:=\min\{|u(T_i)-u(T_j)|:\H^{d-1}(\partial T_i\cap \partial T_j)>0, u(T_i)\neq u(T_j)\},\]
which corresponds to the minimum non zero variation of $u$ between two adjacent simplices.
Consider a perturbation $h\in \Prc$ satisfying the following property: for every pair of adjacent simplices $T_i,T_j$, there exists $\lambda_{ij}\in \R$ such that 
\begin{equation}\label{eq:perturbation}
    h(T_i)-h(T_j)=\lambda_{ij} (u(T_i)-u(T_j)).
\end{equation}
This condition requires that if $u$ has no variation between adjacent simplices, neither does $h$. Whenever $h(T_i)=h(T_j)$, we set $\lambda_{ij}=0$. This condition limits the set of possible perturbations that we can choose from, but there are usually many that do satisfy this property, e.g. $h\in \Span(u)$. We will always exclude the trivial case $h=0$, so that we can define
\[\Delta_h:=\max \{|h(T_i)-h(T_j)|:\H^{d-1}(\partial T_i\cap \partial T_j)>0\},\]
which is strictly positive since $h\not\in \mathcal{C}$ by assumption. Choose an $\epsilon>0$ smaller than $\delta_u/\Delta_h$, so that, by construction, $\epsilon |\lambda_{ij}|\leq 1$ for every $i,j$. Then, the crucial part of each step will be finding the right perturbation $h$ for which $\TV(u\pm \epsilon h,\Omega)\neq 0$. If that is the case, we then define
\[u^+:=\frac{u+\epsilon h}{\TV(u+\epsilon h,\Omega)} \in B_{\mathring{P}_0} \qquad \text{and} \qquad u^-:=\frac{u-\epsilon h}{\TV(u-\epsilon h,\Omega)} \in B_{\mathring{P}_0},\]
and, by construction, it is possible to prove that $u$ is a convex combination of $u^+$ and $u^-$. Indeed,
\begin{align*}
    &\TV(u+\epsilon h,\Omega)+\TV(u-\epsilon h,\Omega)\\
    &\qquad=\frac{1}{2}\sum_{i,j}\H^{d-1}(\partial T_i\cap \partial T_j)(|(1+\epsilon \lambda_{ij})(u(T_i)-u(T_j))|+|(1-\epsilon \lambda_{ij})(u(T_i)-u(T_j))|)\\
    &\qquad=\frac{1}{2}\sum_{i,j}\H^{d-1}(\partial T_i\cap \partial T_j)((1+\epsilon \lambda_{ij})|u(T_i)-u(T_j)|+(1-\epsilon \lambda_{ij})|u(T_i)-u(T_j)|)\\
    &\qquad=2 \TV(u,\Omega)=2.
\end{align*}
Of course $u^+$ and $u^-$ could coincide, but this is only true if $h\in \Span(u)$. So in conclusion, if there exists a zero average perturbation $h$ that satisfies equation \eqref{eq:perturbation}, such that $h\not\in \Span(u)$ and $\TV(u\pm\epsilon h,\Omega)\neq 0$ for a sufficiently small $\epsilon>0$, then $u$ is not extremal.

%\textbf{Case 1:} Let $u\in \partial B_{\mathring{P}_0} $ be a piecewise constant function with decomposable support. Let $C_1,...,C_n\in \S_\T(\Omega)$ be indecomposable components of the support and assume $n\geq 2$. There exists $u_1,...,u_n\in P_0(\T)$ such that $\supp(u_i)=\overline{C_i}$ and $u=\sum_{i=1}^nu_i$.
%Define
%\[h:=\mathring{u}_1-\mathring{u}_2.\]
%This perturbation is clearly not in the span of $u$. Moreover, it satisfies equation \eqref{eq:perturbation}, with $\lambda_{ij}\in \{-1,0,1\}$. 
%For $0<\epsilon<\delta_u/\Delta_h$, we get
%\begin{align*}
%    \TV(u\pm\epsilon h,\Omega)=\TV(u, \Omega \setminus ({\overline{C_1}\cup \overline{C_2}}))+(1\pm\epsilon)\TV(u_1,\Omega)+(1\mp\epsilon)\TV(u_2,\Omega).
%\end{align*} 
%Note that $\Omega \setminus (\overline{C_1}\cap \overline{C_2})$ is open by construction, so the total variation of $u$ over this subset does not depend on $u_1$ or $u_2$. 
%By hypothesis, $\TV(u_i,\Omega)>0$ for every $i=1,..,n$. Therefore, there is always at least a non zero value in the sum.

\textbf{Case 1:} Let $u\in \partial B_{\mathring{P}_0}$ be a piecewise constant function, whose image contains at least three values $\{\beta_1,\beta_2,\beta_3\}$. Without loss of generality, assume $\beta_1\neq 0$, consider an indecomposable component $C_1\in S_\T(\Omega)$ of the preimage $u^{-1}(\beta_1)$, and define
\[h:=\beta_1\mathring{\1}_{C_1}.\]
This perturbation satisfies equation \ref{eq:perturbation}, and for $0<\epsilon<\delta_u/\Delta_h$, we get
\begin{align*}
    \TV(u\pm \epsilon h,\Omega)= \TV(u, \Omega \setminus \overline{C_1})+\sum_{T_i\subseteq C_1}\sum_{T_j\not\subseteq C_1}|(1\pm \epsilon)\beta_1-u(T_j)|\H^{d-1}(\partial T_i\cap \partial T_j)
\end{align*}
Since the image of $u$ has at least 3 values and $C_1$ is in the preimage of one of them, then the first term in the sum is always non zero. Therefore, $u$ is not extremal.

\textbf{Case 2:} Let $u:=\beta\mathring{\1}_E\in \partial B_{\mathring{P}_0}$ be a characteristic function over an indecomposable but not simple domain.  
Let $\{C_i\}_{i=1}^n$ be indecomposable components of $\Omega \setminus E$. By hypothesis, we know that $n\geq 2$, so define 
\[h:=\beta \mathring{\1}_{C_1}.\]
It clearly satisfies equation \ref{eq:perturbation}. Moreover, for any $0<\epsilon\leq \delta_u/\Delta_h$, we get
\begin{align*}
    \TV(u\pm\epsilon h,\Omega)=\TV(u,\Omega \setminus \overline{C_1})+\sum_{T_i\subseteq C_1}\sum_{T_j\not\subseteq C_1}|(1\pm \epsilon)\beta-u(T_j)|\H^{d-1}(\partial T_i\cap \partial T_j)
\end{align*}
where the first term cannot be zero because $u$ has variation at $\partial C_2$. Therefore $u$ is not extremal. 

\textbf{Case 3:} Let $u:=\beta\mathring{\1}_E\in \partial B_{\mathring{P}_0}$ be the characteristic function of a simple domain $E$, and assume by contradiction that $u$ is not an extremal point of $B_{\mathring{P}_0}$. Then, there exists a perturbation $h\in \Prc$ for which $\TV(u\pm \epsilon h)\neq 0$ for any $\epsilon>0$ sufficiently small and satisfying equation \eqref{eq:perturbation}, which implies that $h$ has variation only in $\partial E$. %Since $E$ is simple, then $\partial E$ is a closed loop. 
This would mean that $h\in \Span(u)$, which is a contradiction. Note also that 
\[\TV(\beta\mathring{\1}_E,\Omega)=|\beta| \Per(E,\Omega)=1 \qquad \implies \qquad \beta=\pm \frac{1}{\Per(E,\Omega)}.\]
Since we are working with zero average characteristic functions on simple domains, we have 
\begin{equation*}
    \frac{1}{ \Per(E,\Omega)} \mathring{1}_E=-\frac{1}{ \Per(\Omega\setminus E,\Omega)}\mathring{1}_{\Omega\setminus E},
\end{equation*}
which concludes the proof.
\end{proof}

We remark that Proposition \ref{prop:pwcextremals} above directly mimics the continuous case of extremal points of the $\TV$ ball in the whole of $\BV(\Omega)$, first proved in \cite{Fle57} (see \cite{AmbCasMasMor01} for a modern exposition) for $\Omega = \R^d$ and extended in \cite{BreCar20} to bounded $\Omega$.

\subsection{Discretization with continuous piecewise affine functions}
Let $\T$ be a triangulation in an open and bounded polyhedron $\Omega\subset \R^d$, and let $P_1(\T)$ be the set of continuous piecewise affine functions on $\Omega$. Analogously to the previous section, we aim to study the extremal points of 
\[B_{\mathring{P}_1}=\big\{f\in \Prl \,\vert\, \TV(f,\Omega)\leq 1\big\}.\]
We also have a convenient description of the total variation in this setting:
\begin{lemma}
    For any $u\in P_1(\T)$, the total variation corresponds to 
    \[\TV(u,\Omega)=\sum_{T\in \T}\L^d(T)\big|\nabla u\!\mid_{T}\!\big|,\]
    where since $\nabla u\mid_{T}$ is constant for all $T$ we have identified these with fixed vectors, and $|\cdot|$ is the Euclidean norm in $\R^d$.
    \begin{proof}
        A continuous piecewise affine function belongs in particular to $W^{1,p}(\Omega)$ for all $1 \leq p \leq \infty$. This implies \cite[Sec.~3.1]{AmbFusPal00} that 
        \[\TV(u,\Omega) = \int_\Omega |\nabla u| \dd x,\]
        for the weak derivative $\nabla u \in L^\infty(\Omega)$, for which we have $\nabla u \in \big[P_0(\T)\big]^d$.
    \end{proof}
\end{lemma}
An analytical study of the extremal points of $B_{\mathring{P}_1}$ needs to take into account the continuity constraints of the piecewise affine functions. Any variation of a function $u\in P_1(\T)$ at a vertex of the triangulation changes the value of $\nabla u$ at all triangles that contains that vertex. 
For this reason, one can simplify the study to the subset of `double tent functions' as in Figure \ref{fig:pwlnbhd}, from which we draw conclusions on the finiteness of the set of extremal points of $\mathring P_1(\T)$. To see that this is the case, we need two preliminary Lemmas.

\begin{lemma}
    A function $u\in P_1(\T)$ is completely described by its values at the nodes of the triangulation, i.e.
    \[P_1(\T)= \Span \big\{g_z \,\vert\, z \text{ is a vertex of }\T\big\},\]
    where $g_z$ is the continuous piecewise affine taking the value $1$ function at the vertex $z$ and vanishing at all other vertices, often called \textit{tent} function.
\end{lemma}
\begin{proof}
See for example \cite[Ch.~3]{BreSco08}.
\end{proof}

%Consider a regular equilateral triangulation $\T_e$ on $\Omega\subset\R^2$ of area $\lambda:=\frac{\sqrt{2}}{4}\ell$ for some $\ell>0$, and the set of continuous piecewise affine functions $\CPwL(\R^2, \T_e)$. As in the previous section we want to study the extremal points of the unit ball of the total variation on $\CPwL(\R^2, \T_e)$.

\begin{lemma}\label{lem:polyhedron}
    Let $Q$ be a polyhedron in $\R^n$. Then $Q \cap W$ is still a polyhedron for any $k$-dimensional plane $W\subset \R^n$.
    \begin{proof}
        $Q$ is defined as the intersection of a finite number of half-spaces in $\R^n$. Each half-space is defined as the set of points $x\in \R^n$ that satisfy a linear inequality of the form
        \[\nu\cdot x\leq b,\]
        where $\nu\in \R^n$ is a vector, and $b>0$ is a constant. A posteriori, $\nu$ is the outward normal vector of a $(n-1)$-dimensional facet of $Q$, and $b$ is the affine coordinate of the plane containing the facet.  
        The set $Q \cap W$ consists of the points $x\in W$ that satisfy
        \[\pi_W(\nu)\cdot x\leq b,\]
        where $\nu$ and $b$ are as above, and $\pi_W$ is the orthogonal projection onto $W$. This set corresponds to a polyhedron in $\R^k$ through any linear isomorphism between $W$ and $\R^k$.
    \end{proof}
\end{lemma}

\begin{figure}[ht]
    \centering
    \begin{tikzpicture}[scale=0.8]
        \filldraw[fill=yellow!70, very thin] (0,-1.75)--(2,-1.75)--(3,0)--(2,1.75)--(0,1.75)--(1,0)--(0,-1.75);
        \filldraw[fill=blue!30, very thin] (0,-1.75)--(-2,-1.75)--(-3,0)--(-2,1.75)--(0,1.75)--(-1,0)--(0,-1.75);
        \filldraw[fill=green!30, very thin] (0,-1.75)--(1, 0)--(0,1.75)--(-1,0)--(0,-1.75);
        \draw[very thin] (-2,-1.75)--(-1,0)
                (1,0)--(2,1.75)
                (-2,1.75)--(-1,0)
                (1,0)--(2,-1.75)
                (-3,0)--(3,0);
        \filldraw[very thin] (-1,0) circle(2pt) node[below left]{$z_2$}
                (1,0) circle(2pt) node[below right]{$z_1$};
        %\draw (0,-1.75) node[below]{$N_6$}
         %       (2,-1.75) node[below left]{$N_{7}$}
          %      (3,0) node[right]{$N_8$}
           %     (2,1.75) node[above right]{$N_9$}
            %    (0,1.75) node[above]{$N_{10}$}
             %   (-2,1.75) node[above left]{$N_3$}
              %  (-3,0) node[left]{$N_4$}
               % (-2,-1.75) node[below left]{$N_5$};
    \end{tikzpicture}
    \hspace{1cm}
    \begin{tikzpicture}[scale=0.7]
            \fill[black!20] (0,0)--(5.5,0)--(8.5,3)--(3,3);
            \filldraw[fill=blue!30, draw=black!70, very thin, opacity=0.9] (2,1.5)--(3.5,3.2)--(2,0.5)--(2,1.5);
            \filldraw[fill=blue!30, draw=black!70, very thin, opacity=0.9] (2,0.5)--(3.5,3.2)--(3.5,0.5)--(2,0.5);
            \filldraw[fill=green!30, draw=black!70, very thin, opacity=0.9] (3.5,0.5)--(3.5,3.2)--(5,4)--(3.5,0.5);
            \filldraw[fill=yellow!70, draw=black!70, very thin, opacity=0.9] (3.5,0.5)--(5,4)--(5,0.5)--(3.5,0.5);
            \filldraw[fill=yellow!70, draw=black!70, very thin, opacity=0.9] (5,0.5)--(5,4)--(6.5,1.5)--(5,0.5);
            \filldraw[fill=yellow!70, draw=black!70, very thin, opacity=0.9] (6.5,2.5)--(5,4)--(6.5,1.5)--(6.5,2.5);
            \draw[dotted] (6.5,2.5)--(3.5,2.5)--(2,1.5);
            
        \end{tikzpicture}
    \caption{Support and graph of a $P_1(\T)$ function with two contiguous nonzero vertex values, on a triangulation of $\R^2$ with equilateral triangles of equal area.}\label{fig:pwlnbhd}
\end{figure}
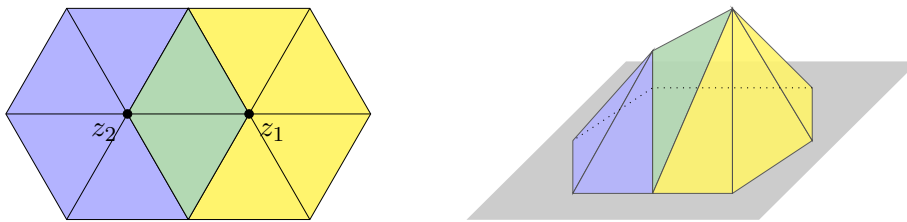
Since we will provide a negative result stating that the number of extremal points may not be finite and a proof in the general case would be cumbersome to write, we restrict ourselves to the particular case in which all $T \in \T$ are equilateral triangles in $\R^2$ of equal area $\lambda>0$, as depicted in Figure \ref{fig:pwlnbhd}. Let $T_1$ and $T_2$ be two adjacent triangles in $\T$, and $z_1,z_2$ be the two common vertices, so that we can classify the neighbor triangles based on which vertices they depend on. Let $C_2$ be the set of the triangles that do not have $z_1$ as a vertex (blue in Figure \ref{fig:pwlnbhd}), and $C_1$ be the set of triangles without $z_2$ as a vertex (yellow in Figure \ref{fig:pwlnbhd}). Let's denote by $C_{1,2}$ the set of the remaining two triangles, i.e. $C_{1,2}=\{T_1, T_2\}$. Consider the set 
\[W:=\Span\{g_{z_1},g_{z_2}\},\] which consists of the piecewise affine functions that have non-zero nodal values only at $z_1$ or $z_2$. 
\begin{figure}[ht]
    \centering
    \includegraphics[width=0.35\textwidth]{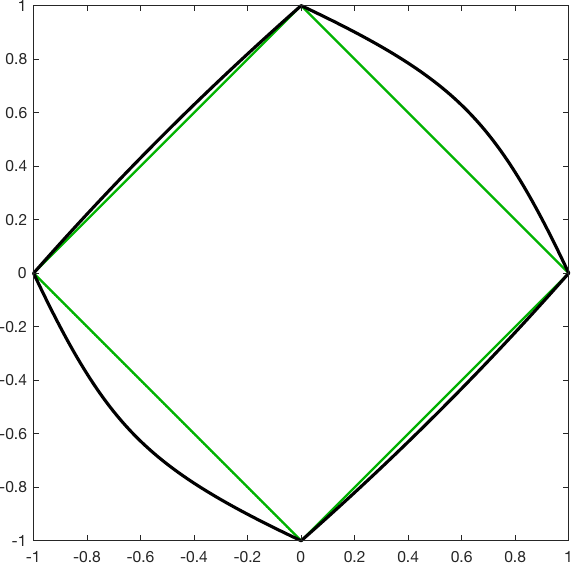}
    \caption{The set $\partial \mathfrak{B}_{1,2}$ for $\lambda = (4\sqrt{3}+\sqrt{11})^{-1}$ in black, and the $\ell^1$ ball in $\R^2$ in green.}
    \label{fig: B12 level}
\end{figure}

\begin{lemma}\label{lem:tvinW}
    The total variation in $W$ is 
    \[\TV(c_1 g_{z_1}+c_2 g_{z_2},\Omega)=4\sqrt{3}\lambda |c|_1+\lambda |c|_A, \quad \text{for} \quad |c|_A= \sqrt{c^\top Ac} \quad \text{with} \quad  A= \begin{pmatrix} 11 & -7 \\ -7 & 11  \end{pmatrix}\]
    for every~$c=(c_1,c_2) \in \R^2$.
    \begin{proof}
        Direct computation splitting the total variation into the components $C_1,C_2$ and $C_{1,2}$, i.e.
        \[\TV(u,\Omega)=\lambda \sum_{T\in C_1}\big|\nabla u\!\mid_T\!\big|+\lambda \sum_{T\in C_2}\big|\nabla u\!\mid_T\!\big|+\lambda \sum_{T\in C_{1,2}}\big|\nabla u\!\mid_{T}\!\big|.\qedhere\]
    \end{proof}
\end{lemma}
Based on Lemma \ref{lem:polyhedron}, we conclude the following result.
\begin{proposition}\label{prop:pwlextremals}
    Let $\T$ be a triangulation of $\Omega \subset \R^2$ in which a complete neighborhood of two adjacent vertices consists only of equilateral triangles. Then the ball $B_{\mathring{P}_1}(\T)$ has infinitely many extremal points.
\end{proposition}
\begin{proof}
First, note that for every~$c=(c_1,c_2) \in \R^2$, using Lemma \ref{lem:tvinW} we can write
\begin{align*}
    h(c)\coloneqq \TV(c_1 g_{z_1}+c_2 g_{z_2},\Omega)=4\sqrt{3}\lambda |c|_1+\lambda |c|_A, \quad \text{for} \quad |c|_A= \sqrt{c^\top Ac} \quad \text{with} \quad  A= \begin{pmatrix} 11 & -7 \\ -7 & 11  \end{pmatrix}
\end{align*}
 a positive definite matrix. Define~$\B_{1,2}:=B_{\mathring{P}_1}\cap W$ as well as 
\begin{align*}
    \mathfrak{B}_{1,2} \coloneqq \{\,c \in \R^2 \;|\; h(c)\leq 1\,\},
\end{align*}
with the latter depicted in Figure \ref{fig: B12 level}. Let us show that~$\Ext(\mathfrak{B}_{1,2})= \partial \mathfrak{B}_{1,2}$. For this purpose, it suffices to show that every boundary point is extremal, since the converse inclusion is trivial.
Indeed, assume that there is~$c_s, c_a,c_b \in \R^2$ as well as~$s \in (0,1)$ with
\begin{align*}
    h(c_a)=h(c_a)=h(c_s)=1, \quad c_s=(1-s)c_a+s c_b.
\end{align*}
Then there also holds
\begin{align*}
    |c_s|_A=(1-s) |c_a|_A+ s|c_b|_A.
\end{align*}
Since~$|\cdot|_A$ is a weighted Euclidean norm and~$s\in(0,1)$, this implies that~$c_a=t c_b$ for some~$t>0$. Due to~$h(c_a)=h(c_b)=1$ as well as the positive 1-homogeneity of~$h$, we finally conclude~$c_a=c_b=c_s$, hence,~$c_s$ is an extremal point of~$\mathfrak{B}_{1,2}$. As a consequence, neither~$\mathfrak{B}_{1,2}$ nor~$\B_{1,2}$ are polytopes. In view of Lemma~\ref{lem:polyhedron}, this proves that $B_{\mathring{P}_1}$ has infinitely many extremal points.
\end{proof}

\section{Practical realization of Algorithm~\ref{alg:abstractgcg}}\label{sec:realization}
Once the extremal points of~$B_{\Vr}$ are characterized, their structure can be exploited in order to devise efficient solution methods for the linear subproblems in Algorithm~\ref{alg:abstractgcg}. We again point out~\cite{CasDuvPet22} where the same philosophy is applied for~$\Omega=\R^d$ and ``simple'' forward operators $K$ induced by explicitly given convolution kernels. Since our main interest lies in more complicated examples such as the PDE-constrained optimization problem~\eqref{eq:PDEconstraint}, which inherently require a discretization of the space variable, and motivated by Section~\ref{subsec:discpiece}, we will further elaborate on the practical realization of Algorithm~\ref{alg:abstractgcg} for piecewise constant discretizations, i.e.,~$V=P_0(\T)$.  
In this case, relying on Proposition~\ref{prop:pwcextremals}, the linear subproblem is solved by
\begin{align*}
\widehat{v}_k=\1_{E_k}/\Per(E_k,\Omega) \quad \text{where} \quad  E_k \in  \argmax_{\substack{E \in \S_\T(\Omega)\\E~\text{simple}}} \frac{\int_{E} p_k~\mathrm{d}x}{\Per(E,\Omega)}.
\end{align*}
While the admissible set of the latter is finite, the restriction to simple sets significantly aggravates its efficient solution. For this reason, we drop this requirement and consider the relaxed problem
\begin{align} \label{def:fracprob}
    E_k \in  \argmax_{E \in \S_\T(\Omega)} \frac{\int_{E} p_k~\mathrm{d}x}{\Per(E,\Omega)}
\end{align}
for the computation of~$\widehat{v}_k$. Due to
\begin{align*}
    \Ext(B_{\mathring{P}_0}) \subset \left\{\, \mathring{\1}_{E}/\Per(E,\Omega) \;|\;E \in S_\T(\Omega) \,\right\} \subset B_{\mathring{P}_0}
\end{align*}
and Remark~\ref{rem:1k}, this relaxation does neither impact the convergence of the algorithm, Proposition~\ref{prop:sublinPDAP}, nor the upper bound in Lemma~\ref{lem:upperbound}.
In the following sections, we show that this, still challenging, fractional minimization can be replaced by a finite number of easier-to-solve subproblems which can be exactly solved by efficient graph-based algorithms.

\subsection{Dinkelbach-Newton method for the insertion step} \label{subsec:dinkelbach}
In the following, let~$u_k$ denote the $k-$th iterate generated by Algorithm~\ref{alg:abstractgcg}. This implies that the associated dual variable~$p_k=-K^*\nabla F(Ku_k)$ satisfies
\begin{align} \label{def:barlambdak}
    \Bar{\lambda}_k \coloneqq \left(\max_{E \subset \Omega}  \frac{\int_E p_k ~\mathrm{d}x}{\Per(E,\Omega)} \right)^{-1}\geq 1.
\end{align}
Furthermore, for~$\lambda \geq 0$, consider the regularized problem
\begin{align} \label{def:regsubprob}
    \min_{E \in S_\T(\Omega)} j_\lambda(E):=  \left \lbrack \Per(E,\Omega)-\lambda \int_E p_k~\mathrm{d}x \right \rbrack \tag{$P_\lambda$}
\end{align}
as well as its associated value function
\begin{align} \label{def:valuefunc}
    G : [0,\infty) \to \R,~\lambda \mapsto \min \eqref{def:regsubprob}.
\end{align}
Note that~\eqref{def:regsubprob} is well-posed since the set~$S_\T(\Omega)$ is finite. Denote by~$E_\lambda \in S_\T(\Omega)$ an arbitrary solution. The following proposition formalizes the connection between \eqref{def:regsubprob} and the fractional maximization problem~\eqref{def:fracprob}.
\begin{proposition} \label{prop:connection}
We have that:
\begin{itemize}
    \item There holds $G(\lambda)\leq 0 $ for all $\lambda \geq 0$, with equality if and only if $\lambda \leq \bar{\lambda}_k$.
    \item For~$\lambda < \Bar{\lambda}_k$, the solutions of~\eqref{def:regsubprob} are given by~$E_\lambda=\emptyset$ and $E_\lambda=\Omega$.
    \item A set~$\Bar{E} \in \S_\T(\Omega)$,~$\Per(\Bar{E}) \neq 0$, is a solution of~\eqref{def:regsubprob} for~$\lambda=\Bar{\lambda}_k$ if and only if it is a maximizer of~\eqref{def:fracprob}. 
    \end{itemize}
\end{proposition}
\begin{proof}
Since~$\emptyset \in \S_\T(\Omega)$, we immediately conclude~$G(\lambda)\leq 0$. Now,~if~$\lambda \leq \Bar{\lambda}_k $, for all~$E \in  \S_\T(\Omega)$ with $\Per(E,\Omega)\neq 0$ there holds
\begin{align*}
    \lambda^{-1} \geq  \frac{\int_{E} p_k~\mathrm{d}x}{\Per(E,\Omega)} \quad \text{and thus} \quad \Per(E,\Omega)- \lambda \int_{E} p_k~\mathrm{d}x \geq 0. 
\end{align*}
Noting that since $\int_\Omega p_k \dd x= 0$ we have~$j_\lambda(\emptyset)=j_\lambda(\Omega)=0$, we get~$G(\lambda)=0$ if~$\lambda \leq \Bar{\lambda}_k$. From this observation, we also conclude the second bullet point. Now assume that~$\lambda > \Bar{\lambda}_k$ and let~$\Bar{E} \in \S_\T(\Omega)$,~$\Per(\Bar{E}) \neq 0$, denote a maximizer of~\eqref{def:fracprob}. Then there holds
\begin{align*}
    0>j_\lambda(\Bar{E}) \geq G(\lambda)
\end{align*}
which finishes the proof of the first bullet point. Now, fix~$\lambda=\Bar{\lambda}_k$ and let $\Bar{E} \in \S_\T(\Omega)$,~$\Per(\Bar{E}) \neq 0$, denote a solution to~\eqref{def:fracprob}. Then, by construction, there holds~$j_\lambda(\Bar{E})=0$. Thus,~$\Bar{E}$ is a minimizer of~\eqref{def:fracprob}.
Conversely, let~$\Bar{E} \in \S_\T(\Omega)$,~$\Per(\Bar{E}) \neq 0$, denote a solution of~\eqref{def:regsubprob} for $\lambda=\Bar{\lambda}_k$. Then there holds
\begin{align*}
    j_\lambda(\Bar{E})=-\lambda \int_{\bar{E}}p_k~\mathrm{d}x+ \Per(\Bar{E},\Omega)=0 \quad \text{as well as} \quad \int_{\bar{E}}p_k~\mathrm{d}x >0
\end{align*}
Reordering yields
\[
 \Bar{\lambda}_k=\lambda=   \frac{\Per(\Bar{E},\Omega)}{\int_{\bar{E}}p_k~\mathrm{d}x} \quad \text{and thus} \quad  \Bar{E} \in  \argmax_{E \in \S_\T(\Omega)} \frac{\int_{E} p_k~\mathrm{d}x}{\Per(E,\Omega)}.\qedhere
\]
\end{proof}
As a consequence, solving the fractional maximization problem~\eqref{def:fracprob} is equivalent to computing a minimizer of~\eqref{def:regsubprob} with non-vanishing perimeter for~$\lambda=\Bar{\lambda}_k$. Since the latter is unknown, this requires two subroutines: first, a method to compute the largest zero of the value function~$G$ on~$[0,\infty)$ and, second, an efficient algorithm for the solution of the discrete minimization problem~\eqref{def:regsubprob}. We focus on the former in the following and derive a generalized Newton method which exhibits convergence in finitely many iterations. The efficient solution of~\eqref{def:regsubprob} is postponed to Section~\ref{subsec:graphs}. First, we require the following additional properties of the value function.
\begin{lemma} \label{lem:propval}
The value function $G$ is concave and locally Lipschitz continuous. For every~$\lambda >0$ there holds
\begin{align} \label{eq:estfordir}
     dG(\lambda,+) \leq -\int_{E_\lambda}p_k~\mathrm{d}x \leq -dG(\lambda,-) \quad \text{for all} \quad E_\lambda \in \argmin \eqref{def:regsubprob}.
\end{align}
where $dG(\lambda,\pm)$ denote the one-sided directional derivative of $G$ at $\lambda$ in the direction of $\pm 1$, respectively.
\end{lemma}
\begin{proof}
    Let $\lambda_1,\lambda_2 \geq 0$ and $\theta \in [0,1]$ be arbitrary but fixed. Set $\lambda_\theta=(1-\theta)\lambda_1+\theta\lambda_2$. By definition of the value function $G$ and that $G(\lambda_\theta)$ is attained at $E_{\lambda_\theta}$ we conclude 
    \begin{align*}
        (1-\theta) G(\lambda_1)+\theta G(\lambda_2)&=(1-\theta) j_{\lambda_1}(E_{\lambda_1})+\theta j_{\lambda_2}(E_{\lambda_2}) \\ 
        & \leq (1-\theta) j_{\lambda_1}(E_{\lambda_\theta})+\theta j_{\lambda_2}(E_{\lambda_\theta})
        \\&=\Per(E_{\lambda_\theta},\Omega)- \lambda_\theta \int_{E_{\lambda_\theta}} p_k ~\mathrm{d}x = j_{\lambda_\theta}(E_{\lambda_\theta}) = G(\lambda_\theta).
    \end{align*}
    Hence $G$ is concave and thus locally Lipschitz continuous. It remains to prove~\eqref{eq:estfordir}. On the one hand, for $\delta>0$ small enough, have
    \begin{align*}
        \frac{1}{\delta}(G(\lambda+\delta)-G(\lambda))=\frac{1}{\delta}(j_{\lambda+\delta}(E_{\lambda+\delta})-j_\lambda(E_\lambda)) \leq \frac{1}{\delta}(j_{\lambda+\delta}(E_{\lambda})-j_\lambda(E_\lambda))= -\int_{E_\lambda}p_k~\mathrm{d}x.
    \end{align*}
    On the other hand, we can estimate
    \begin{align*}
       -\int_{E_\lambda}p_k~\mathrm{d}x=\frac{1}{\delta}(j_{\lambda+\delta}(E_{\lambda})-j_\lambda(E_\lambda))=\frac{1}{\delta}(j_{\lambda}(E_{\lambda})-j_{\lambda-\delta}(E_\lambda))\leq -\frac{1}{\delta}(G(\lambda-\delta)-G(\lambda))
\end{align*}
    Taking the limit for $\delta \rightarrow 0$ in both inequalities finally yields the desired statement by definition of the directional derivatives, which always exist by virtue of $G$ being concave.
\end{proof}
We draw several conclusions from this auxiliary result. First, the value function~$G$ is not differentiable at~$\lambda=\Bar{\lambda}_k$. Indeed, denoting by~$E_k$ a maximizer of~\eqref{def:fracprob} with~$\Per(E_k,\Omega)\neq 0$, there holds 
\begin{align*}
    \mathrm{d}G(\Bar{\lambda}_k,+)\leq -\int_{E_{\Bar{\lambda}_k}}p_k~\mathrm{d}x<0=-\int_{\emptyset}p_k~\mathrm{d}x \leq-\mathrm{d}G(\Bar{\lambda}_k,-)
\end{align*}
due to~$j_{\Bar{\lambda}_k}(E_k)=0$ and~$\Per(E_k,\Omega)\neq 0$. However, secondly, we also deduce that~$G$ is semismooth at every~$\lambda>0$ and, due to~\eqref{eq:estfordir},~$-\int_{E_\lambda} p_k~\mathrm{d}x$ provides a Clarke subgradient of~$G$ at~$\lambda$. Hence, we can hope to compute~$\bar{\lambda}_k$ by using a generalized (or semismooth) Newton method, see \cite[Ch.~2]{Ulb11} for an introduction to such methods and the definitions of the concepts mentioned. More in detail, starting from~$\lambda_0=1 \geq \Bar{\lambda}_k $, we iteratively update  
\begin{align*}
  \lambda_{\ell+1}=\lambda_{\ell}+ \frac{G(\lambda_\ell)}{\int_{E_{\lambda_\ell}} p_k~\mathrm{d}x}= \frac{\lambda_\ell \int_{E_{\lambda_\ell}} p_k~\mathrm{d}x+G(\lambda_\ell)}{\int_{E_{\lambda_\ell}} p_k~\mathrm{d}x}= \frac{\Per(E_{\lambda_\ell},\Omega)}{\int_{E_{\lambda_\ell}} p_k~\mathrm{d}x}
\end{align*}
whenever~$\int_{E_{\lambda_\ell}} p_k~\mathrm{d}x \neq 0$.
The resulting algorithm, also referred to as the Dinkelbach-Newton method \cite{schaible}, is given in Algorithm~$\ref{alg:dinkelbach}$. Note that the method terminates in two scenarios. The first is if~$\lambda_0=1$ satisfies~$G(\lambda_0)=0$, in which case we immediately conclude
\begin{align*}
\max_{v \in \Ext(B_{\Vr}) } \int_\Omega p_k v~\mathrm{d}x=    
\max_{E \in \S_\T(\Omega)} \frac{\int_{E} p_k~\mathrm{d}x}{\Per(E,\Omega)}=1,    
\end{align*}
i.e.,~$u_k$ is a minimizer of~\eqref{def:BVprob}.
The second is if there holds~$G(\lambda_{L+1})=0$ for some~$L>0$, implying that we have
\begin{align*}
    \lambda_{L+1}= \Bar{\lambda}_k,~\Per(E_{\lambda_L},\Omega)>0 \quad \text{as well as} \quad j_{\lambda_{L+1}}(E_{\lambda_L})=0
\end{align*}
by construction. As a consequence,~$E_{\lambda_L}$ is also a solution of Problem~\eqref{def:regsubprob} for~$\lambda=\lambda_{L+1}$ and thus, see Proposition~\ref{prop:connection}, a maximizer of~\eqref{def:fracprob}.
\LinesNumbered
\begin{algorithm}[ht]\label{alg:dinkelbach}
\setstretch{1.15}
\caption{Dinkelbach-Newton method for Problem~\eqref{def:fracprob}}
    \KwInput{Initial parameter~$\lambda_0=1$ and dual variable~$p_k=-K^*\nabla F(Ku_k)$.}
    Find $E_{\lambda_0} \in \argmin_{E \in \S_\T(\Omega)} \Per(E, \Omega)-\lambda_0 \int_E p_k~\mathrm{d}x$.
    \\
    \If{$G(\lambda_{0})=0$}{Set~$\Bar{\lambda}_k=\lambda_0$ as well as $E_k=\emptyset$ and stop.}
    \For{$G(\lambda_{\ell}) < 0$}{
            $\lambda_{\ell+1}\leftarrow \Per(E_{\lambda_{\ell}},\Omega)\big/\int_{E_{\lambda_\ell}} p_k~\mathrm{d}x$\\  
            Find $E_{\lambda_{\ell+1}} \in \argmin_{E \in \S_\T(\Omega)} \Per(E, \Omega)-\lambda_{\ell+1}\int_E p_k~\mathrm{d}x$.\\
            \If{$G(\lambda_{\ell+1})=0$}{Set~$\Bar{\lambda}_k=\lambda_\ell$ as well as $E_k=E_{\lambda_\ell}$ and stop.}
            $\ell\leftarrow \ell+1$
    }
    \KwOutput{$\bar{\lambda}_k$  and~$E_k$ satisfying~\eqref{def:fracprob} and~\eqref{def:barlambdak}.}
\setstretch{1}
\end{algorithm}
The following proposition addresses the wellposedness of Algorithm~\ref{alg:dinkelbach} and summarizes some of its key properties.
\begin{proposition}
Let~$\lambda_\ell$ denote the current iterate in step~$\ell \in \N$  of Algorithm~\ref{alg:dinkelbach}. Then we either have~$\lambda_\ell=\bar{\lambda}_k$ or there holds~$\lambda_\ell > \Bar{\lambda}_k$. In the second case,~$\lambda_{\ell+1}$ is well-defined and there holds~$\Bar{\lambda}_k \leq \lambda_{\ell+1} < \lambda_{\ell} $.
\end{proposition}
\begin{proof}
We only prove the statement for~$\ell=0$, the general case follows mutatis mutandis. If~$\lambda_0=1 \neq \Bar{\lambda}_k$,  Proposition~\ref{prop:equivalence} yields~$G(\lambda_0) <0$ and thus~$\int_{E_{\lambda_{0}}} p_k~\mathrm{d}x >0$. Consequently,~$\lambda_1$ is well-defined and there holds
\begin{align*}
    \lambda_{1}=\lambda_{0}+ \frac{G(\lambda_0)}{\int_{E_{\lambda_0}} p_k~\mathrm{d}x}< \lambda_{0} \quad \text{as well as} \quad \lambda_{1}=\frac{\Per(E_{\lambda_0},\Omega)}{\int_{E_{\lambda_0}} p_k~\mathrm{d}x} \geq \Bar{\lambda}_k
\end{align*}
due to~$\int_{E_{\lambda_{0}}} p_k~\mathrm{d}x >0$,~$G(\lambda_0)<0$ as well as the definition of~$\Bar{\lambda}_k$.
\end{proof}
Following, e.g.,~\cite{schaible} we readily obtain that the iterates~$\lambda_\ell$ generated by Algorithm satisfy
\begin{align*}
\frac{1}{\Bar{\lambda}_k} - \frac{1}{\lambda_{\ell+1}}  \leq c_\ell \left(\frac{1}{\Bar{\lambda}_k} - \frac{1}{\lambda_\ell} \right) \quad \text{for some sequence} \quad c_\ell\geq 0 \quad \text{with} \quad c_\ell \searrow 0,
\end{align*}
i.e., their reciprocals converge superlinearly towards the maximum value of Problem~\eqref{def:fracprob}.
Exploiting the finite cardinality of~$S_{\mathcal{T}}(\Omega)$ as well as the strict monotonicity of~$\lambda_\ell$, we further conclude the finite termination of Algorithm~\ref{alg:dinkelbach}.  
\begin{proposition} \label{prop:convergence}
Let~$\lambda_\ell$ be generated by Algorithm~\ref{alg:dinkelbach}. Then there is~$L\in \N$ such that~$\lambda_{L}=\bar{\lambda}_k$, i.e., Algorithm~\ref{alg:dinkelbach} stops after finitely many steps.
\end{proposition}
\begin{proof}
Assume that Algorithm~\ref{alg:dinkelbach} does not stop after finitely many iterations. Then there holds
\begin{align*}
    \int_{E_{\lambda_\ell}} p_k~\mathrm{d}x >0 \quad \text{for all} \quad \ell \in \N.
\end{align*}
By definition, we further get
\begin{align*}
  \frac{\Per(E_{\lambda_{\ell+1}},\Omega)}{\int_{E_{\lambda_{\ell+1}}} p_k~\mathrm{d}x}=  \lambda_{\ell+2} < \lambda_{\ell+1}=\frac{\Per(E_{\lambda_{\ell}},\Omega)}{\int_{E_{\lambda_{\ell}}} p_k~\mathrm{d}x} \quad \text{for all}~\ell \in \N.
\end{align*}
Consequently, we necessarily have~$E_{\ell_{1}}\neq E_{\ell_{2}}$ for all~$\ell_{1},\ell_{2} \in \N$,~$\ell_{1}\neq \ell_{2}$. Since~$S_{\mathcal{T}}(\Omega)$ is finite, we arrive at a contradiction.
\end{proof}
\subsubsection{Graph cuts for Problem~\texorpdfstring{(\ref{def:regsubprob})}{Plambda} } \label{subsec:graphs}
Finally, it remains to discuss the efficient solution of the perimeter minimization problem \eqref{def:regsubprob} which is required in Algorithm~\ref{alg:dinkelbach}. For this purpose, we briefly recap its equivalence to a suitable minimal graph cut problem, see also~\cite{Picard1975}, which can be solved exactly and efficiently using, for example, maximum flow algorithms such as \cite{BoKo2004}. In the following, let $G=(\mathcal{V},\mathcal{E},\mathcal{W})$ be the dual graph associated to $\T$, with a node for each $T_j \in \T$ and where the weight of each graph edge $(i,j)$ is given by the length of the common face of the simplices $T_i,T_j$, i.e. $w_{i,j}=\H^{d-1}(\partial T_i\cap \partial T_j)$.
% Already in \cite{Picard1975}, the authors suggested the use of minimal graph cuts for solving discrete versions of perimeter minimizing problems of the form \eqref{def:regsubprob}. But it was not until a significant improvement in the algorithm efficiency (\cite{BoKo2004}), that this approach experienced a renewed interest.

% Let $\Omega\subset \R^d$ be a triangulated domain, and denote by $\S_\T(\Omega)$ the collection of subsets of $\Omega$ obtained as the union of closed triangles of the triangulation $\T$ in $\Omega$. Let $G=(V,E,W)$ be the Voronoi dual graph associated to $\T$, where the weight of each edge $(i,j)$ is given by the length of the common edge of the triangles $T_i,T_j$, i.e. $w_{i,j}=\H^{d-1}(T_i\cap T_j)$. 
Thanks to this construction, the perimeter in $\Omega$ of any set $E \in \S_\T(\Omega)$ can be computed through a sum of weights of edges in $\mathcal{E}$. Indeed, let $E\in \S_\T(\Omega)$, and denote by $I\subset \mathcal{V}$ the collection of vertices that correspond to triangles in $E$. Then,
\begin{equation}
   \Per(E,\Omega)=\sum_{T_i\subset E}\sum_{T_j\not\subset E} \H^{d-1}(\partial T_i\cap \partial T_j)=\sum_{i\in I}\sum_{j\in V\setminus I} w_{i,j}.
\end{equation}
Therefore, the dual graph encodes the perimeter information of the elements of $\S_\T(\Omega)$.
In order to incorporate information regarding the dual variable $p_k$, an additional construction is required. For this purpose, let us introduce two new vertices, denoted as $s$ (the source) and $t$ (the sink), in the set of vertices $\mathcal{V}$. Next, for each node $i\in \mathcal{V}$, we create directed edges $(s,i)$ from the source to the node $i$, and $(i,t)$ from the node to the sink. Finally, noting that
\begin{align*}
    \lambda \int_E p_k~\mathrm{d}x= \sum_{T_i \subset E} \lambda p_k(T_i)\L^d(T_i),
\end{align*}
we can encode the integral $-\lambda \int_E p_k\;dx$ into the weights of the new edges by defining
\begin{align*}
    w_{s,i}&:=\max\{0,-\lambda p_k(T_i)\L^d(T_i)\} \quad \text{as well as} \quad w_{i,t}&:=\max\{0,\lambda p_k(T_i)\L^d(T_i)\} \quad \text{for all} \quad i \in \mathcal{V}.
\end{align*}
The resulting graph is commonly referred to as an $(s,t)$-graph. Visually, this construction adds an extra "dimension" to the graph, which can be illustrated graphically as shown in Figure \ref{fig:s,t-graph}.

\begin{figure}[ht]
    \centering
    \begin{tikzpicture}[scale=1.2]
    \foreach \i in {0,2}{
    \foreach \k\t in {1.25/0.2,0.9/0.5, 2.1/0.5,1.75/0.8}{
        \draw[red!40,thin] (3,-1)--(\k+\i,\t);
        \draw[red!40,thin] (3,-1)--(\k+\i+1,\t+1);
        } 
    }
    \draw[thick] (0,0)--(6,2)
        (2,2)--(4,0)
        (1,1)--(5,1)
        (2,0)--(4,2)
        (1,1)--(4,2)--(5,1)--(2,0)--(1,1);
    \filldraw[dashed, fill=black!10, opacity=0.7] (0,0)--(4,0)--(6,2)--(2,2)--(0,0);
    \foreach \i in {0,2}{
    \foreach \k\t in {1.25/0.2,0.9/0.5, 2.1/0.5,1.75/0.8}{
        \filldraw[red] (\k+\i,\t) circle(0.7pt);
        \filldraw[red] (\k+\i+1,\t+1) circle(0.7pt);
        \draw[red!40,very thin,->] (3,3)--(\k+\i,\t);
        \draw[red!40,very thin,->] (3,3)--(\k+\i+1,\t+1);
        } 
    }
    \foreach \k in {0,1}{
        \draw[red!50] (2.1+\k,0.5+\k)--(1.75+\k,0.8+\k)--(0.9+\k,0.5+\k)--(1.25+\k,0.2+\k)--(2.1+\k,0.5+\k)--(2.9+\k,0.5+\k)--(3.25+\k,0.2+\k)--(4.1+\k,0.5+\k)--(3.75+\k,0.8+\k)--(2.9+\k,0.5+\k);
    }
    \filldraw[red] (3,3) circle(0.7pt) node[above]{\tiny $s$};
    \filldraw[red] (3,-1) circle(0.7pt) node[below]{\tiny $t$};
    \draw[red!50] (1.75,0.8)--(2.25,1.2)
        (3.75,0.8)--(4.25,1.2);
    \draw (4.2,0) node[right]{\tiny $(\Omega,\T)$};
    \end{tikzpicture}
    \begin{tikzpicture}[scale=1.2]
    \foreach \i in {0,2}{
    \foreach \k\t in {1.25/0.2,0.9/0.5, 2.1/0.5,1.75/0.8}{
        \draw[red!40,thick] (3,-1)--(\k+\i,\t);
        %\draw[red!40,dotted] (3,-1)--(\k+\i+1,\t+1);
        } 
    }
    \draw[thick] (0,0)--(6,2)
        (2,2)--(4,0)
        (2,0)--(4,2)
        (1,1)--(4,2)--(5,1)--(2,0)--(1,1);
    \filldraw[dashed, fill=black!10, opacity=0.7] (0,0)--(4,0)--(6,2)--(2,2)--(0,0);
    \draw[thick] (1,1)--(5,1);
    \fill[blue!,opacity=0.1](0,0)--(4,0)--(5,1)--(1,1);
    \foreach \k in {0,1}{
        \draw[red!50,loosely dotted, thick] (2.1+\k,0.5+\k)--(1.75+\k,0.8+\k)--(0.9+\k,0.5+\k)--(1.25+\k,0.2+\k)--(2.1+\k,0.5+\k)--(2.9+\k,0.5+\k)--(3.25+\k,0.2+\k)--(4.1+\k,0.5+\k)--(3.75+\k,0.8+\k)--(2.9+\k,0.5+\k);
    }
    \draw[red!50,thick] (1.75,0.8)--(2.25,1.2)
        (3.75,0.8)--(4.25,1.2);
    \foreach \i in {0,2}{
    \foreach \k\t in {1.25/0.2,0.9/0.5, 2.1/0.5,1.75/0.8}{
        \draw[red!40,thick] (3,3)--(\k+\i+1,\t+1);
        \filldraw[blue] (\k+\i,\t) circle(0.7pt);
        \filldraw[green] (\k+\i+1,\t+1) circle(0.7pt);
        %\draw[red!40,dotted,->] (3,3)--(\k+\i,\t);
        } 
    }
    \filldraw[blue] (3,3) circle(0.7pt) node[above]{\tiny $s$};
    \filldraw[green] (3,-1) circle(0.7pt) node[below]{\tiny $t$};
    \draw (4.2,0) node[right]{\tiny $(\Omega,\T)$};
    \end{tikzpicture}
    \caption{Our triangulated domain $(\Omega,\T)$ in black, and the associated $(s,t)$-graph in red. On the right, an example of $(s,t)$-cut (blue vs green) for the unweighted $(s,t)$-graph, with the edges passing through the cut (i.e. connecting different color nodes) in red.}
    \label{fig:s,t-graph}
\end{figure}
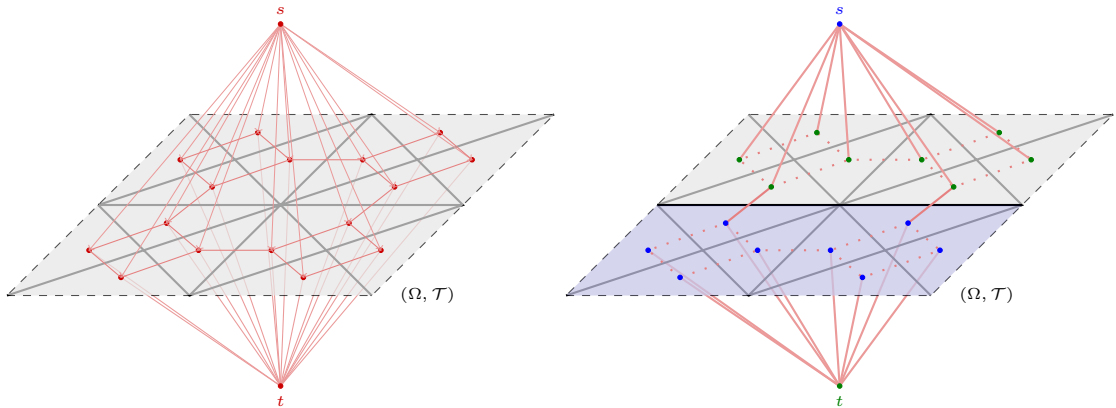
Given the $(s,t)$-graph as well as~$I \subset \mathcal{V}$, we define the associated $(s,t)$-\textit{cut} as the binary partition of the nodes of the graph into $I\cup\{t\}$ and $I^c\cup\{s\}$. The respective \textit{cut value} is given by the flow through the cut, i.e.
\begin{equation}
    \Theta_I:=\sum_{i\not\in I}\sum_{j\in I} w_{i,j}+\sum_{i\in I} w_{s,i}+\sum_{i\not\in I}w_{i,t}.
\end{equation}
Finally, a \textit{minimal cut} is an $(s,t)$ cut that minimizes the cut value. The following lemma addresses the equivalence between finding a minimal cut and solving the perimeter minimization problem~\eqref{def:regsubprob}.
\begin{lemma}
    Let $I\subset \mathcal{V}$  and $E(I)=\operatorname{int}\left(\bigcup_{i\in I} \overline{T_i}\right) \in \S_\T(\Omega)$ be the corresponding union of triangles. If $I\cup \{t\}$, $I^c\cup \{s\}$ is a minimal cut, then 
    \begin{equation}\label{eq:cut_value}
        E(I)\in \argmin_{E\in \S_\T(\Omega)}\;\Per(E,\Omega)-\lambda \int_E p_k \;dx
    \end{equation}
    \end{lemma}
    \begin{proof}
    By definition of the cut value, we have
        \begin{align*}
    \Theta_I&=\Per(E(I),\Omega)+\sum_{i\in I} w_{s, i}+\sum_{j\not\in I} w_{j,t}\\
    &=\Per(E(I),\Omega) +\sum_{i\in I} \max(0,-\lambda p_k(T_i) \L^d(T_i))+\sum_{i \in \mathcal{V}\setminus I} \max(0,\lambda p_k(T_i)\L^d(T_i))\\
    &=\Per(E(I),\Omega)-\sum_{i\in I}\lambda p_k(T_i)\L^d(T_i)+\sum_{i \in \mathcal{V}} \max(0,\lambda p_k(T_i)\L^d(T_i))\\
    &=\Per(E(I),\Omega)-\lambda\int_{E(I)} p_k\mathrm{d}x +\sum_{i \in \mathcal{V}}\max(0,\lambda p_k(T_i)\L^d(T_i)).
\end{align*}
    Since the last term is independent of $E(I)$, minimizing the cut value would minimize \eqref{eq:cut_value}.
    \end{proof}
    The full practical realization of FC-GCG for~$P_0$ elements is given in Algorithm~\ref{alg:dinkelbachgcg}. Note that the scalars~$\zeta_k$ can be computed as a by-product of Algorithm~\ref{alg:dinkelbachgcg} and provide an upper bound 
    on the suboptimality~$J(u_k)-\min_{u\in P_0(\mathcal{T})}J(u)$, by Lemma \ref{lem:upperbound}.
% \begin{remark}
%     When the computation of the perimeter at the boundary of $\Omega$ is required, it is sufficient to redefine the weights of the source and the sink edges as the positive and negative part of
%     \[\lambda_0 p_k\L^d+\H^{d-1}\mres \partial \Omega.\]
% \end{remark}

% \begin{lemma}
%     if $\{\lambda_0,...,\lambda_{n}\}$ is the sequence of coefficients, the $n$-th cut of Dinkelbach minimizes
%     \begin{equation}
%        \Per(U)+\lambda_n\int_U p_k+\int_\Omega \max(0,-v_0)+\sum_{i=1}^{n}\int_\Omega \max(0,-(\lambda_i-\lambda_{i-1}) p_k).
%     \end{equation}
% \end{lemma}
\begin{algorithm}[ht]\label{alg:dinkelbachgcg}
\setstretch{1.15}
\caption{Dinkelbach-FC-GCG for $V=P_0(\T)$ discretization of Problem~\eqref{def:BVprob}}
    \KwInput{$u_0=0$, $\mathcal A_0=\emptyset$, $N_0=0$, $k=0$, $\zeta_0=1$}
    calculate $c^0 \in \argmin_{c \in \R} F(cK_{\1})$ \\
    $u_0\leftarrow c^0 \1_{\Omega}$\\
    \While{$\zeta_k > 0$}{
        $p_k\leftarrow -K_*\nabla F(K u_k)$\\
        calculate $(\overline{\lambda}_k, E_k)$ using Algorithm \ref{alg:dinkelbach} with inputs $\lambda_0 = 1$ and $p_k$\\
        $N_k^+\leftarrow N_k+1$\\
        $u^k_{N_k^+} \leftarrow \1_{E_k}/ \Per(E_k, \Omega)$\\
        $\mathcal A_k^+\leftarrow \mathcal A_k\cup \{u^k_{N_k^+}\}$\\
        calculate (using SSN):
       % \begin{equation}
       %      (\gamma^{k+1},c^{k+1})\in \!\!\! \argmin_{(\gamma,c)\in \R_+^{N_k^+} \times \R} J_\triangledown(\gamma,c,\mathcal{A}_{k}) = \!\!\! \argmin_{(\gamma,c)\in \R_+^{N_k^+} \times \R} \left[F\left(cK_{\1}+\sum_{i=1}^{N_k^+}\gamma_i Ku_i^k\right)+\sum_{i=1}^{N_k^+}\gamma_i\right]
       %  \end{equation}
        \begin{equation}
            (\gamma^{k+1},c^{k+1})\in \!\!\! \argmin_{(\gamma,c)\in \R_+^{N_k^+} \times \R} \!\!F\left(K\mathcal{U}_{\mathcal{A}^+_{k}}(\gamma) + cK_{\1}\right)+\sum_{i=1}^{N_k^+}\gamma_i \quad\text{with minimum value }M_k
        \end{equation}
        $\!\!\zeta_{k+1} \leftarrow \big(-1\big/\,\overline{\lambda}_k - 1\big)M_k$\\
        $\!u_{k+1}\leftarrow \mathcal{U}_{\mathcal{A}^+_k}(\gamma^{k+1}) + c^{k+1}\1_\Omega = \sum_{i=1}^{N_k^+}\gamma^{k+1}_i u_i^k + c^{k+1}\1_\Omega$\\
        %$u_{k+1}\leftarrow c^{k+1}\1_\Omega + \sum_{i=1}^{N_k^+}\gamma_i^{k+1} u_i^k$\\
        $\mathcal{A}_{k+1}\leftarrow \mathcal{A}_{k}^+\setminus \{u_i^k:\gamma_i^{k+1}=0\}$\\        
        $N_{k+1}\leftarrow \#\mathcal{A}_{k+1}$\\        
        $k \leftarrow k+1$
    }
    \KwOutput{$\overline{u}= u_{k}$}
\setstretch{1}
\end{algorithm}
Finally, again exploiting the finite cardinality of~$S_\mathcal{T}(\Omega)$, we conclude the convergence of Algorithm~\ref{alg:dinkelbachgcg} in finitely many steps, assuming that the finite dimensional minimization problem for~$(\gamma^{k+1}, c^{k+1})$ is solved exactly.
\begin{proposition} \label{prop:finitestepgcg}
Let~$u_k$ be generated by Algorithm~\ref{alg:dinkelbachgcg}. Then there is~$K \geq 0$ such that~$\Bar{u}=u_K$ is a minimizer of Problem~\eqref{def:BVprob} for~$V=P_0(\mathcal{T})$.
\end{proposition}
\begin{proof}
Without loss of generality, assume that the iterate~$u_k=$ in the $k$-th step of Algorithm~\ref{alg:dinkelbachgcg},~$j \geq 1$, is not optimal.  
Then, utilizing arguments from~\cite{BreCarFanWal23}, there holds
\begin{align*}
    F(K u_{k+1})+ \sum^{N^+_k}_{i=1} \gamma^{k+1}_i <F(K u_{k})+ \sum^{N^+_{k-1}}_{i=1} \gamma^{k}_i .
\end{align*}
In particular, this implies~$\mathcal{A}_{k+1} \neq \mathcal{A}_j$ for all~$j \in \N$,~$j \leq k$.
Since the power set of~$S_{\mathcal{T}}$ is also finite, we conclude the termination of Algorithm~\ref{alg:dinkelbachgcg} in finitely many steps.
\end{proof}
\begin{remark} \label{rem:comparison}
The procedure summarized in Algorithm~\ref{alg:dinkelbachgcg} bears some similarities to a version of the conditional gradient method for the deconvolution of zero-mean images presented in \cite{HarJudNem15}. In this earlier work, the authors bypass the nonsmoothness of regularizer by introducing a dummy variable bounding the total variation of the image and considering the constrained minimization problem   
  \begin{align*}
      \min_{\substack{u \in \Vr,\ r \in \R}} \left \lbrack F(Ku)+ r \right \rbrack \quad\text{ s.t. }\quad \TV(u, \Omega) \leq r,~0 \leq r \leq D  
  \end{align*}
  where~$D$ is a given bound on the total variation of optimal reconstructed images. Subsequently, (accelerated) variants of the conditional gradient method are applied for the simultaneous update of the image~$u_k$ and the bound~$r_k$. Similar to Algorithm~\ref{alg:dinkelbachgcg}, this requires the repeated solution of linear programs of the form
  \begin{align} \label{eq:primal}
      \max_{ v \in \Vr} \int_\Omega p_k v~\mathrm{d}x \quad\text{ s.t. }\quad \TV(v,\Omega) \leq 1. 
  \end{align}
  which is realized by applying an interior point method to the dual of the equivalent problem
  \begin{align} \label{eq:altprimal}
      \min_{v \in \Vr} \TV(v,\Omega)  \quad\text{ s.t. }\quad \int_\Omega p_k v~\mathrm{d}x \geq 1.
  \end{align}
  The exact, geometry exploiting, solution approach considered in the present work, i.e. Algorithm~\ref{alg:dinkelbach}, can also be motivated from this perspective but outperformed black-box linear programming approaches for~\eqref{eq:primal} and \eqref{eq:altprimal} as well as their respective dual problems significantly in our numerical tests. Indeed, formulating the saddle-point problem for the Lagrangian of~\eqref{eq:altprimal}, we observe that  
  \begin{align*}
      \max_{\lambda \geq 0} \min_{v \in \Vr} \left \lbrack \TV(v,\Omega)-\lambda \int_\Omega p_k v~\mathrm{d}x+ \lambda  \right \rbrack &= 
      \max_{\lambda \geq 0} \min_{E \in S_{\mathcal{T}}(\Omega)} \left \lbrack \Per(E,\Omega)-\lambda \int_E p_k ~\mathrm{d}x+ \lambda  \right \rbrack
      \\ &=\max_{\lambda \geq 0} \left \lbrack G(\lambda)+\lambda \right \rbrack= \Bar{\lambda}_k 
  \end{align*}
  due to~$G(\lambda) \leq 0$ for all~$\lambda \geq 0$ and since~$\Bar{\lambda}_k$ is the largest zero of~$G$. Hence, Algorithm~\ref{alg:dinkelbach} computes a saddle-point of the Lagrangian in a finite number of iterations using the characterization of~$Ext(B_{\Vr})$. 
\end{remark}
\section{Discrete-to-continuum limits of TV problems on P0 functions}\label{sec:disccont}
In this section, for sequences of triangulations $\T_h$ with $h>0$ an upper bound on the diameter of the triangles, we investigate the behavior of minimizers of problems of the form \eqref{eq:PDEconstraint} formulated over functions $P_0(\T_h)$ as $h \to 0$. Specifically, we consider minimization problems
\begin{align} \label{def:discreteproblem}
    \min_{u_h \in P_0(\T_h) } J_h(u_h) \coloneqq \left \lbrack F(K_h u_h) + \TV(u_h,\Omega) \right \rbrack.
\end{align}
where $K_h$ is defined by~$y_h=K_h u_h \in P_1(\T_h) \cap H^1_0(\Omega)$ being the unique solution to a Galerkin weak formulation of a linear elliptic PDE such as
\begin{align}\label{eq:discweakform}
   \text{find}~y \in P_1(\T_h) \cap H^1_0 (\Omega) \quad \text{s.t. }\  \int_\Omega \nabla y \cdot \nabla \theta_h + y \theta_h ~\mathrm{d}x= \int_\Omega u \theta_h~\mathrm{d}x
\end{align}
for all $\theta_h \in P_1(\T_h) \cap H^1_0(\Omega)$, and the corresponding continuum $K$ is the solution operator of the analogous weak formulation in $H^1_0(\Omega)$ without the $P_1(\T_h)$ constraint.

The challenge, as already indicated in the introduction (see the references therein), is that the total variation does not behave consistently with respect to approximation with piecewise constant functions. Therefore, we consider an approach based on $\Gamma$-convergence (see \cite[Ch.~1]{Bra02} for the definition and basic properties) in which limits with respect to $h$ of solutions over each mesh $\T_h$ are proved to be minimizers of a continuum variational problem, which in turn depends on which sequence of meshes $\T_h$ was used. Such an approach is based on formulating both the approximating and the limiting energies in the same functional space. To this end, we define the functional $\TV_{\T_h}:L^1(\Omega) \to \R^+ \cup \{0, +\infty\}$ by
\[\TV_{\T_h}(u):=\begin{cases}\TV(u,\Omega) &\text{if }u \in P_0(\T_h),\\+\infty &\text{otherwise},\end{cases}\]
and the corresponding perimeter
\[\Per_{\T_h}(E):=\begin{cases}\Per(E,\Omega) &\text{if }E \in \S_{\T_h}(\Omega)\\+\infty &\text{otherwise}.\end{cases}\]
In this framework and under suitable assumptions on $\T_h$, we will see that in the limiting energies the total variation term is replaced with an anisotropic version defined by
\begin{equation}\begin{aligned}\label{eq:deftvrho}
\TV_{\varphi}(u) :&= \sup \left\{ \int_\Omega u \,\div \psi \dd x \,\middle\vert\, \psi \in \mathcal{C}^1_c(\Omega; \R^d),\ \varphi^\ast(\psi(x)) \leq 1 \text{ for all } x \in \Omega\right\} \\&= \int_\Omega \varphi\left(\frac{Du}{|Du|}\right)\dd |Du|,
\end{aligned}\end{equation}
for some convex positively one-homogeneous real valued function $\varphi$, and with $Du/|Du|$ denoting the Radon-Nikodym derivative of $Du \in M(\Omega;\R^d)$ with respect to the nonnegative measure $|Du|$.

A first step towards this goal is to notice that this type of convergence for the perimeters implies the same type of convergence for the PDE-constrained optimization problem with $\TV$ regularization, and convergence of the corresponding minimizers. To this end, we recall that a sequence of sets $E_n \subset \R^d$ is said to converge in $L^1$ to a limit set $E$ whenever \[\L^d(E_n \triangle E) = \L^d(E_n \setminus E)+\L^d(E \setminus E_n) \xrightarrow[n \to \infty]{} 0.\] It is straightforward to check that the quantity $(E,F) \mapsto \L^d(E \triangle F)$ defines a metric on Lebesgue equivalence classes of subsets of $\R^d$, which in particular allows us to speak of $\Gamma$-convergence with respect to this topology.

\begin{proposition}\label{prop:gamamconvJ}
Let $1 \leq d \leq 4$, $1 \leq p < d/(d-1)$, and $\T_h$ be triangulations of $\Omega$ (which we assume is polyhedral) satisfying $h \geq \max \{\diam(T)\,\vert\,T \in \T_h\}$ and such that there is $C_s >0$ with
\begin{equation}\label{eq:nondegenerate}\frac{1}{h}\sup\big\{ r> 0 \,\vert\, B(x,r) \subset T \text{ for some }x\in T\big\} \geq C_s  \quad\text{for all }h > 0 \text{ and } T \in \T_h.\end{equation}
Assume that $F$ is lower semicontinuous with respect to the strong topology of $L^p(\Omega)$, and that the discrete perimeter $\Per_{\T_h}$ $\Gamma$-converges with respect to the $L^1$ convergence of sets as $h \to 0$ to the anisotropic perimeter functional given by
\[A \mapsto \int_{\partial^\ast A} \varphi(\nu_A(x)) d \H^{d-1}(x),\]
where $\nu_A = D\1_A/|D\1_A|$ and $\partial^\ast A$ is the reduced boundary of $A$. Then we have
\begin{equation}\label{eq:gammaconvSum}J_h := F\circ K_h + \TV_{\T_h} \xrightarrow{\Gamma-L^p} J:= F\circ K + \TV_\varphi,\end{equation}
and any sequence $\bar{u}_h$ of minimizers of $J_h$ has a subsequence converging to a minimizer $\bar{u}$ of $J$ strongly in $L^p(\Omega)$, for which $\TV(\bar{u}_h) \to \TV_\varphi(\bar{u})$.
\end{proposition}
\begin{proof}
By \cite[Prop.~3.5]{ChaGiaLus10} the $\Gamma$-convergence of $\Per_{\T_h}$ with respect to $L^1$ convergence of sets is equivalent to the convergence of the corresponding total variations, so we have
\begin{equation}\label{eq:gammaconvL1}\TV_{\T_h} \xrightarrow{\Gamma-L^1(\Omega)} \TV_\varphi.\end{equation}

Now assume $u_h \wkto u$ weakly in $L^{d/(d-1)}(\Omega)$, and let us consider the discrete and continuous weak formulations \eqref{eq:discweakform}
with unique solution $y_h \in P_1(\T_h) \cap H^1_0(\Omega)$, and
\begin{align}\label{eq:contweakform}
   \text{find}~y \in H^1_0 (\Omega) \ \text{s.t.}\  \int_\Omega \nabla y \cdot \nabla \theta + y \theta ~\mathrm{d}x= \int_\Omega u \theta~\mathrm{d}x \quad\text{for all }\theta \in H^1_0(\Omega),
\end{align}
with unique solution $y \in H^1(\Omega)$. Choosing $\theta=y_h$ in \eqref{eq:discweakform}, using the H\"older inequality and then Poincar\'e-Sobolev inequality in $H^1_0(\Omega)$ for $y_h$, we obtain the estimate
\begin{equation}\label{eq:discgradientboud}\begin{aligned}\|\nabla y_h\|^2_{L^2(\Omega)} &= \int_\Omega |\nabla y_h|^2 \dd x \leq \int_\Omega |u_h y_h|\dd x \leq \|u_h\|_{L^{d/(d-1)}(\Omega)} \|y_h\|_{L^d(\Omega)} \\ &\leq C_\Omega \|u_h\|_{L^{d/(d-1)}(\Omega)} \|\nabla y_h\|_{L^2(\Omega)},\end{aligned}\end{equation}
so that $\|\nabla y_h\|_{L^2(\Omega)} \leq C_\Omega \|u_h\|_{L^{d/(d-1)}(\Omega)}$. Notice that this is possible because $\Omega$ is bounded and
\begin{equation}\label{eq:exponents}\left(\frac{d-1}{d}\right)' = d \leq \frac{2d}{d-2}\end{equation}
is satisfied for $d \leq 4$, with equality if $d=4$. Again using the Poincar\'e inequality, estimate \eqref{eq:discgradientboud} means that the sequence $y_h$ is bounded in $H^1(\Omega)$. Therefore, using Banach-Alaoglu and Rellich-Kondrakov there is $y_\ell \in H^1_0(\Omega)$ such that (up to a subsequence) we have $\nabla y_h \wkto \nabla y_\ell$ weakly in $L^2(\Omega)$. For any $\theta \in C_c^\infty(\Omega)$ and by the assumptions on $\T_h$, we can find (see for example \cite[Thm.~4.4.20]{BreSco08}) a sequence $\theta_h \in P_1(\T_h) \cap H^1_0(\Omega)$ such that $\theta_h \to \theta$ strongly in $H^1(\Omega)$ and use it in \eqref{eq:discweakform}. Using the Poincar\'e-Sobolev inequality in $H^1_0(\Omega)$ for $\theta_h$ and taking into account \eqref{eq:exponents} we see that
\begin{gather*}\nabla y_h \wkto \nabla y \text{ in }L^2(\Omega), \quad\nabla \theta_h \to \nabla \theta \text{ in }L^2(\Omega),\\ u_h \wkto u \text{ in }L^{d/(d-1)}(\Omega), \ \ \theta_h \to \theta \text{ in }L^d(\Omega),\end{gather*}
so we can pass to the limit on both sides of \eqref{eq:discweakform} to obtain that $y_\ell$ satisfies 
\begin{equation}\label{eq:fakeweakform}\int_\Omega \nabla y_\ell \nabla \theta \dd x = \int_\Omega u \theta \dd x \quad\text{ for all }\theta \in C_c^\infty(\Omega).
\end{equation}
Using the density of $C_c^\infty(\Omega)$ in the strong topology of $H^1_0(\Omega)$ we have that $y_\ell$ in fact satisfies \eqref{eq:contweakform}, and uniqueness of solutions for this weak formulation implies that $y_\ell = y$.

Let us now prove \eqref{eq:gammaconvSum}. If we consider $u_h \to u$ strongly in $L^p(\Omega)$ for the $\Gamma$-$\liminf$ inequality, then since $\Omega$ is bounded also $u_h \to u$ strongly in $L^1(\Omega)$. This allows us to use \eqref{eq:gammaconvL1} to obtain $\TV_{\varphi}(u) \leq \liminf_{h \to 0} \TV_{\T_h}(u_h)$. Moreover, if $\TV_\varphi(u) < +\infty$ this limit inferior can be approximated with (not relabelled) subsequences for which $\TV_{\T_h}(u_h) < +\infty$, so that $u_h \in P_0(\T_h)$ and $\TV_{\T_h}(u_h) = \TV(u_h) \leq C$. By weak compactness and the Poincar\'e-Sobolev inequality in $\BV(\Omega)$ \cite[Thm.~3.23, Thm.~3.44]{AmbFusPal00} and since $\Omega$ is bounded we have (after possibly taking a further subsequence) that $u_h \wkto u$ in $L^{d/(d-1)}$, which combined with the compactness and continuity of $y_h$ proved above yields
\[(F \circ K)(u) + \TV_{\varphi}(u) \leq \liminf_{h \to 0} (F \circ K_h)(u_h) + \TV_{\T_h}(u_h) \text{ for all }u_h \to u \text{ in }L^p(\Omega),\]
where we note that this inequality also holds if $\TV_\varphi(u) = + \infty$, since in that case by the use of \eqref{eq:gammaconvL1} also the right hand side must equal $+\infty$. For the other half of the $\Gamma$-convergence, if we are given $u \in L^2(\Omega)$ we would like to find $u_h \to u$ strongly in $L^p(\Omega)$ for which 
\begin{equation}\label{eq:gammalimsup}(F \circ K)(u) + \TV_{\varphi}(u) \geq \limsup_{h \to 0} (F \circ K_h)(u_h) + \TV_{\T_h}(u_h),\end{equation}
and may assume that $\TV_\varphi(u) < +\infty$, since otherwise the inequality is trivial. Using the $\Gamma$-convergence of the discerte perimeters \eqref{eq:gammaconvL1} we can find a recovery sequence of functions $u_h$ such that $u_h \to u$ in $L^1(\Omega)$, $\TV_{\T_h}(u_h) < \infty$ and $u_h \in P_0(\T_h)$ for all $h$, and 
\[\limsup_{h\to 0} \TV_{\T_h}(u_h) = \limsup_{h\to 0} \TV(u_h) \leq \TV_\varphi(u).\]
Since $\TV_\varphi(u) < +\infty$ we have that $\TV(u_h)$ is bounded and we can proceed as above to see that, up to taking a not relabelled subsequence, also $u_h \to u$ strongly in $L^p(\Omega)$ and $y_h \to y$ strongly in $L^q(\Omega)$ for $q < d$, and from that conclude that for this particular sequence \eqref{eq:gammalimsup} holds.

To conclude the convergence of minimizers, we notice that $\TV_{\T_h}(u) \geq \TV(u)$ for every $u \in L^p(\Omega)$, which in combination with Assumption $\mathbf{A}3$ and the Poincar\'e-Sobolev inequality in $\BV(\Omega)$ that the family $J_h$ is equicoercive in the norm of this space. This allows us, for a sequence of minimizers $\bar{u}_h$ and by the same compact embedding arguments used before, to extract a subsequential limit $\bar{u}$ in the strong $L^p(\Omega)$ and weak $L^{d/(d-1)}$ topologies, and by the fundamental theorem of $\Gamma$-convergence \cite[Cor.~7.20, Thm.~7.8]{Dal93} this limit must be a minimizer of $J$, and \begin{equation}\label{eq:convofenergy}J_h(\bar{u}_h)=\lim_{h \to 0} J(\bar{u}).\end{equation} 
Moreover, since $\bar{u}_h \wkto \bar{u}$ in $L^{d/(d-1)}$, we can run the same continuity argument for $K(\bar{u}_h)$ to infer from \eqref{eq:convofenergy} that $\lim_{h \to 0}\TV(\bar{u}_h) = \TV_\varphi(\bar{u})$.
\end{proof}

The above result is in particular applicable for periodic triangulations of the cube $\Omega=(0,1)^d$:
\begin{definition} \label{def:periodictriang}
     Given an initial triangulation $\T_1=\{T_1,...,T_{n_\tau}\}$ of $\Omega$, we say that $\{\T_k\}_k$ is a sequence of periodic triangulations on $\Omega=(0,1)^d$ if
\[ \T_k = \bigcup_{i_1,\ldots,i_d=0}^{k-1} \frac{1}{k}\big(\T_1 + (i_1,\ldots,i_d)\big).\]
\end{definition}
In the previous equation, the two operations are conceptualized as shifting and rescaling of the elements of $\T_1$ to facilitate the division of $\Omega$ into $k^d$ identical cubes of size $1/k$ with a rescaled copy of $\T_1$ inside each. A typical example of such triangulations is the `double diagonal' depicted in Figure \ref{fig:latticegraph}, which will be our main focus later on.

\begin{proposition}\label{prop:periodicconv}
If $\{\T_k\}_k$ is a sequence of periodic triangulations on $\Omega = (0,1)^d$, 
the discrete perimeter $\Per_{\T_k}$ $\Gamma$-converges as $k \to \infty$ and with respect to the $L^1$ convergence of sets, to the anisotropic perimeter functional given by
\begin{equation}\label{eq:phiper}A \mapsto \int_{\partial^\ast A} \varphi_{\T_1}\big(\nu_A(x)\big) d \H^{d-1}(x),\end{equation}
where $\varphi_{\T_1}$ is Lipschitz continuous and given at rational directions $\nu$ by the single cell formula
\begin{equation}\label{eq:onecellformula}\varphi_{\T_1}(\nu) = 3^{d-1} \inf \left\{ |Du|\big(\left[1/3,2/3\right]^d\big) \,\middle\vert\, u \in P_0(\T_3) \text{ s.t. } u(x) - \nu \cdot x \text{ is }1/3\text{-periodic}\right\},\end{equation}
in which the piecewise constant functions involved are real-valued, satisfy
\begin{align*}
    u(x+1/3 z) - \nu \cdot (x+1/3 z)=u(x) - \nu \cdot x \quad \text{for all} \quad z \in \mathbb{Z}^d,
\end{align*}
and the cost includes jumps at $\partial [1/3,2/3]^d$.
\end{proposition}
\begin{proof}
We aim to use results from the literature on discrete-to-continuum convergence of variational problems on lattice systems, specifically the integral representation of \cite[Thm.~2.4]{BraPia13} and cell formula of \cite[Prop.~2.6]{ChaKre23}. To this end, we define the lattice dual to $\T_k$ consisting of the centroids of the simplices in $\T_k$, that is
\[L_k = \bigcup_{i_1,\ldots,i_d=0}^{k-1} \frac{1}{k}\big(\left\{\ce(T^1_1), \ldots, \ce(T^1_{n_\tau})\right\} + (i_1,\ldots,i_d)\big).\]
The obstacle to apply the mentioned results directly is that they are formulated for the embedding into continuum sets which are unions of Voronoi cells of the lattice, whereas the sets we consider are unions of triangles of $\T_k$. However, for any subset $A \in \S_{\T_k}(\Omega)$ we can define 
\[L_k \supset \Xi_k(A) = \{\ce(T_1),\ldots, \ce(T_{N_A})\} \text{ whenever } A = \bigcup_{i=1}^{N_A} T_i.\]
Noticing that the $\Xi_k$ so defined is a bijection, we define the corresponding discrete perimeter on subsets of $L_k$ given by
\[\Per_{L_k}(B):= \Per_{\T_k}\left(\Xi^{-1}_k(B)\right) \text{ for each } B \subset L_k.\]
Given this correspondence of the interaction energies, to conclude it is enough to check that for all sequences \[\{A_k\} \subset \S_{\T_k}(\Omega) \text{ with }\Per_{\T_k}(A_k) \leq M < +\infty\] we have
\begin{equation}\label{eq:convcorr} A_k \xrightarrow{L^1} A \subset (0,1)^d \quad \text{if and only if}\quad (V_k\!\circ \Xi_k)(A_k) \xrightarrow{L^1} A,\end{equation}
where $V_k$ assigns to subsets of $L_k$ the union of closures of cells of the Voronoi tesellation induced by $L_k$, which is the continuum embedding for lattice subsets used in \cite{BraPia13} and \cite{ChaKre23}. To prove \eqref{eq:convcorr}, we can denote the set of indices of simplices of $A_k$ having one or more faces in $\partial A_k$ as 
\[I_k = \left\{(i_1,\ldots,i_d,\ell) \, \middle\vert\, \frac{1}{k}\big(T^1_\ell + (i_1,\ldots,i_d)\big) \subset A_k \text{ and } \frac{1}{k}\big(\partial T^1_\ell + (i_1,\ldots,i_d)\big) \cap \partial A_k \neq \emptyset \right\}.\]
Noticing that only these simplices and corresponding Voronoi cells can contribute to the symmetric difference between $A_k$ and $(V_k\!\circ \Xi_k)(A_k)$, that their volume must be bounded by the size of a complete periodicity cube, and that $\Per_{\T_k}(A_k) = \mathcal{H}^{d-1}(\partial A_k \cap \Omega)$, we estimate
\[\begin{aligned}\L^d\Big(A_k \triangle \big[(V_k\!\circ \Xi_k)(A_k)\big]\Big) &\leq k^{-d} \,\#\big(I_k\big) \\ &\leq k^{-d} \Big(\!\min_{\ell,F}\big\{\H^{d-1}(F) \,\big\vert\, F \text{ a face of } \partial T^1_\ell\big\}\Big)^{-1}M k^{d-1} \xrightarrow[k\to \infty]{} 0.
\end{aligned}\]
Having proved \eqref{eq:convcorr}, $\Gamma$-convergence of $\Per_{\T_k}$ is equivalent to that of the corresponding lattice system, which gives rise to a functional a functional of the form \eqref{eq:phiper} as proved in \cite[Thm.~2.4 and Sec.~2.1.1]{BraPia13}, and allows us to express $\varphi$ using the cell formula \eqref{eq:onecellformula} as a specialization of \cite[Prop.~2.6]{ChaKre23}. For this formula, since the infimum needs to take into account interactions with the neighboring cubes, we have chosen to formulate it on the middle cube of $\T_3$, but clearly other analogous expressions are possible.
\end{proof}

\begin{figure}[ht]
    \centering
    \begin{minipage}{.33\textwidth}
    \begin{tikzpicture}[scale=0.6]
    \fill[color=lightgray!50] (0,0) rectangle (4,4);
    \fill[color=lightgray!20] (0,0)--(2,-2)--(4,0)--cycle;
    \fill[color=lightgray!20] (4,0)--(6,2)--(4,4)--cycle;
    \fill[color=lightgray!20] (4,4)--(2,6)--(0,4)--cycle;
    \fill[color=lightgray!20] (0,0)--(-2,2)--(0,4)--cycle;
    \foreach \i\j in {3.333/6,0.666/6,6/3.333,6/0.666,3.333/-2,0.666/-2,-2/3.333,-2/0.666}{
    \filldraw[red!30] (\i,\j) circle(3pt);
    }
    \foreach \i\j in {4.666/6,-0.666/6,6/4.666,6/-0.666,4.666/-2,-0.666/-2,-2/4.666,-2/-0.666}{
    \filldraw[red!30] (\i,\j) circle(3pt);
    }
    \draw[thick](0,0)--(4,0)--(4,4)--(0,4)--(0,0)
                (0,0)--(4,4)
                (0,4)--(4,0);
    \draw
         (0,4)--(2,6)--(4,4)
         (4,0)--(6,2)--(4,4)
         (0,4)--(-2,2)--(0,0)
         (4,0)--(2,-2)--(0,0);
    \draw
         (4,4)--(4,6)
         (4,4)--(6,4)
         (4,4)--(6,6)
         (0,0)--(0,-2)
         (0,0)--(-2,0)
         (0,0)--(-2,-2)
         (4,0)--(4,-2)
         (4,0)--(6,0)
         (4,0)--(6,-2)
         (0,4)--(0,6)
         (0,4)--(-2,4)
         (0,4)--(-2,6);
    \foreach \i\j in {2/3.333,2/4.666,3.333/2,4.666/2,2/0.666,2/-0.666,0.666/2,-0.666/2}{
    \filldraw[red] (\i,\j) circle(3pt);
    } 
    \draw[red, thick]
        (2,3.333)--(3.333,2)--(2,0.666)--(0.666,2)--(2,3.333)
        (2,3.333)--(2,4.666)
        (3.333,2)--(4.666,2)
        (2,0.666)--(2,-0.666)
        (0.666,2)--(-0.666,2);
    \end{tikzpicture}
    \end{minipage}
    \hspace{1.5cm}
    \begin{minipage}{.33\textwidth}
    \begin{tikzpicture}[scale=0.6]
    \fill[color=lightgray!50] (0,0) rectangle (4,4);
    \fill[color=lightgray!20] (0,0)--(2,-2)--(4,0)--cycle;
    \fill[color=lightgray!20] (4,0)--(6,2)--(4,4)--cycle;
    \fill[color=lightgray!20] (4,4)--(2,6)--(0,4)--cycle;
    \fill[color=lightgray!20] (0,0)--(-2,2)--(0,4)--cycle;
    \foreach \i\j in {5/5,3/5,-1/5,-1/3,-1/-1,1/-1,5/-1,5/1}{
    \filldraw[red!30] (\i,\j) circle(3pt);
    }
    \draw[thick](0,0)--(4,0)--(4,4)--(0,4)--(0,0)
                (0,0)--(4,4)
                (0,4)--(4,0);
    \draw
         (0,4)--(2,6)--(4,4)
         (4,0)--(6,2)--(4,4)
         (0,4)--(-2,2)--(0,0)
         (4,0)--(2,-2)--(0,0);
    \draw[blue, dashed]
        (2,-2)--(2,6)
        (-2,2)--(6,2);
    \draw[blue, dashed]
        (6,-2)--(6,6)
        (-2,-2)--(-2,6)
        (-2,6)--(6,6)
        (-2,-2)--(6,-2)
        (-2,4)--(0,4)
        (0,4)--(0,6)
        (4,4)--(4,6)
        (4,4)--(6,4)
        (4,0)--(6,0)
        (4,0)--(4,-2)
        (0,0)--(-2,0)
        (0,0)--(0,-2);
    \foreach \i\j in {1/1,1/3,3/1,3/3,3/-1,1/5,5/3,-1/1}{
    \filldraw[red] (\i,\j) circle(3pt);
    }
    \draw[red, thick]
        (1,1)--(1,3)--(3,3)--(3,1)--(1,1)        
        (3,3)--(1,5)
        (3,1)--(5,3)
        (1,3)--(-1,1)
        (1,1)--(3,-1);
    \end{tikzpicture}
    \end{minipage}
    \caption{Representing the regular `double diagonal' triangulation of the plane as interactions of periodic lattice systems, with the shaded central square being a complete period. The red segments represent the interactions corresponding to the thicker black edges of the triangulation. The assignment of Proposition \ref{prop:periodicconv} is depicted on the left, with the triangles themselves being the Voronoi cells. For the alternative construction on the right using the regular lattice $\frac{1}{2k}\Z^2$ each triangle is assigned counterclockwise to a red point, and the corresponding dashed blue square Voronoi cell has the same area.}
    \label{fig:latticegraph}
\end{figure}
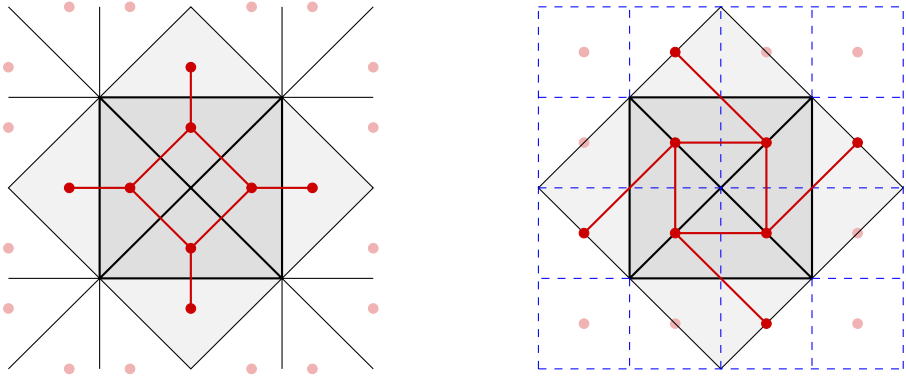

\begin{remark}
If $\T_k$ is the `double-diagonal' triangulation in the plane with period $1/k$, we can also make an explicit correspondence to an interaction energy on $\frac{1}{2k}\Z^2$, as depicted in the right diagram of Figure \ref{fig:latticegraph}. We remark as well that even though for this very symmetric triangulation the Voronoi cells corresponding to the centroids of the triangles (see the left diagram of Figure \ref{fig:latticegraph}) are the triangles themselves, this is not true for more general triangulations, and in general Voronoi cells may also intersect the complement of their corresponding periodicity cube.
\end{remark}

\begin{remark}
When using our algorithm on pseudo-random fine triangulations from a mesh generator (see Section \ref{sec:computations} below), we empirically observe solutions resembling those corresponding to the isotropic case $\varphi \equiv 1$. Proving a result trying to formalize this behavior is beyond the scope of this article, but we point out that such results on isotropy of the limiting energy exist in the literature on stochastic homogenization of lattice systems, see for example \cite{AliCicRuf15}.
\end{remark}

\begin{remark}
Note that Proposition~\ref{prop:sublinPDAP} guarantees a worst-case convergence rate of
\begin{align*}
    J_h(u^h_k)- \min_{u_h \in P_0(\T_h)} J_h(u_h) \leq \frac{c_h}{k+1}
\end{align*}
if~$u^h_k$ is generated by Algorithm~\ref{alg:dinkelbachgcg} and where~$c_h >0$ is a, potentially,~$h$-dependent constant. In view of the~$\Gamma$-convergence result in Proposition~\ref{prop:gamamconvJ}, a natural question that arises is whether its worst-case convergence behavior is actually independent of the underlying mesh and consequently does not degenerate as~$h \rightarrow 0$. While a complete treatment of this topic is outside of the scope of the present manuscript, we briefly discuss the stability of the constant~$c_h$ for the particular case of~$F(\cdot)= (1/2) \|\cdot-y_o\|^2_{L^2(\Omega)}$ for some~$y_o \in L^2(\Omega)$. In this case, we readily verify that there holds 
\begin{align*}
   C \|u^h_k\|_{L^p(\Omega)} \leq J_h(u^h_k) \leq M_h \coloneqq \min_{c\in \R} \|cK_h \1_\Omega-y_o\|^2_{L^2(\Omega)}   
\end{align*}
for some~$C>0$ independent of~$h$ and~$k$. Then, following the exposition in~\cite{TraWal23} we obtain that
\begin{align*}
    c_h = c \left( r^h_0, M_h \right) \quad \text{where} \quad r^h_0=  \min_{c\in \R} \|cK_h \1_\Omega-y_o\|^2_{L^2(\Omega)}- \min_{u_h \in P_0(\T_h)} J_h(u_h).
\end{align*}
and where the right hand side depends continuously on the stated parameters. As in the proof of Proposition~\ref{prop:gamamconvJ}, we get 
\begin{align*}
    M_h \rightarrow \min_{c\in \R} \|cK \1_\Omega-y_o\|^2_{L^2(\Omega)} \quad \text{as well as} \quad r^h_0 \rightarrow \min_{c\in \R} \|cK \1_\Omega-y_o\|^2_{L^2(\Omega)} - \min_{u \in L^p(\Omega)} J(u).
\end{align*}
\end{remark}
\subsection{Explicit anisotropy for planar double-diagonal triangulations}

We begin with a `slicing' result, inspired by the continuous analogue (see for example \cite[Lem.~3.2]{BraMasSig08}), to modify sequences of triangulated sets converging to a half plane so that the intersection of their boundaries and $\partial (0,1)^2$ match the corresponding one for a half plane with rational normal: for $\nu=(\sigma,\tau)\in S^1$, we will denote by $E_\nu$ the subset of $(0,1)^2$ given by 
\begin{equation}\label{eq:Enudef}
    E_\nu=\begin{cases}
        \{x\in (0,1)^2: x\cdot \nu>0 \} &\text{ if } \sigma\tau\leq 0,\\
        \{x\in (0,1)^2: |x\cdot \nu|>1\} &\text{ if } \sigma\tau>0.
    \end{cases}
\end{equation}

\begin{definition}
    We say that a triangulation $\T$ on $\Omega=(0,1)^2$ is diagonally separable if there exist simple sets $A_1,A_2\in S_\T(\Omega)$ such that 
    \begin{equation}
        \partial A_i\cap \partial \Omega= \partial E_{\nu_i}\cap \partial \Omega, 
    \end{equation}
    where $\nu_1=\frac{1}{\sqrt{2}}(1,1)\in S^1$, and $\nu_2=\frac{1}{\sqrt{2}}(-1,1)\in S^1$.
\end{definition}
Note that if $\T_1$ of a sequence of periodic triangulations $\{\T_k\}_k$ is diagonally separable, then so is any $\T_k$ of the sequence.

\begin{lemma}\label{lem:discretedegiorgi}
    Let $\{\T_k\}_k$ be a sequence of periodic triangulations of $\Omega=(0,1)^2$, such that $\T_1$ is diagonally separable, and $E_\nu$ as in \eqref{eq:Enudef}. For any rational direction $\nu\in S^1$, and any sequence $\{E_k\in S_{\T_k}(\Omega)\}_k$ converging in $L^1$ to $E_\nu$,
    there exists a (sub)sequence $\{\w{E}_{k_m}\in S_{\T_{k_m}}(\Omega)\}_m$ satisfying $\partial \w{E}_{k_m} \cap \partial \Omega=\partial E_\nu \cap \partial \Omega$, and 
    \begin{equation}\label{eq:modnoaddper}
        \lim_{m \to \infty} \; \H^1(\Omega\cap (\partial \w{E}_{k_m}\setminus \partial E_{k_m}))=0.
    \end{equation}
\begin{proof}
    Let $\nu=\frac{1}{\sqrt{\sigma^2+\tau^2}}(\sigma,\tau)\in S^1$ be a rational direction for some coprime integers $\sigma,\tau\in \Z$ and let $n=\max\{|\sigma|,|\tau|\}$. We assume without loss of generality that neither $\sigma$ nor $\tau$ is equal to zero. If this is not the case, then we can construct the new sequence as $\w{E}_k:=E_\nu$ which is admissible for every $k$ and has zero perimeter inside $\Omega$.
    
    For $\delta>0$, consider the neighbourhood of $\partial \Omega$ given by
    \begin{equation}
        \Omega_{\delta}=\left\{x\in \Omega: d(x,\partial \Omega)<\delta\right\},
    \end{equation}
    where $d(x,\partial\Omega)=\inf_{y\in \partial\Omega}\|x-y\|$. It is clear that $\Omega_{\delta}\in S_{\T_k}(\Omega)$ only for specific choices of $k$ and $\delta$. More precisely, given $\delta=\frac{1}{m}$ for some fixed $m\in \N$, the previous statement holds if $k$ is a multiple of $m$. For this reason, we will only consider elements of the sequence of triangulations $\T_k$ satisfying $k=\ell m n$ for some $\ell\in \N$. We set the index $k$ to also be a multiple of $n$ so that the boundary of $E_\nu$ intersects $\partial \Omega_{1/m}$ precisely at vertices of periodic squares of $\T_k$, as depicted in Figure \ref{fig:slicing}. It will also be the case for any other neighborhood of $\partial \Omega$ of the form $\Omega_{\frac{nj}{k}}$ with $j\in \{1,...,\ell\}$, where in particular, $\frac{nj}{k}=\frac{1}{m}$ if $j=\ell$.

    Let 
    \begin{equation*}
        C(\T_1):=\max_{T\in \T_1} \frac{\H^1(\partial T)}{\L^2(T)}.
    \end{equation*}
    be a bound on the shape of triangles in the initial triangulation. For every triangle $T\in \T_{k}$ and $j\in \left\{1,...,\ell\right\}$, since $d=2$ and triangles in $\T_k$ are obtained by rescaling triangles of $\T_1$ by a factor $1/k$, we have 
    \begin{equation*}
        \H^1\left(\partial \Omega_{\frac{nj}{k}}\cap (E_{k}\triangle E_\nu)\cap \overline{T}\right)\leq C(\T_1)k \L^2\big(E_{k}\triangle E_\nu)\cap \overline{T}\big).
    \end{equation*}
    The previous inequality holds true even if we sum the lefthand side over $j\in \left\{1,...,\ell\right\}$. Indeed, by the assumption on the initial triangulation, each triangle in $\T_{k}$ intersects with at most one slice $\partial \Omega_{\frac{nj}{k}}$ (see Figure \ref{fig:slicing}). Thus, 
    \begin{equation*}
        \sum_{j=1}^{\ell}\H^1\left(\partial \Omega_{\frac{nj}{k}}\cap (E_{k}\triangle E_\nu)\cap \overline{T}\right)\leq C(\T_1)k \L^2\big((E_{k}\triangle E_\nu)\cap \overline{T}\big)
    \end{equation*}
    Summing the righthand side of the inequality over $T\in \T_{k}$ gives $C(\T_1)k \L^2\big(E_{k}\triangle E_\nu\big)$ because triangles are disjoint. On the other hand, for the lefthand side, we have
    \begin{equation}
        \sum_{T\in \T_{k}} \H^1\left(\partial \Omega_{\frac{nj}{k}}\cap (E_{k}\triangle E_\nu)\cap \overline{T}\right)= 2 \H^1\left(\partial \Omega_{\frac{nj}{k}}\cap (E_{k}\triangle E_\nu)\right).
    \end{equation}
    To see this, take a look at Figure \ref{fig:slicing}, and note that, by periodicity, each piece of $\partial \Omega_{\frac{nj}{k}}$ (in blue) appears in two triangles of $\T_{k}$.
    All together,
    \begin{equation}\label{eq: 1-slicing}
        \sum_{j=1}^{\ell}\H^1\left(\partial \Omega_{\frac{nj}{k}}\cap (E_{k}\triangle E_\nu)\right)\leq \frac{1}{2}C(\T_1) k \L^2\big(E_{k}\triangle E_\nu\big).
    \end{equation}
    One of the terms of the sum in Equation \eqref{eq: 1-slicing} must be not larger than the average. Therefore, there exists $j(\ell,m)\in \left\{1,...,\ell\right\}$ for which
    \begin{equation}\label{eq: good-slice}
        \H^1\left(\partial \Omega_{\frac{nj(\ell,m)}{k}}\cap (E_{k}\triangle E_\nu)\right)\leq  \frac{1}{\ell}\sum_{j=1}^{\ell}\H^1\left(\partial \Omega_{\frac{nj}{k}}\cap (E_{k}\triangle E_\nu)\right).
    \end{equation}
    Together with Equation \eqref{eq: 1-slicing}, we conclude that 
    \begin{equation}
        \H^1\left(\partial \Omega_{\frac{nj(\ell,m)}{k}}\cap (E_{k}\triangle E_\nu)\right)\leq \frac{m n}{2}C(\T_1) \L^2\big(E_{k}\triangle E_\nu\big).
    \end{equation}
    By the $L^1$ convergence of $E_{k}$ to $E_\nu$, we know that for every $m\in \N$, there exists $\ell(m)\in N$ such that for each $\ell\geq \ell(m)$, 
    \begin{equation}
        \L^2(E_{k}\triangle E_\nu)<\frac{1}{m^2}.
    \end{equation}
    Thus, 
    \begin{equation}
        \H^1\left(\partial \Omega_{\frac{nj(\ell,m)}{k}}\cap (E_{k}\triangle E_\nu)\right)< \frac{n}{2 m}C(\T_1) \; \xrightarrow[m\to \infty] \;0 .
    \end{equation}
    In particular, if we consider the subsequence of triangulations $\T_{k_m}$ with $k_m:=\ell(m)  m n$, we get
    \begin{equation}\label{eq:slicepervanish}
        \lim_{m\to \infty}  \H^1\left(\partial \Omega_{\frac{nj_m}{k_m}}\cap (E_{k_m}\triangle E_\nu)\right)=0,
    \end{equation}
    where $j_m=j(\ell(m),m)\in \{1,...,\ell(m)\}$ of Equation \eqref{eq: good-slice}.

    We construct a new sequence of triangulated sets $\w{E}_{k_m}\in S_{\T_{k_m}}(\Omega)$ as
    \begin{equation}
        \w{E}_{k_m}:=\begin{cases}
            E_{k_m} &\text{ on } \Omega\setminus \Omega_{\frac{nj_m}{k_m}}\\
            A_m &\text{ on } \Omega_{\frac{n j_m}{k_m}},
        \end{cases}
    \end{equation}
    where $A_m\in S_{\T_{k_m}}(\Omega)$ is a set satisfying the two attachment conditions:
    \begin{align}
    \partial A_m\cap \partial \Omega = \partial E_\nu\cap \partial \Omega, \qquad \qquad \partial A_m\cap \partial \Omega_{\frac{nj_m}{m}} = \partial E_\nu\cap \partial \Omega_{\frac{nj_m}{m}},
    %\text{ For all }T\in \T_{k_m}\text{ with }\overline{T}\cap \partial \Omega\neq \emptyset, \quad \big(T\subset A_m \  \iff \  \partial T\cap \partial \Omega \subset \partial E_\nu\cap \partial \Omega\big),\\
    %\text{ For all }T\in \T_{k_m}\text{ with }\overline{T}\cap \partial \Omega_{\frac{nj_m}{k_m}}\neq \emptyset, \quad \big(T\subset A_m \  \iff \  \partial T\cap \partial \Omega_{\frac{nj_m}{k_m}} \subset \partial E_\nu\cap \partial \Omega\big),
    \end{align}
    and such that
    \begin{equation}\label{eq:bdycontrol}
    \partial A_m \subset \bigcup \big\{ \partial T \,\big\vert\, T \in \T_{k_m} \text{ and }T \subset Q = \big[(0,1/k_m)^2 + (i_1,i_2)\big] \text{ with } Q \cap \partial E_\nu \neq \emptyset \big\}.
    \end{equation}
    The last condition enforces the boundary of $A_m$ to be contained in periodic squares of $\T_{k_m}$ that intersect the boundary of the target $E_\nu$. In this way, the maximum possible perimeter inside the squares containing $\partial E_\nu \cap \Omega_{\frac{nj_m}{m}}$ bounds the perimeter of $A_m$ in $\Omega_{\frac{nj_m}{m}}$.
    The assumption that $\T_1$ is diagonally separable guarantees the existence of such an $A_m$. An example of this construction is depicted in Figure \ref{fig:slicing}. Finally, to prove \eqref{eq:modnoaddper}, we notice that
    \begin{equation}
        \Omega \cap (\partial \w{E}_{k_m}\setminus \partial E_{k_m}) \subseteq \left(\Omega_{\frac{n j_m}{k_m}}\cap \partial A_m\right) \cup \left(\partial \Omega_{\frac{nj_m}{k_m}}\cap (E_{k_m}\triangle E_\nu)\right),
    \end{equation}
    and by \eqref{eq:bdycontrol} we know that
    \begin{equation}
        \lim_{m\to \infty} \H^1\left(\Omega_{\frac{n j_m}{k_m}}\cap \partial A_m\right)=0,
    \end{equation}
    which combined with \eqref{eq:slicepervanish} proves \eqref{eq:modnoaddper}.
\end{proof}
\end{lemma}
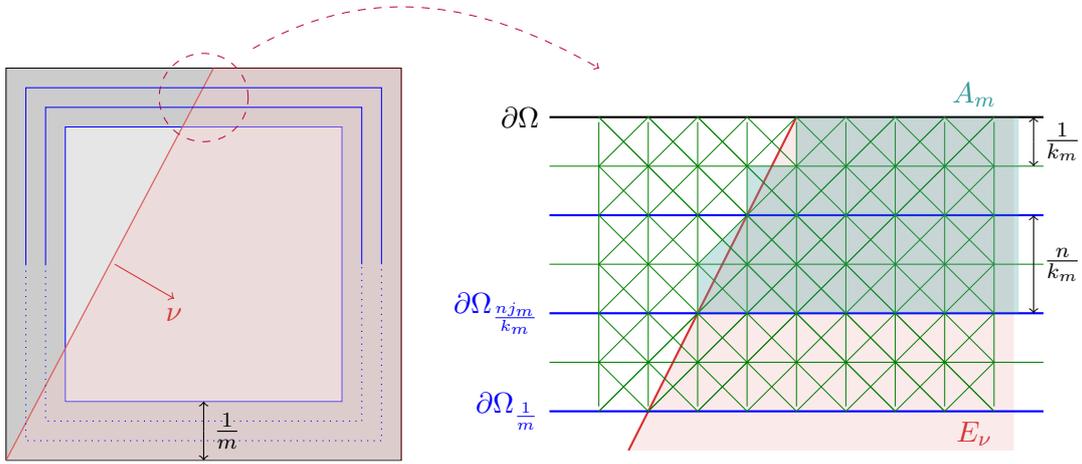
\begin{figure}[!htb]
    \centering
    \begin{tikzpicture}[scale=1.3]
        \filldraw[fill=black!20, line width=0] (-2,-2)--(-2,2)--(2,2)--(2,-2)--(-2,-2);
        \filldraw[fill=black!10, draw=blue, line width=0] (-1.4,-1.4)--(-1.4,1.4)--(1.4,1.4)--(1.4,-1.4)--(-1.4,-1.4);
        \filldraw[fill=red!20, draw=red, opacity=0.4](-2,-2)--(0.1,2)--(2,2)--(2,-2);
        \draw[red!60](-2,-2)--(0.1,2);
        \draw[red!80, ->] (-0.9,0)--(-0.3,-0.35) node[below]{$\nu$};
        \foreach \k in {1.6,1.8}{
            \draw[blue,line width=0](-\k,0)--(-\k,\k)--(\k,\k)--(\k,0);
            \draw[blue, dotted](-\k,0)--(-\k,-\k)--(\k,-\k)--(\k,0);
        }
        \draw[->](0,-1.7)--(0,-1.4);
        \draw[<-](0,-2)--(0,-1.7) node[right]{$\frac{1}{m}$};
        \draw[purple, line width=0.2, dashed] (0,1.7)  circle (0.45cm); 
        \draw[purple,->, line width=0.2, dashed](0.5,2.2) to [in=150,out=30] (4,2);
        \fill[fill=red!20, opacity=0.4] (4.3, -1.9)--(6,1.5)--(8.2,1.5)--(8.2,-1.9);
        \draw[red!80, thick] (4.3, -1.9)--(6,1.5);
        \draw[blue, thick] (3.5,0.5)--(8.5,0.5)
                    (3.5,-.5)--(8.5,-.5)
                    (3.5,-1.5)--(8.5,-1.5);
        \draw[thick] (3.5,1.5)--(8.5,1.5);
        \foreach \k in {-1,0,1}{
            \draw[green, line width=0] (3.5,\k)--(8.5,\k);
            }
        \foreach \k [evaluate={\y=max(1.5-(\k-4),-1.5); \Y=max(1.5-(8-\k),-1.5); \x=max(4,\k-3); \X=min(8,\k+3)}] in {4,4.5,...,8}{
            \draw[green, line width=0] (\k,1.45)--(\k,-1.45);
            \draw[green, line width=0] (\X,\Y)--(\k, 1.5)--(\x,\y);
            \draw[green, line width=0] (\X,-\Y)--(\k, -1.5)--(\x,-\y);
            }
        \fill[teal!70, opacity=0.3] (5,-0.5)--(5,0)--(5.5,0.5)--(5.5,1)--(6,1)--(6,1.5)--(8.25,1.5)--(8.25,-0.5);
        \draw[->](8.4,1.25)--(8.4,1.5);
        \draw[<-](8.4,1)--(8.4,1.25) node[right]{$\frac{1}{k_m}$};
        \draw[->](8.4,0)--(8.4,0.5);
        \draw[<-](8.4,-0.5)--(8.4,0) node[right]{$\frac{n}{k_m}$};
        \draw (3.5,1.5) node[left]{$\partial \Omega$};
        \draw[blue] (3.5,-1.5) node[left]{$\partial \Omega_{\frac{1}{m}}$}
                (3.5,-0.5) node[left]{$\partial \Omega_{\frac{nj_m}{k_m}}$};
        \draw[red!80](7.8,-1.5) node[below]{$E_\nu$};
        \draw[teal!80](7.8,1.5) node[above]{$A_m$};
        \end{tikzpicture}   
    \caption{A sketch of the constructions in the proof of Lemma \ref{lem:discretedegiorgi}. In this case $\nu=(2,-1)/\sqrt{5}$.}  
    \label{fig:slicing}
\end{figure}

We will now describe the anisotropy $\varphi_{\T_1}$ for a specific choice of periodic planar triangulations, the `double diagonal'. The result implies, in particular, that sets with fixed mass which are optimal for $\Per_{\varphi_{\T_1}}$ (i.e.~the Wulff shape) are precisely the octagons which can be realized without error with piecewise constant functions on this type of mesh, as appearing for example in Figure \ref{fig:castle_regular}. The strategy is to directly compute the limit energy of half-planes, for which we can modify the recovery sequence of Proposition \ref{prop:periodicconv} to match the boundary condition of the target set with the help of the previous lemma. 

\begin{proposition}[Octagon]\label{prop:octagon}
    Let $\{\T_k\}_k$ be a sequence of periodic triangulations of $\Omega=(0,1)^2$, such that $\T_1$ is the `double diagonal' triangulation of a square. Then the unit ball $\varphi_{\T_1}^{-1}(1)\subset \R^2$ corresponds to the regular octagon with a vertex at $(1,0)$, and the Wulff shape is the rotation of $\varphi_{\T_1}^{-1}(1)$ by $\frac{\pi}{8}$.
\end{proposition}
\begin{proof}
Given a rational direction $\nu\in S^1$, our strategy is to consider a recovery sequence $E_k \xrightarrow[k \to \infty]{} E_\nu$ coming from the $\Gamma$-convergence of Proposition \ref{prop:periodicconv}, with $E_\nu$ defined as in Equation \ref{eq:Enudef}. By construction, it holds 
\[\lim_{k} \Per_{\T_k}(E_k)=\Per_{\varphi}(E_\nu) = \varphi_{\T_1}(\nu) \H^1(\partial E_\nu \cap \Omega).\]
From Lemma \ref{lem:discretedegiorgi} we can build a modified subsequence $\tilde{E}_{k_m}$ that matches the boundary of $E_\nu$ for which 
\[\lim_{k} \Per_{\T_k}(E_k) = \lim_{m} \Per_{\T_{k_m}}(\tilde{E}_{k_m}).\]
Then, given a sequence of perimeter minimizers with fixed boundary constraint
\begin{equation}\label{eq: discrete geodesics}
    G_{k_m}\in \argmin_{\substack{A\in S_{\T_{k_m}}(\Omega),\\
    \partial A\cap \partial \Omega=\partial E_\nu\cap \partial \Omega}} \Per_{\T_{k_m}}(A),
\end{equation}
%discrete geodesic $G_{k_m}$ subject to the same boundary conditions that $\tilde{E}_{k_m}$ satisfies,
we obtain
\[\Per_{\T_{k_m}}(G_{k_m}) \leq \Per_{\T_{k_m}}(\tilde{E}_{k_m}).\]
If we have that $G_{k_m} \xrightarrow[m \to \infty]{} E_\nu$ in $L^1$, then the $\Gamma$-$\liminf$ inequality implies
\[\Per_{\varphi}(E_\nu) \leq \liminf_m \Per_{\T_{k_m}}(G_{k_m}),\]
and ultimately
\begin{equation}\label{eq: explicit density}
    \varphi_{\T_1}(\nu) \H^1(\partial E_\nu \cap \Omega)= \liminf_m \Per_{\T_{k_m}}(G_{k_m}).
\end{equation}
Therefore, our goal now is to describe, for this particular triangulation, the minimal perimeter $\Per_{\T_{k_m}}(G_{k_m})$ and specific (since they are not unique) minimizers $G_{k_m}$ with 
\begin{equation}\label{eq:massconvgeod}\L^2(G_{k_m} \triangle E_\nu) \to 0\end{equation}

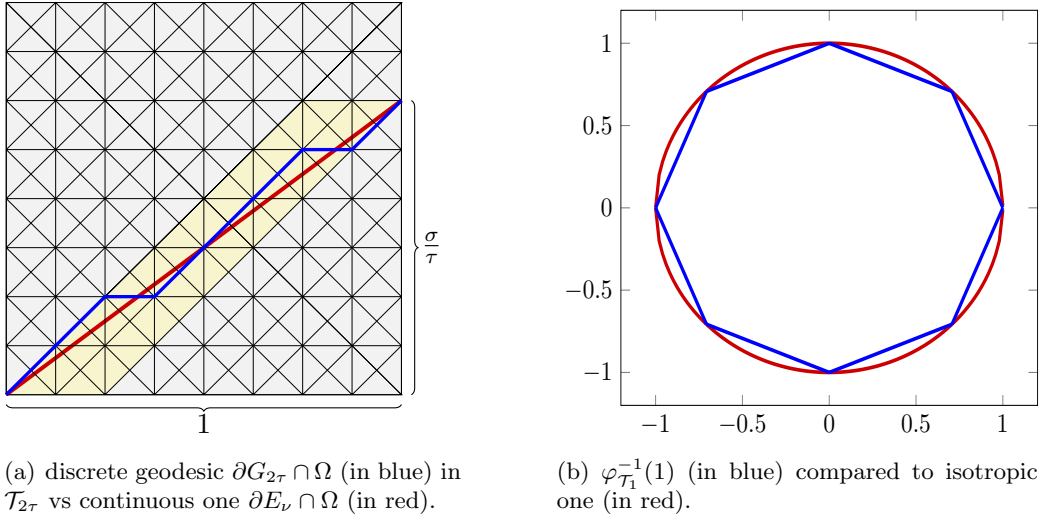
\begin{figure}[!htb]
    \centering
    \subfigure[discrete geodesic $\partial G_{2\tau}\cap \Omega$ (in blue) in $\T_{2\tau}$ vs continuous one $\partial E_\nu\cap \Omega$ (in red).]{\begin{tikzpicture}[scale=1.3]
        \filldraw[fill=lightgray!20, line width=0] (-2,-2)--(-2,2)--(2,2)--(2,-2)--(-2,-2);
        \fill[fill=yellow!50, opacity=0.4](-2,-2)--(1,1)--(2,1)--(-1,-2);
        \draw[red, line width=1.5] (-2,-2)--(2,1);
        \foreach \k in {-2,-1.5,...,2}{
            \draw[ultra thin] (-2,\k)--(2,\k)
                (\k,-2)--(\k,2)
                (\k,-2)--(2,-\k)
                (-2,\k)--(-\k,2)
                (\k,-2)--(-2,\k)
                (2,\k)--(\k,2);
            }
        \draw[blue, line width=1.25](-2,-2)--(-1,-1)--(-0.5,-1)--(1,0.5)--(1.5,0.5)--(2,1);   
        \draw [decorate, decoration = {brace}] (2.1,1) --  (2.1,-2);
        \draw [decorate, decoration = {brace}] (2,-2.1) --  (-2,-2.1);
        \draw (2.1,-.5) node[right]{$\frac{\sigma}{\tau}$};
        \draw (0,-2.1) node[below]{$1$};
        \end{tikzpicture}   
    \label{fig:discrete-geodesics}}
    \hspace{1cm}
    \subfigure[$\varphi_{\T_1}^{-1} (1)$ (in blue) compared to isotropic one (in red).]{\begin{tikzpicture}[scale=0.8, yscale=1.15]
      \begin{axis}
      %\fill[yellow!50, opacity=0.5](0,0)--(0,-1.2)--(0.9,-0.9);
      \addplot [domain=-1:1, samples=100, line width=1.5pt, color=red] {(1-x^2)^(1/2)};
      \addplot [domain=-1:1, samples=100, line width=1.5pt, color=red] {-(1-x^2)^(1/2)};
      \addplot [domain=1:2^(1/2)/2, samples=100, line width=1.5pt, color=blue] {(2^(1/2)+1)*(-x + 1)};
      \addplot [domain=2^(1/2)/2:0, samples=100, line width=1.5pt, color=blue] {-(2^(1/2)-1)*x + 1};
      \addplot [domain=0:-2^(1/2)/2, samples=100, line width=1.5pt, color=blue] {(2^(1/2)-1)*x + 1};
      \addplot [domain=-2^(1/2)/2:-1, samples=100, line width=1.5pt, color=blue] {(2^(1/2)+1)*(x + 1)};
      \addplot [domain=-1:-2^(1/2)/2, samples=100, line width=1.5pt, color=blue] {(2^(1/2)+1)*(-x - 1)};
      \addplot [domain=-2^(1/2)/2:0, samples=100, line width=1.5pt, color=blue] {-(2^(1/2)-1)*x - 1};
      \addplot [domain=0:2^(1/2)/2, samples=100, line width=1.5pt, color=blue] {(2^(1/2)-1)*x - 1};
      \addplot [domain=2^(1/2)/2:1, samples=100, line width=1.5pt, color=blue] {(2^(1/2)+1)*(x - 1)};
      \end{axis}
    \end{tikzpicture}\label{fig:octagon}}
    \caption{Description of the anisotropy $\varphi_{\T_1}$ of the `double diagonal' triangulation $\T_1$.}  
\end{figure}

Because the initial triangulation $\T_1$ is invariant under rotations by multiples of $\frac{\pi}{4}$, we can assume $\nu=\frac{1}{\sqrt{\sigma^2+\tau^2}}(\sigma,-\tau)\in S^1$ for some coprime integers $\sigma,\tau\in \N$ with $0\leq \sigma\leq \tau$, as in Figure \ref{fig:discrete-geodesics}. 
Note that $E_\nu$ is a triangle with vertices at $(0,0), (1, 0)$ and $\left(1, \frac{\sigma}{\tau}\right)$, which, by the construction of the subsequence in Lemma \ref{lem:discretedegiorgi}, are all nodes of periodic squares in $\T_{k_m}$ for all $m$.  

Consider the minimization problem \eqref{eq: discrete geodesics} defining $G_{k_m}$. We claim that its minimal value can be found by minimizing over paths made up of edges none of which is contained in $\partial \Omega$. To see this, consider an arbitrary competitor $A$ for \eqref{eq: discrete geodesics}. Then, since $\partial A\cap \partial \Omega=\partial E_\nu\cap \partial \Omega$ we have that all of the triangles  of $\T_{k_m}$ with an edge contained in either $[0,1]\times \{0\}$ or $\{1\}\times[0,\sigma/\tau]$ must be included in $A$. Moreover, for each pair $T_1,T_2$ of those triangles such that $\overline{T_1} \cap \overline{T_2} = \{v_0\}$ for some $v_0 \in \partial \Omega$, we can assume that the two other triangles having $v_0$ as a vertex are also included in $A$, since adding them to it cannot increase the perimeter. Therefore, we have obtained that $\partial A\cap \partial \Omega$ must be part of the boundary of an indecomposable component of $A$, say $A_0$. Seeing $A_0$ as a positively oriented simplicial chain, we can consider the induced orientation for $\partial A_0$, in which each vertex of $\T_{k_m}$ must appear only in an even number of edges with by virtue of being a boundary, and in fact only twice since otherwise $A_0$ would not be indecomposable. Alternatively, one can use general results for two-dimensional sets of finite perimeter expressing their essential boundaries as unions of Jordan curves, see \cite[Thm.~7, Cor.~1]{AmbCasMasMor01} or \cite[Rem.~II]{CafRiv76}.

With this in mind, we only need to consider injective continuous paths $\gamma:[0,1]\to \overline{\Omega}$ connecting $(0,0)$ with $\left(1,\frac{\sigma}{\tau}\right)$ whose image is made up of (non-boundary) edges of triangles in $\T_{k_m}$. Furthermore, since we are only interested in $\H^1(\gamma([0,1]))$ and for injective curves this quantity equals (see \cite[Thm.~2.6.2]{BurBurIva01}) the length of $\gamma$ defined as a supremum of sums of lengths of polygons with vertices in $\gamma([0,1])$ induced by partitions of $[0,1]$, we can assume $\gamma$ to be piecewise linear with switching only at the preimages of vertices of triangles in $\T_{k_m}$. Let
\begin{equation}
    D^{0,1}_\gamma:=\{t_0,...,t_{n+1}\}\subset [0,1]
\end{equation}
be the ordered set of points where $\gamma$ is non-differentiable. Then, $\gamma(D^{0,1}_\gamma)$ must be vertices of the triangulation $\T_{k_m}$. It can be proven that if $\gamma$ is length-minimizing among such curves, then $\gamma(t_{j+1})-\gamma(t_{j})$ must be either (up to positive rescaling) the vector $\nu_1=(1,0)$ or $\nu_2=(1,1)$, for all $t_j\in D_{\gamma}^{0,1}$, which in particular implies that $\gamma(D_\gamma^{0,1})$ consists of nodes of periodic squares inside the polygon with vertices 
\begin{equation}
    \left\{(0,0), \left(\frac{\sigma}{\tau},\frac{\sigma}{\tau}\right), \left(1,\frac{\sigma}{\tau}\right), \left(1-\frac{\sigma}{\tau},0\right)\right\},
\end{equation}
illustrated in yellow in Figure \ref{fig:discrete-geodesics}. Moreover, the length of such an optimal path can be computed from the amount of edges directed by $\nu_1$ and $\nu_2$ respectively, resulting in 
\begin{equation}\label{eq: min perimeter}
    \Per_{\T_{k_m}}(G_{k_m}) = \frac{1}{\tau}\big( \sigma \sqrt{2} + (\tau-\sigma)\big) = 1 + (\sqrt{2}-1)\frac{\sigma}{\tau}.
\end{equation}
In the `double-diagonal' geometry under consideration these facts are quite intuitive, so we relegate a detailed proof to Appendix \ref{app:discretegeodesics}. 

The only thing left to justify is that, as $m \to \infty$, one can choose optimal paths of the type we have just described, in such a way that \eqref{eq:massconvgeod} is satisfied. One adequate choice is to follow a diagonal path from the origin with $t \mapsto \frac{1}{k_m} t \nu_2$ until the first $t^\ast \in \N$ satisfying $t^\ast \nu_2 + \frac{1}{k_m}\nu_1 \notin \mathring{E}_\nu$, concatenating a horizontal segment $t \to \frac{1}{k_m} \nu_2$ for $t\in[t^\ast,t^\ast+1]$, iterating this procedure until hitting the point $(1,\sigma/\tau)$, and rescaling the parametrization so that its domain is $[0,1]$. By construction, the image of a path $\gamma$ constructed in this fashion satisfies $d(\gamma(t), \partial E_\nu) \leq 1/k_m$ for all $t$, implying in particular \eqref{eq:massconvgeod}.

Finally, by combining Equation \eqref{eq: explicit density} and \eqref{eq: min perimeter}, we obtain that, for a rational direction $\nu =\frac{1}{\sqrt{\sigma^2+\tau^2}}(\sigma,-\tau)\in S^1$ with $0\leq \sigma\leq \tau$, 
\begin{equation}
    \varphi_{\T_1}(\nu)=(\sqrt{2}-1, -1)\cdot \nu.
\end{equation}
By symmetry, we can extend the previous discussion to other rational directions $\nu\in S^1$, finding a description of $\varphi_{\T_1}$ on a dense subset of $S^1$. Then again, we can extend it to the whole $S^1$ by continuity and to $\R^2\setminus\{0\}$ by homogeneity $\varphi_{\T_1}(v)=|v|\varphi_{\T_1}\left(\frac{v}{|v|}\right)$. The $1-$levelset of the resulting map $\varphi_{\T_1}$ is depicted in Figure \ref{fig:octagon}, and it is precisely the expected regular octagon. The claim about the corresponding Wulff shape follows by the Wulff theorem \cite[Thm.~20.8]{Mag12} and convex duality.
\end{proof}

\section{Numerical examples}\label{sec:computations}

Finally, we provide two numerical examples which illustrate the findings of the present work, in particular the practical efficiency  of Algorithm~\ref{alg:dinkelbachgcg} as well as the asymptotic behavior of the total variation as explained in Proposition~\ref{prop:octagon}. For this purpose, let~$\mathcal{T}$ be a triangulation of~$\Omega=(-1,1)^d$,~$d=2,3$, either given by the ``double-diagonal'' triangulation from Figure~\ref{def:periodictriang} or pseudo-random ones obtained from a mesh generator. Associated to~$\mathcal{T}$, we then consider discretized optimal control problems of the form   
\begin{align*}
    \min_{u_h \in P_0(\mathcal{T}) } J_{\mathcal{T}}(u_h) \coloneqq \left\lbrack \frac{1}{2} \|y_h-y_o\|^2_{L^2(\Omega)}+ \alpha \TV(u_h,\Omega) \right\rbrack
\end{align*}
where~$y_h=K_h u_h \in P_1(\mathcal{T}) \cap H^1_0(\Omega)$ is the unique solution of
\begin{align} \label{eq:discstatenum}
   \text{find}~y \in P_1(\T) \cap H^1_0 (\Omega) \quad\text{ s.t. }\ \int_\Omega \nabla y \cdot \nabla \theta_h + c y \theta_h ~\mathrm{d}x= \int_\Omega u_h \theta_h~\mathrm{d}x
\end{align}
for all~$\theta_h \in P_1(\mathcal{T}) \cap H^1_0(\Omega)$, and~$\alpha >0$,~$c \geq 0$ and~$y_o \in L^2(\Omega) $ are example specific parameters.
\subsection{Implementation of Algorithm~\ref{alg:dinkelbachgcg}} \label{subsec:implementation}
We give a brief overview on the implementation details concerning Algorithm~\ref{alg:dinkelbachgcg}. First, using standard adjoint calculus, we verify that the dual variable~$p_k = (1/\alpha) K^*_h (y_o-K_hu_k) \in V' \simeq P_0(\mathcal{T}) $ is given by~$p_k = (1/\alpha) \Pi_{0}(z_k)$ where~$\Pi_0 \colon P_1(\mathcal{T}) \to P_0(\mathcal{T}) $ denotes the orthogonal projection onto~$P_0(\mathcal{T})$ and~$z_k \in P_1(\mathcal{T}) \cap H^1_0(\Omega)$ is the unique solution of the adjoint problem 
 \begin{align*}
   \text{find}~z \in P_1(\T) \cap H^1_0 (\Omega) \quad\text{ s.t. }\ \int_\Omega \nabla z \cdot \nabla \theta_h + c z\theta_h ~\mathrm{d}x= \int_\Omega  \theta_h~\mathrm{d}x 
\end{align*}
for all~$\theta_h \in P_1(\mathcal{T}) \cap H^1_0(\Omega)$. Second, alongside the current active set~$\mathcal{A}_k= \{u^j_k\}^{N_k}_{j=1}$ and iterate~$u_k=\sum^{N_k}_{j=1}\gamma^k_j u^j_k+ c^k \1_\Omega$, we also update the associated set of states~$\mathcal{S}_k\coloneqq \{K_h u^j_k\}^{N_k}_{j=1} \cup \{K \1_\Omega\}$, accordingly. As a consequence, and due to the linearity of~$K_h$, every iteration of Algorithm~\ref{alg:dinkelbachgcg} requires two PDE-solves, one for the computation of the state associated to the new characteristic function as well as one solve of the adjoint equation to get~$p_k$. The solution of the~$\ell_1$-regularized subproblems does not require further PDE-solves since a matrix representation of~$K \mathcal{U}_{\mathcal{A}_k}$ is available from~$\mathcal{S}_k$. 
The algorithm is implemented in Python, using FEniCS 2019 for the finite elements computations, the PyMaxflow library (\url{https://github.com/pmneila/PyMaxflow}) wrapping the Boykov-Kolmogorov implementation of maxflow \cite{BoKo2004} for the minimal cuts in the Dinkelbach iteration, as well as a semismooth Newton method based on the normal map approach for the finite dimensional subproblems of the form~\eqref{def:subprobPDAP}, see e.g.~\cite{Milzarek}. The latter is warmstarted using~$\gamma^k$ and~$c^k$ to construct a good starting point. In all examples, the progress of the algorithm is monitored by the upper bounds~$\zeta_k$, see Lemma~\ref{lem:upperbound}, obtained as a by-product of Algorithm~\ref{alg:dinkelbach}. We stop Algorithm~\ref{alg:dinkelbachgcg} in iteration~$\bar{k} \in \N$ if~$\zeta_{\Bar{k}} \leq 10^{-10}$ while each application of the semismooth Newton method is terminated once the corresponding residual is below~$10^{-14}$. The unstructured meshes we use are generated using CGAL through the mshr component of FEniCS. For abbreviation and with some abuse of notation, we set~$\bar{u}=u_{\Bar{k}}$ in the following. Note that these termination criteria imply 
\begin{align*}
    J_\mathcal{T} (\Bar{u})- \min_{u_h \in P_0(\mathcal{T})} J_\mathcal{T} (u_h) \leq 10^{-10}
\end{align*}
in view of Lemma~\ref{lem:upperbound}.

In order to assess the convergence behavior of Algorithm~\ref{alg:dinkelbachgcg}, we compute the suboptimality of~$u_k$ in terms of its objective functional value residual, i.e.,
\begin{align*}
    J_{\mathcal{T}}(u_k)-J_{\mathcal{T}}(\Bar{u}) \approx J_{\mathcal{T}}(u_k)-\min_{u_h \in P_0(\mathcal{T})}J_{\mathcal{T}}(u_h).
\end{align*}
All calculations were  performed on a 2021 Macbook Pro with 10-core M1 Max CPU. Python code for our implementation and configurations to reproduce some of the following examples are available at \url{https://doi.org/10.5281/zenodo.10048384}.

\subsection{Two spheres}\label{subsec:spheres}
As a first simple example, we choose~$c=0.5$ and~$y_o \in P_0(\mathcal{T})$ as the unique solution of~\eqref{eq:discstatenum} for the piecewise constant $L^2$ projection~$u=\Pi_0 (u_0)$ of   
\begin{align*}
    u_0=\1_{D_1}+2\1_{D_2}-\1_\Omega
\end{align*}
where $D_1$ and $D_2$ are balls centered at $1/3 \,\mathbf{1}$ and $-1/3 \,\mathbf{1}$ with radius $\sqrt{0.15}$ as well $\sqrt{0.1}$, respectively. Here,~$\mathbf{1}\in \R^d$ denotes the all-ones vector.

\begin{figure}[ht]
    \centering
    \subfigure[$d=2$, triangulation of $\approx 5\cdot 10^5$ faces.]{\includegraphics[width=0.49\textwidth]{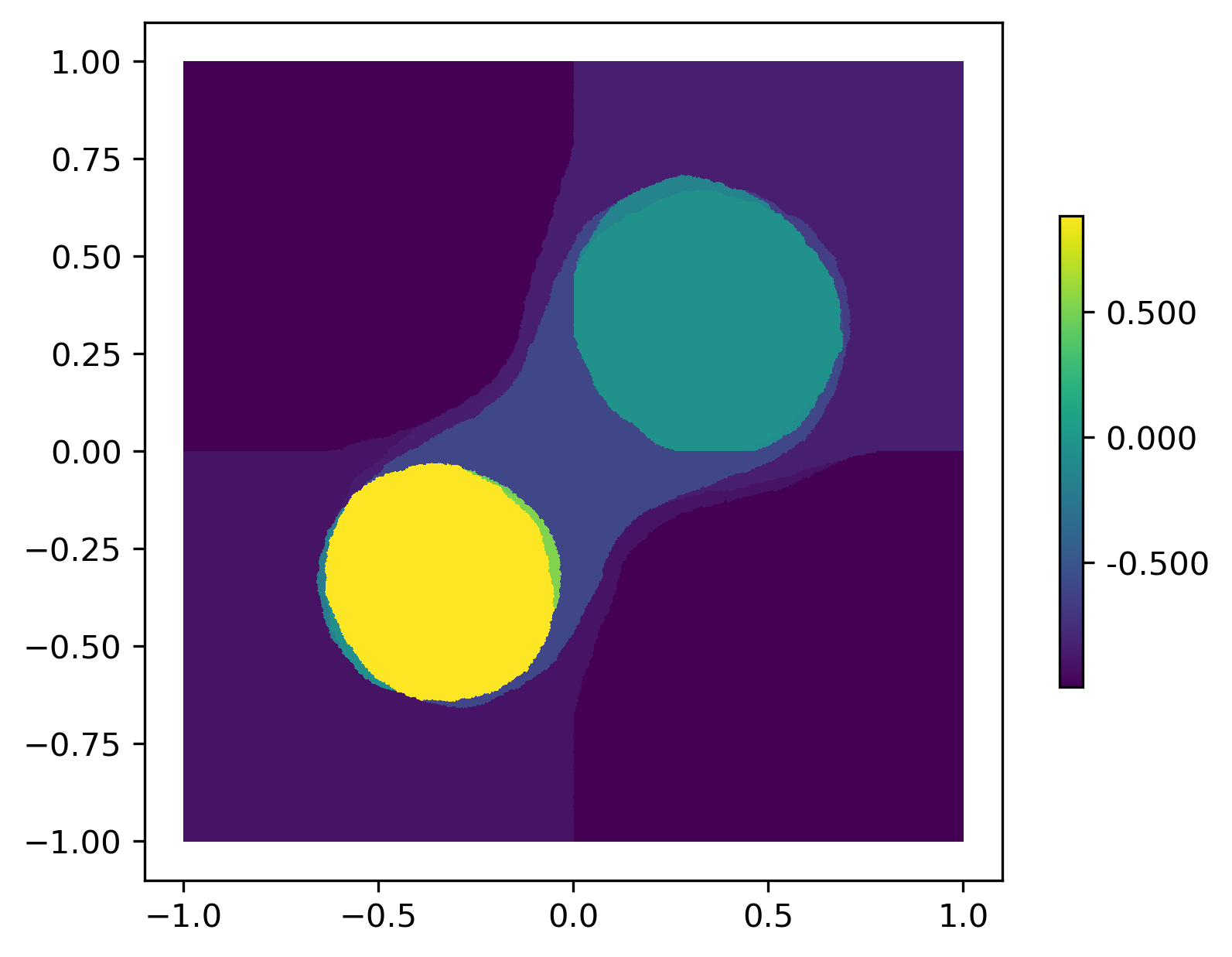}\label{fig;circles}}
    \subfigure[$d=3$, triangulation of $\approx 8.5\cdot 10^5$ cells.]{\includegraphics[width=0.49\textwidth]{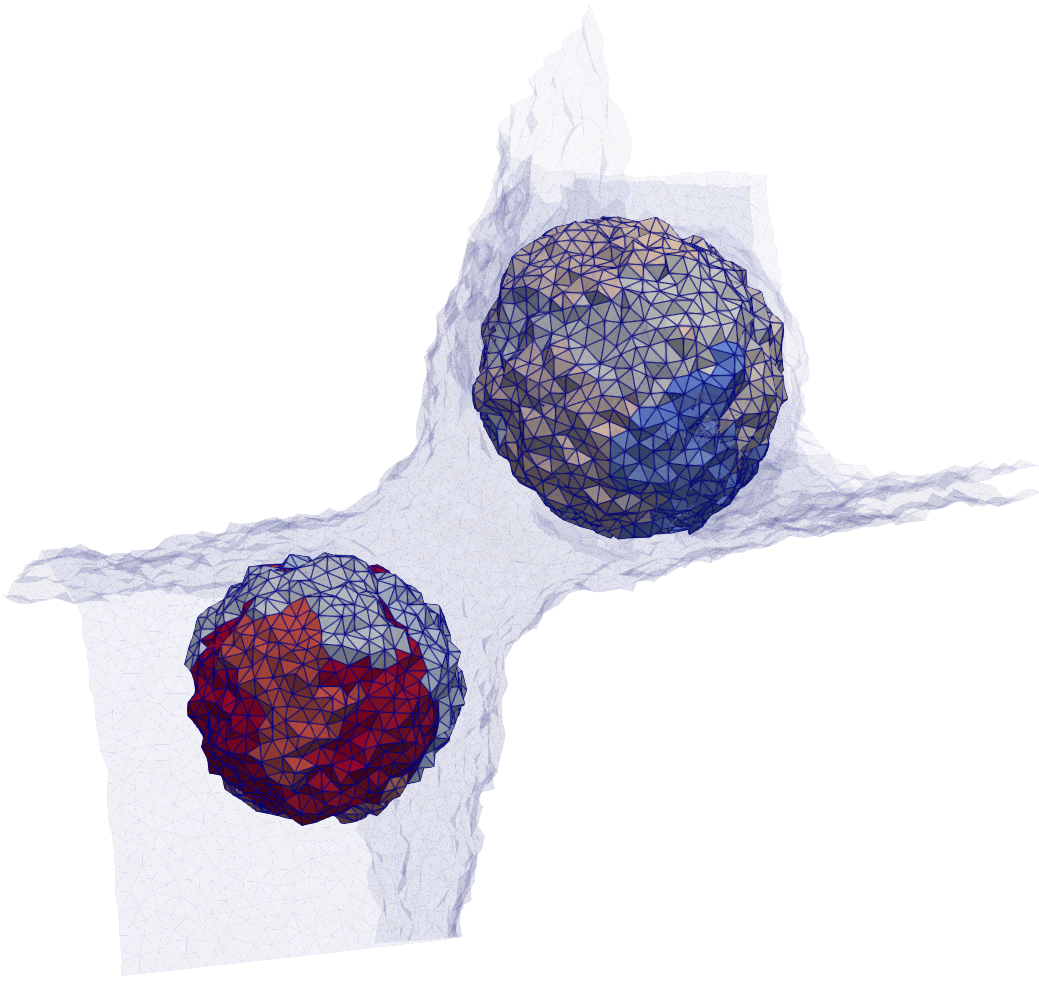}\label{fig:2spheres}}
    \caption{Output $\bar{u}$ of Algorithm \ref{alg:dinkelbachgcg} for $d=2,3$ on pseudo-random triangulations.}
    \label{fig:spheres}
\end{figure}

The regularization parameter is set to $\alpha=10^{-5}$ for both~$d=2$ and~$d=3$. The obtained results from Algorithm~\ref{alg:dinkelbachgcg} for~$d=2,3$ on pseudo-random triangulations are depicted in Figure~\ref{fig:spheres}. Note that the computed approximations contain level sets which are not present in $u_0$. Indeed, even for low values of $\alpha$, a loss of contrast is expected, so the input $u_0$ is never an exact minimizer of $J$; see \cite{CasNovPoe16} for a detailed analysis of this phenomenon in the planar denoising case with $K=\Id$. The presence of the catenoids in the outcomes of the algorithm reduces the value of the total variation while keeping the fidelity term close to zero. In particular, for the present example, we have 
\begin{equation}
    J_{\mathcal{T}}(\Pi_0 (u_0))=\alpha \TV(\Pi_0 (u_0),\Omega)\approx\begin{cases}
        8.29\cdot 10^{-5} &\; d=2\\
        7.02\cdot 10^{-5} &\; d=3
    \end{cases}
    \quad \text{and} \quad J_{\mathcal{T}}(u_{\Bar{k}}) \approx\begin{cases}
        6.15\cdot 10^{-5} &\; d=2\\
        4.87 \cdot 10^{-5} &\; d=3,
        \end{cases}
\end{equation}
supporting this observation.
% In 2 and 3 dimensions and for pseudo-random triangulations, Algorithm \ref{alg:dinkelbachgcg} converges to the controls displayed in Figure $\ref{fig:spheres}$. The energy values of the outcomes are around $6.15e{-5}$ and $4.87e{-5}$, respectively.

% Thus, at least in two dimensions, it is already clear that the toy control is not a minimizer of (the discretized) $J$. In general, minimizers are expected to contain level sets not present in $u_0$, and even for low values of $\alpha$ loss of contrast is expected, so the input $u_0$ is never an exact minimizer of $J$; see \cite{CasNovPoe16} for a detailed analysis of this geometry in the planar denoising case with $K=\Id$. The presence of catenoids in the outcomes of the algorithm reduces the value of the total variation while keeping the fidelity term close to zero.

\subsection{The castle} \label{subsec:castle}
As a second, more, challenging setting, set~$c=0$ and~$y_o=\1_{(-1/2,1/2)^d}$, see also~\cite{ClaKun11,HafMan22,NatWac22} for the case of~$d=2$. In the absence of regularization, to perfectly match this observation one would need to have as a control the distributional Laplacian~$\Delta y_o$, a distribution of order $2$ which in particular does not belong to $\BV(\Omega)$. For this reason,
%Conceptually, the resulting minimization problem can be interpreted as a variational regularization approach for the denoising of the distributional Laplacian of~$y_o$. Since this problem is severly ill-posed and in particular~$\Delta y_o \not \in \BV(\Omega)$
the solution~$\Bar{u}$ of the regularized problem is expected to be more complex than the solution of the first example. The regularization parameter is set to~$\alpha=10^{-4}$. 

\begin{figure}[!htb]
    \centering
    \subfigure[pseudo-random triangulation]{\includegraphics[width=0.49\textwidth]{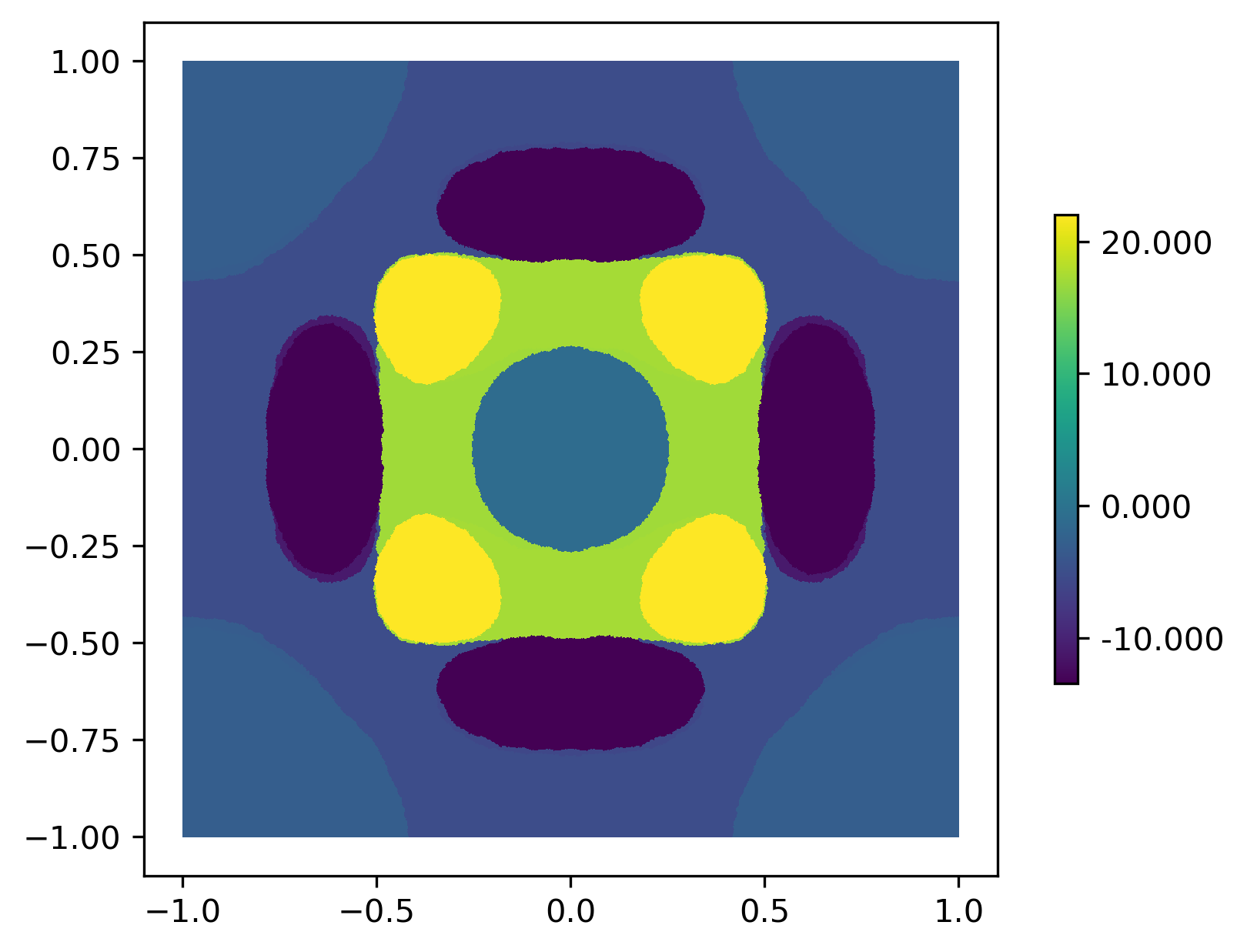}\label{fig:castle}}
    \subfigure['double diagonal' triangulation]{\includegraphics[width=0.49\textwidth]{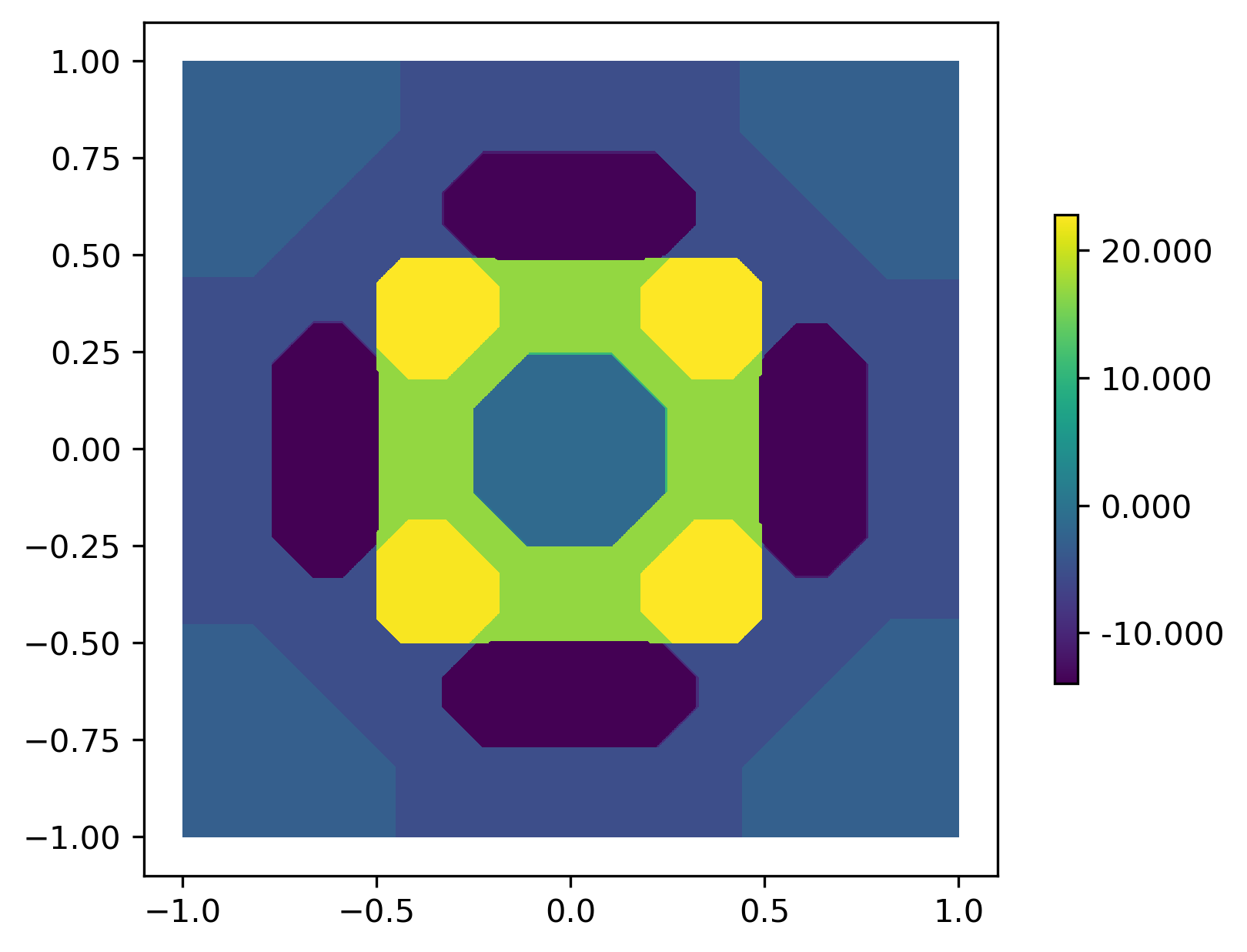}\label{fig:castle_regular}}
    \caption{Top view of optimal control numerically computed through Algorithm \ref{alg:dinkelbachgcg} from different triangulations of $\approx 5\cdot 10^5$ faces, $d=2$ and $\alpha=10^{-4}$.}  
\end{figure}

For~$d=2$ and a pseudo-random triangulation $\T$ of $\Omega$, the output of Algorithm \ref{alg:dinkelbachgcg} is depicted in Figure \ref{fig:castle}. This result is consistent with \cite{HafMan22} where the authors computed the minimizer of problem \eqref{def:BVprob} for $V=P_1(\T)$. When considering a regular 'double diagonal' triangulation $\T$ instead, the algorithm's outcome, see Figure~\ref{fig:castle_regular}, still shares some similarities to Figure \ref{fig:castle} but clearly reflects the anisotropic behavior of the total variation in this setting, as we previously discussed in Proposition \ref{prop:octagon}. Finally, Figure~\ref{fig:hypercastle} depicts the solution for~$d=3$ and a pseudo-random triangulation.
% The result of the algorithm in the three-dimensional context, displayed in Figure \ref{fig:hypercastle}, bears resemblance to the two-dimensional case of Figure \ref{fig:castle}. However, due to the higher dimension, the algorithm demonstrates an extended runtime when applied to a sufficiently refined triangulation. Consequently, the use of a reduced number of triangles became necessary, despite the trade-off with precision in the outcome.

\begin{figure}[!htb]
    \centering
    \subfigure{\includegraphics[width=0.49\textwidth]{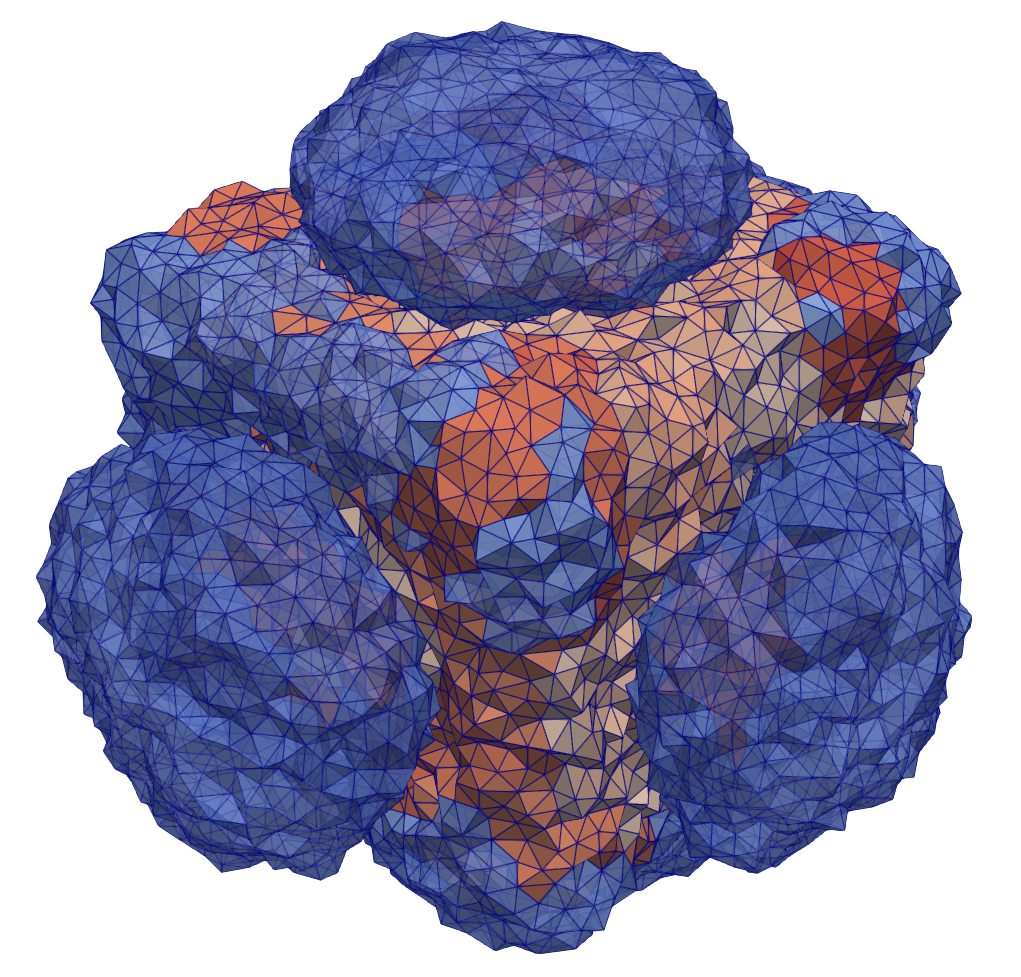}}
    \subfigure{\includegraphics[width=0.49\textwidth]{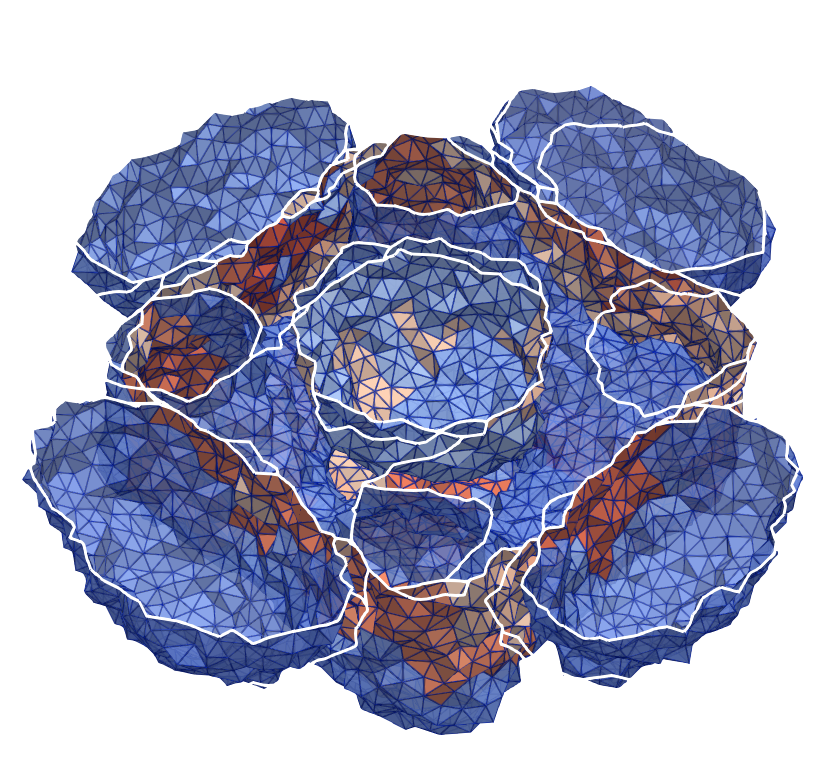}}
    \caption{The plot shows where the variation of $u_k$ between cells happens. The color gradient, ranging from blue to red, is an indicator of the magnitude of the variation. The setting is a pseudo-random triangulation of $\approx 8.5\cdot 10^5$ cells, $d=3$, and $\alpha=10^{-4}$. }
    \label{fig:hypercastle}
\end{figure}

\subsubsection{Practical performance \& discussion} % against \texorpdfstring{\cite{HafMan22}}{a path-following method}
In order to investigate the practical performance of Algorithm~\ref{alg:dinkelbachgcg}, we plot the upper bounds~$\zeta_k$ as well as the residuals in terms of the objective functional value associated to~$u_k$ for both examples and~$d=2$. We refer to Figure~\ref{fig:convergence_plots} where the convergence criterion was met after~$\bar{k}=66$ and~$\Bar{k}=95$, respectively. Several observations can be made: First, as claimed by Lemma~\ref{lem:upperbound},~$\zeta_k$ provides an upper bound on the residual. Furthermore, for the present examples, it solely deviates by a maximal factor of~$10^2$ from the latter. Thus, it represents an easily accessible and reliable convergence criterion. Second, as we can see from both plots, Algorithm~\ref{alg:dinkelbachgcg} admits a global linear rate of convergence rather than a sublinear one as predicted by Proposition~\ref{prop:convergence}. While such improved convergence results can also be deduced from the abstract results in~\cite{BreCarFanWal23}, their application to our setting is nontrivial and beyond the scope of the current work. Finally, we point out that, in both examples, we observed that there holds~$\TV(u_k,\Omega)= \sum^{N_k}_{j=1} \gamma^k_j$ throughout every iteration. As a consequence, the residual exhibits a monotone convergence behavior.

\begin{figure}[!htb]
    \centering
    \subfigure[Indicator and residual for the example of Fig. \ref{fig:2spheres}.]{\includegraphics[width=0.49\textwidth]{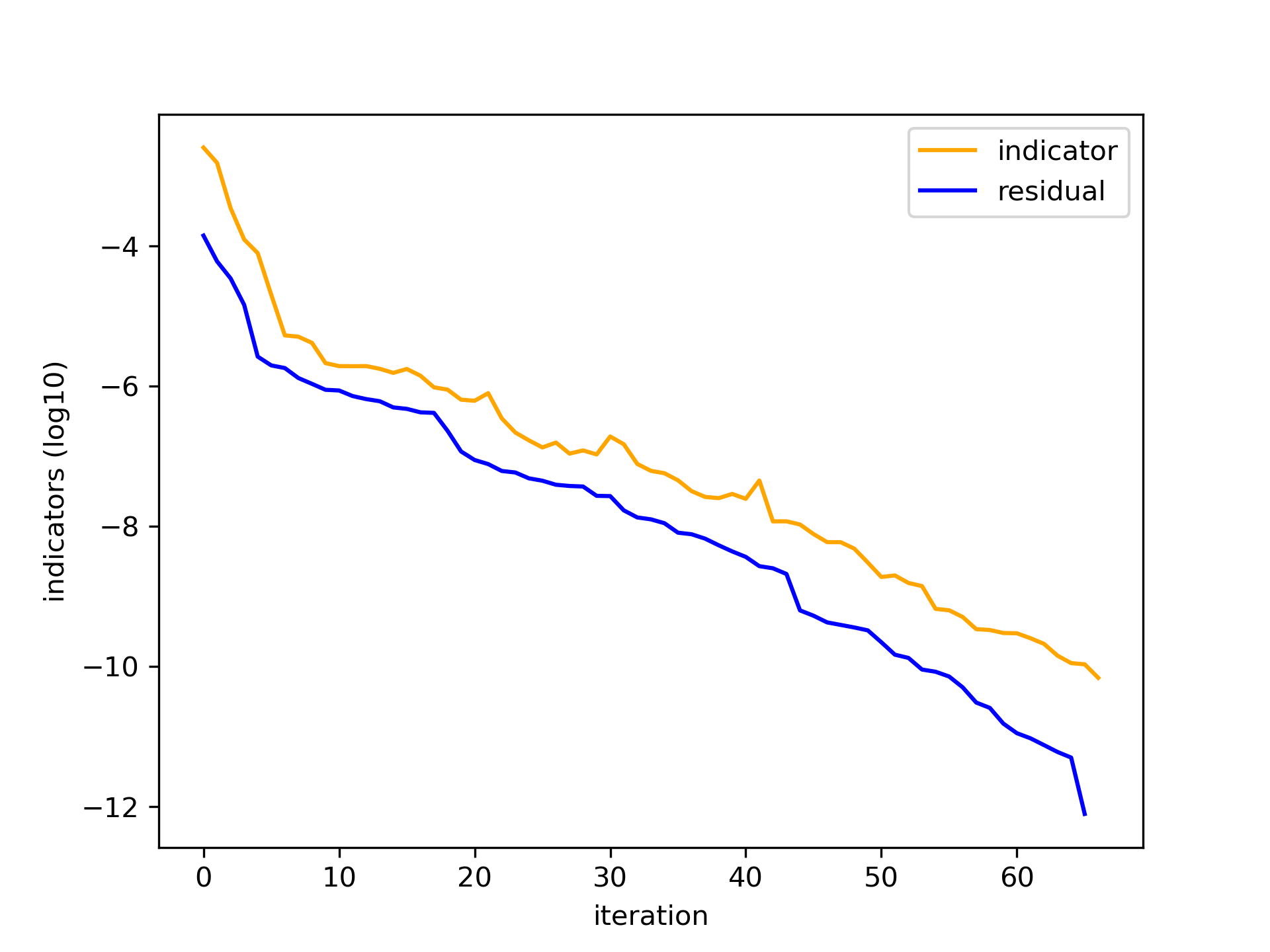}}
    \subfigure[Indicator and residual for the example of Fig. \ref{fig:castle}.]{\includegraphics[width=0.49\textwidth]{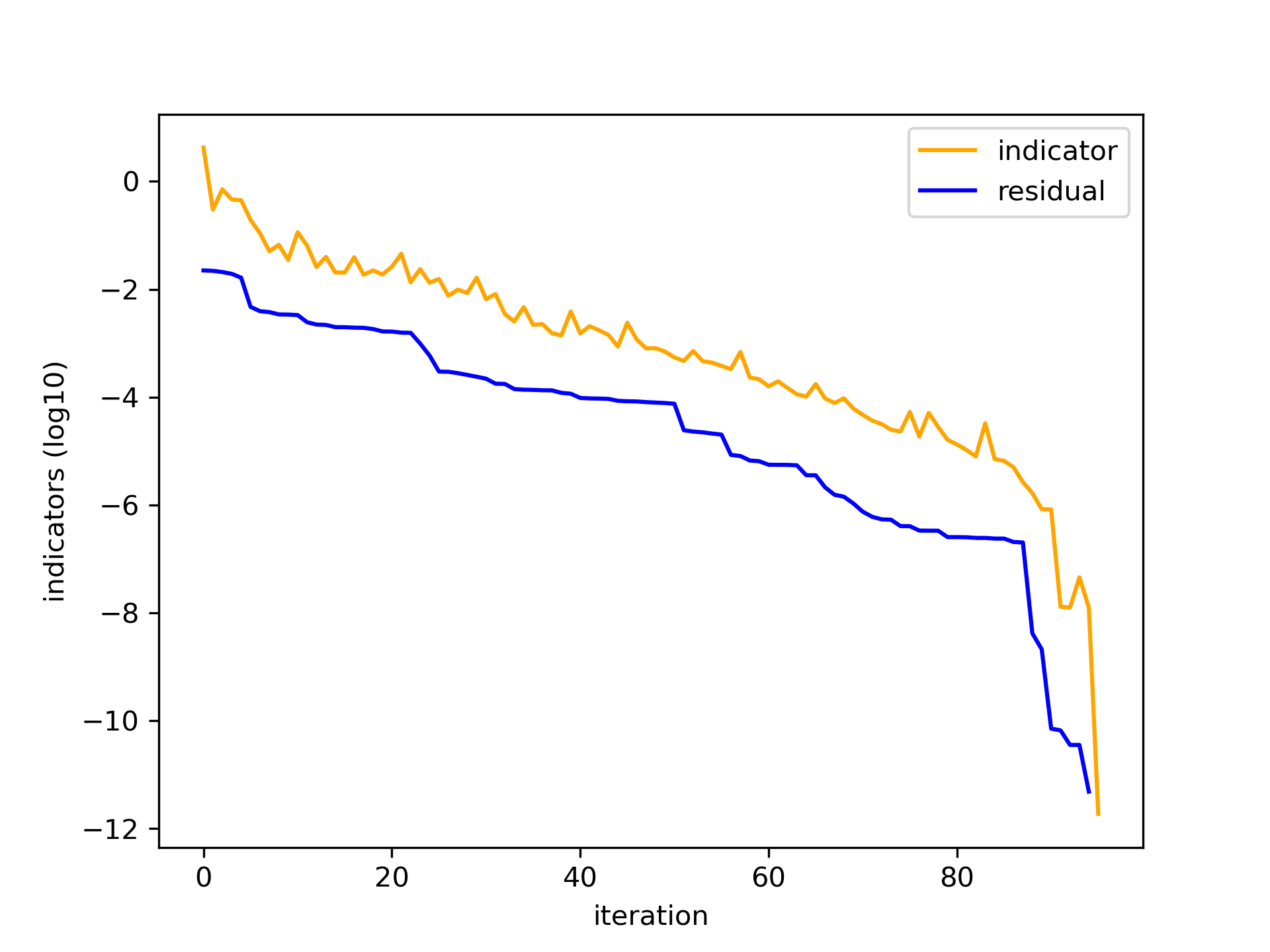}}
    \caption{Convergence plots for 3D spheres (left) and 2D castle (right)}
    \label{fig:convergence_plots}
\end{figure}

We also want to point out that Algorithm \ref{alg:dinkelbachgcg} exhibits remarkable computational speed. We largely attribute this to the fast convergence of the Dinkelbach method, usually within $4/5$ steps, as well as the efficiency of the employed minimal cut implementation \cite{BoKo2004}. Moreover, due to the mentioned warm starting strategy, the employed semismooth Newton method also converges in few steps. Coupled with the observed improved convergence behavior of Algorithm~\ref{alg:dinkelbachgcg}, this results in noteworthy runtime performance. In order to illustrate this observation, we again consider the second example on pseudo-random triangulations and plot the energy~$J_\mathcal{T} (u_k)$ against the computational time. For comparison, we also consider the path-following based method from~\cite{HafMan22}. An implementation of the latter using~$P_1(\mathcal{T})$ as discretized control space can be found at the source files of~\url{https://arxiv.org/abs/2010.11628}. We particularly point out that Algorithm~\ref{alg:dinkelbachgcg} provides near optimal iterates before the first subproblem in the path-following method is solved, see Figure \ref{fig:comparison}. However, we also see that, while Algorithm~\ref{alg:dinkelbachgcg} converges rapidly in the first iterations, its per-iteration decrease in the energy quickly deteriorates. For a closer investigation of this observation, we plot the relative change of the iterate~$u_k$ w.r.t to the~$L^1$-norm for both examples in Figure~\ref{fig:relchangeplot}.

\begin{figure}[ht]
    \centering
    \subfigure[triangulation of $\approx 1.25\cdot 10^5$ faces]{\includegraphics[width=0.49\textwidth]{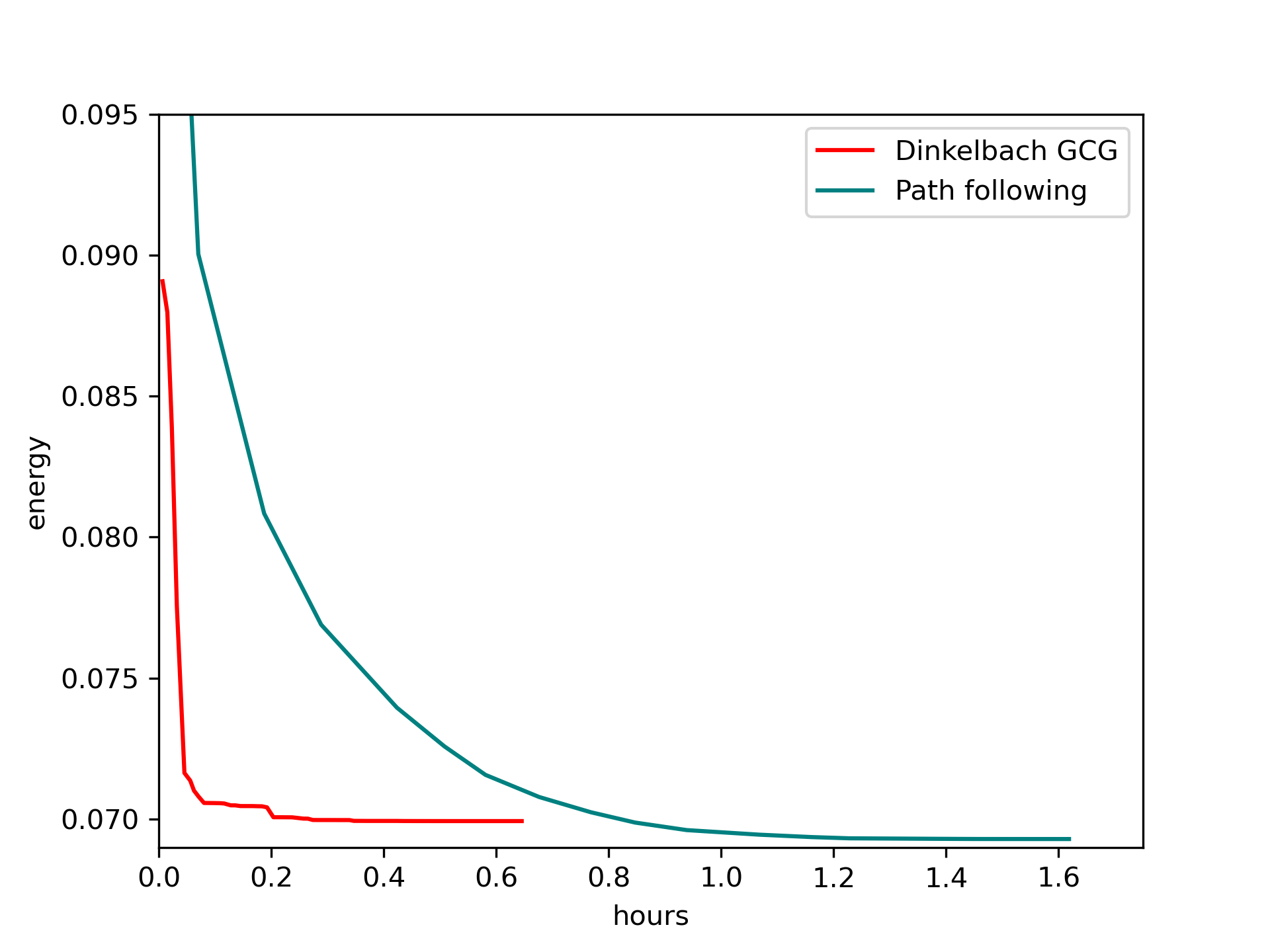}}
    \subfigure[triangulation of $\approx 5 \cdot 10^5$ faces]{\includegraphics[width=0.49\textwidth]{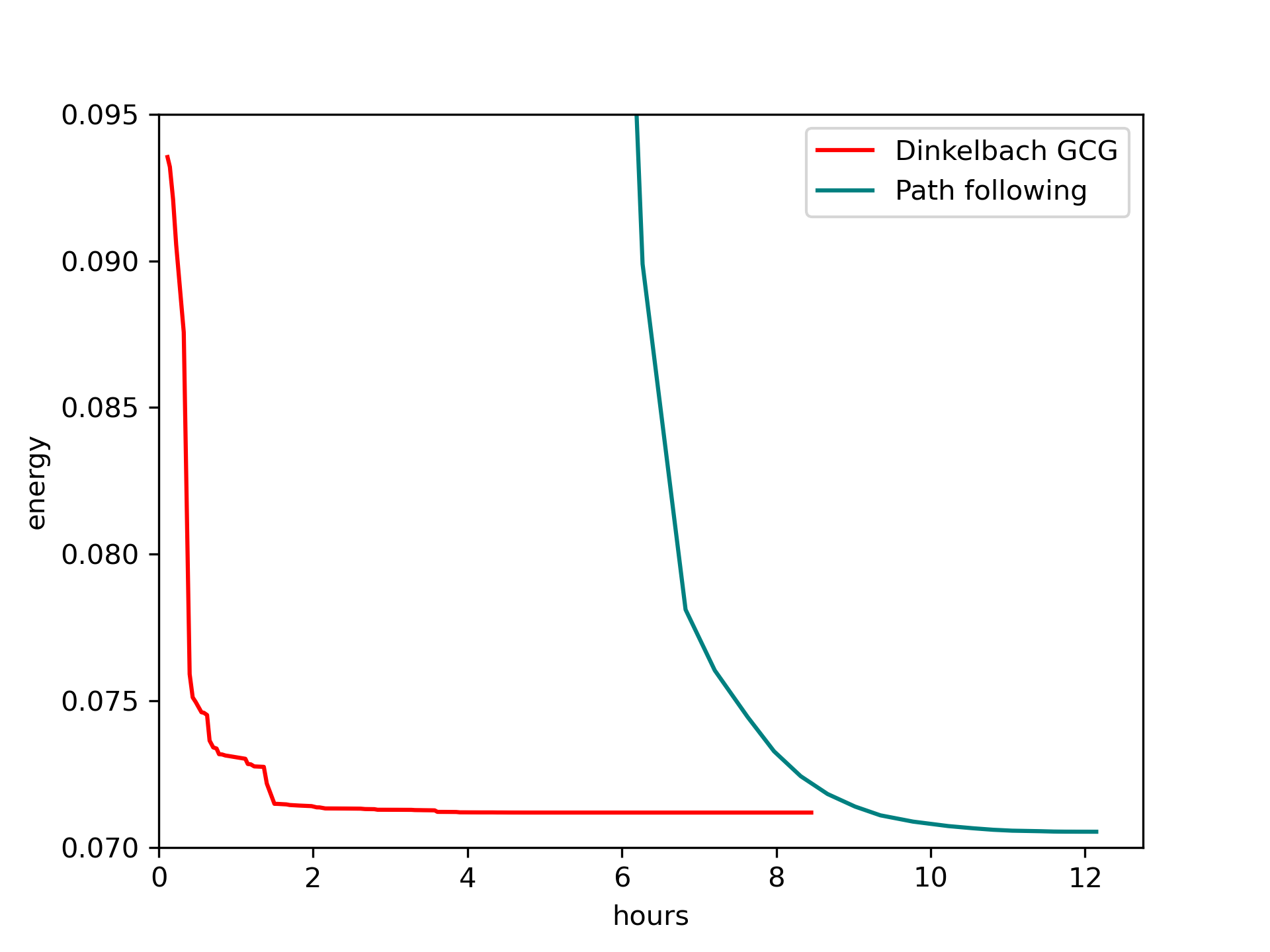}}
    \caption{Comparison in terms of computational time}
    \label{fig:comparison}
\end{figure}

\begin{figure}[!htb]
    \centering
    \subfigure[Two spheres, Figure~\ref{fig:2spheres}]{\includegraphics[width=0.49\textwidth]{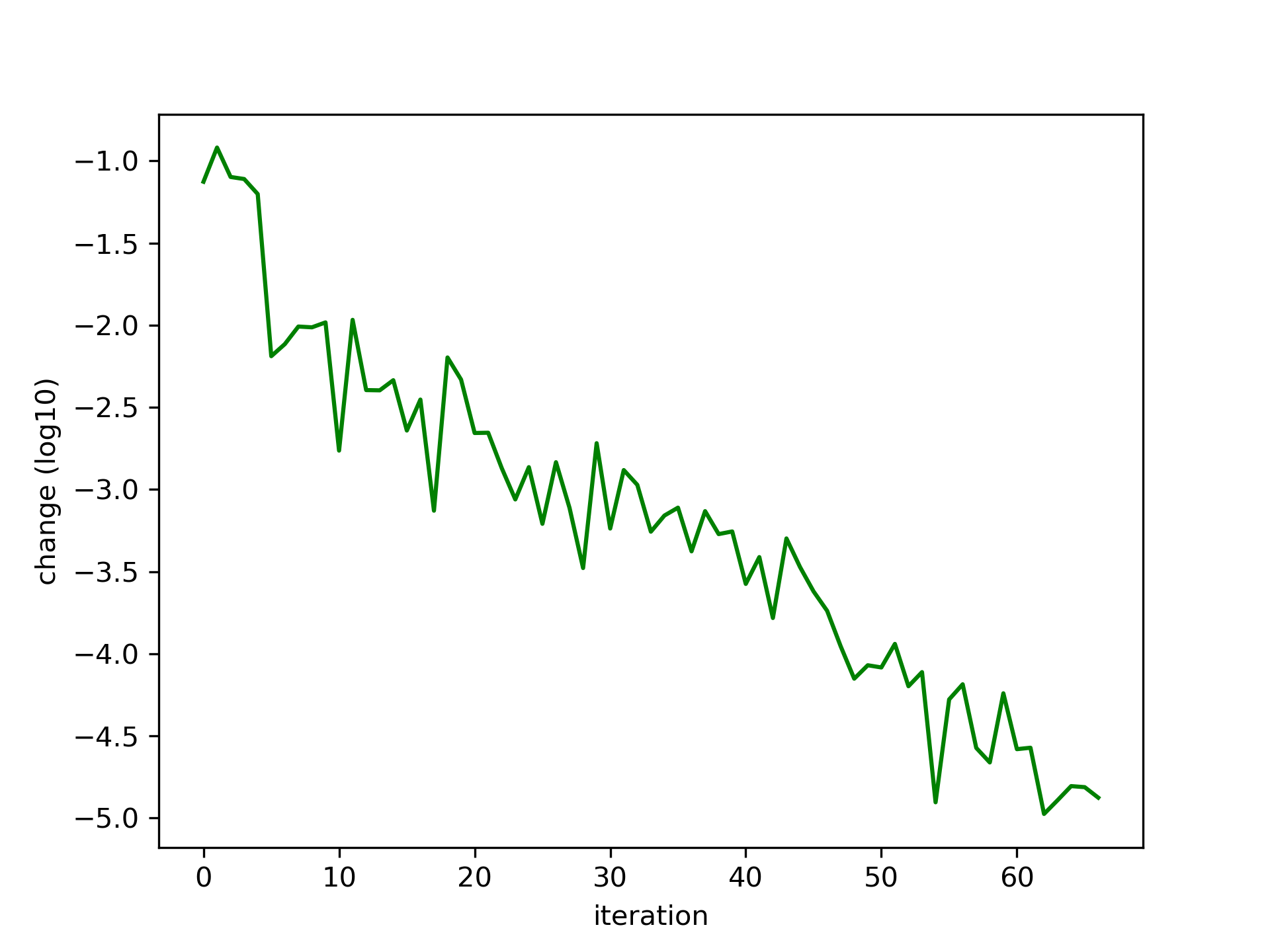}}
    \subfigure[Castle, Figure \ref{fig:castle}.]{\includegraphics[width=0.49\textwidth]{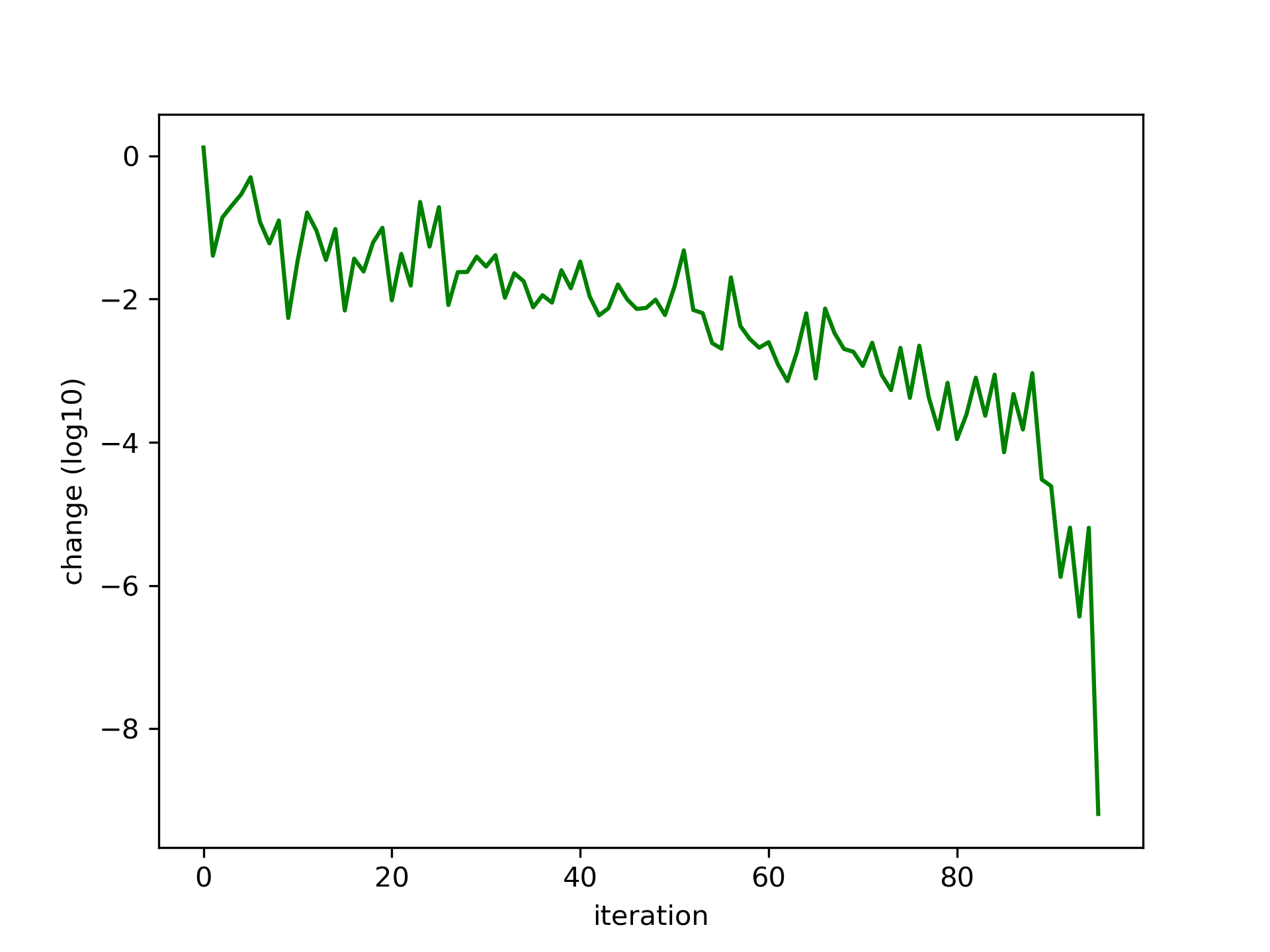}}
    \caption{Relative change~$\|u_{k+1}-u_k\|_{L^1(\Omega)}/\|u_k\|_{L^1(\Omega)}$.}
    \label{fig:relchangeplot}
\end{figure}

These results suggest that, after few iterations,~$u_k$ only changes slightly, which we suspect to be a consequence of clustering effects in the method, i.e. for~$k$ large enough, we have that~$\min_{u \in \mathcal{A}_k} \|u^k_{n^+_k}-u\|_{L^1(\Omega)}$ is very small. In this case, faster convergence can be expected if additional ``sliding'' steps are included, i.e., instead of adding new scaled characteristic functions to the active set in very iteration, we also allow, in some sense, optimal deformations of those already included in~$\mathcal{A}_k$. We again point out~\cite{CasDuvPet22} for a potential practical realization of such a step for simpler deconvolution problems. However, similarly to sparse minimization problems, see, a straightforward inclusion of a sliding step might prove difficult due to the PDE-constraint. Moreover, additional speedup can potentially expected if the graph cut method can be warmstarted using elements of the active set as initial guess. This is, however, not supported by the PyMaxflow library we used in our implementation. Since the present work is conceptual and puts the focus on the feasibility of FC-GCG methods for total variation regularization, both topics, sliding steps as well as warmstarting of the graph cut, are postponed to future work. Finally, it's important to acknowledge that the path-following seemingly yields smaller function values if run long enough. We attribute this behavior to the selection of different minimization spaces in both implementations, $V=P_0(\T)$ for Algorithm~\ref{alg:dinkelbachgcg} and $V=P_1(\T)$ for~\cite{HafMan22}, as well as the resulting anisotropic nature of the total variation, see Section \ref{sec:disccont}, rather than inaccuracies in Algorithm~\ref{alg:dinkelbachgcg}.

\section*{Acknowledgements}
The authors are thankful to Vincent Duval for a fruitful discussion on the Dinkelbach method, and to Annika Bach for insightful comments on discrete-to-continuum convergence of variational problems on lattice systems. 
\section*{Declarations}
The authors have no relevant financial or non-financial interests to disclose.  Python code for our implementation and configurations to reproduce some of the following examples are available at \url{https://doi.org/10.5281/zenodo.10048384}.
\appendix
\section{Proofs related to Section~\ref{sec:existence}}\label{app:convproofs}
In this section, we provide the remaining technical proofs of Section~\ref{sec:existence} which were omitted until now for the sake of brevity.
\begin{proof}[Proof of Proposition \ref{prop:propofparF}]
The weak-to-strong continuity of~$K$ follows immediately from Assumption~\ref{Ass:basics},~$\textbf{A}1$. Now, in order to show that~$\Bar{c}$ is at least two times continuously differentiable, note that Assumption~\ref{Ass:basics},~$\textbf{A}3$, implies~$K_\1 \neq 0$. Consider the nonlinear function
\begin{align*}
    H(y,c)=\int_\Omega K^* \nabla F(y+ c K_{\1})~\mathrm{d}x
\end{align*}
Since~$F$ is thrice continuously differentiable,~$H(\cdot,\cdot)$ is of class~$\mathcal{C}^2$. According to Remark~\ref{rem:firstorderforcoeff}, there holds~$H(y,\Bar{c}(y))=0$ for all $y\in Y$. Moreover, the strict convexity of~$F$ implies
\begin{align*}
    \partial_c H(y,c)= (K_\1,\nabla^2 F(y+\bar{c}(y) K_1) K_\1)_Y >0.
\end{align*}
Hence, the differentiability of~$\Bar{c}$ is a consequence of the implicit function theorem as well as of the smoothness of~$F$.
Noting that
\begin{align*}
    \mathcal{F}(y)= F(y+\bar{c}K_\1)
\end{align*}
we immediately conclude that~$\mathcal{F}$ is two times continuously Fr\'echet differentiable and its gradient is Lipschitz continuous on compact sets. Furthermore, the chain rule together with Remark~\ref{rem:firstorderforcoeff} implies
\begin{align*}
 (\nabla \mathcal{F}(y), \delta y)_Y &= (\nabla F(y+\bar c(y) K_{\1}), \delta y)_Y+   (\nabla \Bar{c}, \delta y)_Y \int_\Omega K^* \nabla F(y+\bar c(y) K_{\1})~\mathrm{d}x \\&= (\nabla F(y+\bar c(y) K_{\1}), \delta y)_Y
\end{align*}
for all~$\delta y \in Y$ which finally yields the characterization of~$\nabla \mathcal{F}(y)$. The convexity of $\mathcal{F}$ follows by direct computation.
\end{proof}

\begin{proof}[Proof of Proposition \ref{prop:sublinPDAP}]
If Algorithm \ref{alg:abstractgcg} stops after $\Bar{k}$ iterations, we verify that the function $\mathring{u}_{\Bar{k}}$ is a minimizer of~\eqref{def:BVpropmean} following the steps in~\cite[Proposition 3.1]{BreCarFanWal23}. Consequently, according to Proposition~\ref{prop:equivalence}, the function $u_{\Bar{k}}=\mathring{u}_{\Bar{k}}+c^{\Bar{k}}$ is a solution of~\eqref{def:BVprob}. Now assume that Algorithm~\ref{alg:abstractgcg} does not stop after finitely many iterations. In this case, using the definition of $\mathcal{F}$ and $c_k$ as well as the convexity of the former and the Lipschitz continuity of its gradient, we conclude  
\begin{align*}
     J(u_k) -\min_{u \in \LO} J(u)= \mathcal{J}(\mathring{u}_k)- \min_{v \in \LrO} \mathcal{J}(v) \leq \frac{c}{1+k}.
\end{align*}
in the same way as in \cite[Theorem 3.3]{BreCarFanWal23}. Moreover, the sequence $\mathring{u}_k$ is bounded in $\LrO$ and we note that $u_k$ satisfies
\begin{align*}
    u_k= \mathring{u}_k+ c^k = \mathring{u}_k+ \bar{c}(K\mathring{u}_k).
\end{align*}
Since~the mapping $\bar{c}$ is continuous and $\{K \mathring{u}_k\}_k$ is strongly compact, we conclude that the sequence $u_k$ is bounded and thus admits at least one subsequence converging weakly in $L^q(\Omega)$, denoted by the same subscript, and with limit $\bar{u} \in V$, since $V$ is weakly closed. Since $J$ is weakly* lower semicontinuous on $\BV(\Omega)$ and $K$ is weak-to-strong continuous, we conclude
\begin{align*}
    0 \leq J(\Bar{u})- \min_{u \in V} J(u) \leq \liminf_{k \rightarrow \infty} \lbrack J(u_k)- \min_{u \in V} J(u) \rbrack =0, \quad F(Ku_k) \rightarrow F(K\Bar{u}). 
\end{align*}
Hence, $\bar{u}$ is a minimizer of~\eqref{def:BVprob} and $\TV(u_k,\Omega) \rightarrow \TV(\Bar{u},\Omega)$, i.e., $u_k$ converges strictly. 
Finally, since the selected subsequence was arbitrary, we conclude the minimality of every strict accumulation point.
\end{proof}

\begin{proof}[Proof of Lemma \ref{lem:upperbound}]
Let $\bar{u}$ denote a minimizer of~\eqref{def:BVprob}.
By one homogeneity and convexity of $\TV(\cdot,\Omega)$ as well as $\TV(u^k_j)=1$, there holds
\begin{align*}
     \TV(u_k,\Omega) \leq \sum^{N_k}_{j=1} \gamma^k_j. 
\end{align*}
As a consequence, we conclude
\begin{align*}
     \TV(\Bar{u}) \leq J(\bar u) \leq J(u_k) \leq F\left(K\mathcal{U}_{\mathcal{A}_k}(\gamma^k) + c^k K_{\1}\right)+\sum_{j=1}^{N_k} \gamma^k_j,
\end{align*}
i.e., $\TV(\Bar{u}) \leq M_k$. By convexity of $F$, there holds
\begin{align*}
    J(u_k)- \min_{u \in \LO} J(u) \leq \int_\Omega p_k (\bar{u}-u_k)~\dd x+ \sum_{j=1}^{N_k} \gamma^k_j-  \TV(\bar{u},\Omega)= \int_\Omega p_k \Bar{u}~\dd x- \TV(\bar{u},\Omega)
\end{align*}
where the equality follows from the first-order optimality conditions of $(\mathcal{P}(\mathcal{A}_k))$. We further estimate
\begin{align*}
    \int_\Omega p_k \Bar{u}~\dd x- \TV(\bar{u},\Omega) \leq  \TV(\bar{u},\Omega) \left( \max_{v \in B} \int_\Omega p_k v~\dd x-1\right) \leq M_k \left( \int_\Omega p_k \widehat{v}_k~\mathrm{d}x -1\right),
\end{align*}
finishing the proof.
\end{proof}

\section{Discrete geodesics on regular triangulations}\label{app:discretegeodesics}
In this section, we give a formal proof for the claimed properties of length-minimizing piecewise linear and continuous curves connecting the points~$(0,0)$ and~$(1,\sigma/\tau)$ along interior edges of the triangulation~$\mathcal{T}_{k_m}$. More in detail, with a slight abuse of notation, define~$\mathcal{E} \coloneqq \bigcup_{T \in \mathcal{T}_{k_m}} \partial T$ and denote by~$\mathcal{N}$, the set of nodes of the mesh. Moreover, recalling the reasoning in the proof of Proposition~\ref{prop:octagon}, it is sufficient to consider piecewise linear, continuous and injective curves~$\gamma \colon [0,1] \to \mathcal{E} $. Let the set of those curves be denoted by~$\mathcal{C}(\mathcal{E})$. Note that each of these curves is differentiable everywhere apart from a finite ordered set
\begin{align*}
    D^{0,1}_{\gamma} \coloneqq \{t_0^{\gamma},\cdots, t^{\gamma}_{N_\gamma+1} \} \quad \text{with} \quad 0=t_0^{\gamma} < \cdots < t^{\gamma}_{N_\gamma+1}=1,
\end{align*}
as well as~$\gamma(D^{0,1}_{\gamma}) \subset \mathcal{N}$. Its length,~$L(\gamma)\coloneqq \mathcal{H}^1(\gamma([0,1]))$, is given by
\begin{align*}
    L(\gamma)= \sum^{N_\gamma}_{j=0} |\gamma(t^\gamma_j)-\gamma(t^\gamma_{j+1})|.
\end{align*}
With these prerequisites and for some coprime integers~$\sigma,\tau \in \Z$ with~$0<\sigma <\tau$, consider the length minimization problem 
\begin{align} \label{def:geodappendix}
    \min_{\gamma \in \mathcal{C}(\mathcal{E})} L(\gamma) \quad \text{s.t.}\quad \gamma(0)=(0,0), \quad \gamma(1)=(1,\sigma/\tau)
\end{align}
as well as
\begin{align} \label{eq:nonboundary}
    \gamma(t) \in \mathcal{E} \setminus \left( \left((0,0),(1,0) \right\rbrack \cup \left((1,0),(1,\sigma/\tau) \right) \right) \quad \text{for all} \quad t \in [0,1].
\end{align}
Since its admissible set is finite, the problem admits at least one solution. Indeed, the following characterization of its minimizers holds true.
\begin{theorem} \label{thm:optgeod}
    An admissible curve~$\gamma \in \mathcal{C}(\mathcal{E})$ is a minimizer of~\eqref{def:geodappendix} if and only if there holds
    \begin{align} \label{eq:necessgeodesic}
        \frac{\gamma(t^\gamma_{j+1})-\gamma(t^\gamma_{j})}{|\gamma(t^\gamma_{j+1})-\gamma(t^\gamma_{j})|}= \nu_j \quad \text{for all} \quad j=0,\dots, N_\gamma
    \end{align}
 where~$\nu_j \in \{\nu^1, \nu^2\}$ and~$\nu^1= (1/\sqrt{2})(1,1),~\nu^2=(1,0)$.
\end{theorem}
Note that it suffices to show that a minimizer~$\Bar{\gamma}$ of~\eqref{def:geodappendix} necessarily satisfies~\eqref{eq:necessgeodesic} since all admissible curves~$\gamma \in \mathcal{C}(\mathcal{E})$ obeying~\eqref{eq:necessgeodesic} have the same length
\begin{align*}
    L(\gamma)= 1+\left(\sqrt{2}-1 \right) \frac{\sigma}{\tau}.
\end{align*}
The former is a direct consequence of the following proposition: 
\begin{proposition}\label{prop:conegeod}
    Let~$\Bar{\gamma} \in \mathcal{C}(\mathcal{E})$ be an optimal path and denote by~$\Bar{\gamma}_i$,~$i=1,2$ its component functions. For every~$j=1,\dots,N_{\bar \gamma}$ and all~$t \geq t^\gamma_j$, we have
    \begin{align*}
            \Bar{\gamma}(t) \in \{\Bar{\gamma}(t^\gamma_j)\}+ \operatorname{cone}\{\nu_1, \nu_2\}.
    \end{align*}
    \end{proposition}
    \begin{proof}
    In the following, we make frequent use of the fact that we can concatenate piecewise linear and continuous paths at nodes of triangulation, provided that both parts satisfy the correct boundary conditions. We start by proving that there holds 
    \begin{align*}
        \Bar{\gamma}_2(t) \leq \sigma/\tau, \quad \Bar{\gamma}_1(t) \leq 1- \sigma/\tau+ \gamma_2(t) \quad \text{for all} \quad t \in [0,1].
    \end{align*}
    Since~$\Bar{\gamma}$ is a piecewise linear and continuous path between nodes and satisfies the boundary condition~$\Bar{\gamma}(1)=(1,\sigma/\tau)$, it suffices to show that these estimates are satisfied whenever~$\Bar{\gamma}(\Bar{t}) \in \mathcal{N}$ for some~$\Bar{t} \in [0,1]$. Indeed, assume that there is~$\Bar{t} \in (0,1)$ such that~$\Bar{\gamma}_2(\Bar{t})>\sigma/\tau$ and~$\Bar{\gamma}(\Bar{t}) \in \mathcal{N}$. Then, by continuity of~$\Bar{\gamma}$ as well as due to the boundary conditions, there are~$\Bar{t}_1< \Bar{t} <\Bar{t}_2$ with~$\Bar{\gamma}(\Bar{t}_i) \in \mathcal{N}$ and~$\Bar{\gamma}_2(\Bar{t}_i)=\sigma/\tau$,~$i=1,2$. 
    Define the horizontal augmented path
    \begin{align*}
        \Tilde{\gamma}(t)= \begin{cases}
            \Bar{\gamma}(t) & t \in [0, \Bar{t}_1] \cup [\Bar{t}_2,1] \\
            \frac{t-\bar{t}_2}{\Bar{t}_1-\Bar{t}_2}\Bar{\gamma}(\Bar{t}_1)- \frac{t-\bar{t}_1}{\Bar{t}_1-\Bar{t}_2} \Bar{\gamma}(\Bar{t}_2) & \text{else}
        \end{cases}.
    \end{align*}
    Then~$\Tilde{\gamma}$ is admissible for~\eqref{def:geodappendix} but there holds~$L(\tilde{\gamma})<L(\bar{\gamma})$, contradicting the minimality of~$\Bar{\gamma}$. For the estimate on~$\Bar{\gamma}_1$, we proceed analogously. Now, due to the constraint~\eqref{eq:nonboundary} as well as the derived estimates, we conclude
        \begin{align*}
            (1, \sigma/\tau) \in \{\Bar{\gamma}(t^\gamma_j)\}+ \operatorname{cone}\{\nu_1, \nu_2\}
        \end{align*}
    for all~$j=1,\dots,N_{\bar\gamma}$, noting that for every~$x \in \Omega$ with~$x_2 \leq \sigma/\tau$ and~$x_1 \leq 1- \sigma/\tau+ x_2$, the set
    \begin{align*}
        \left(\{x\}+ \operatorname{cone}\{\nu_1, \nu_2\} \right) \cap \Omega
    \end{align*}
    is a triangle containing~$(1,\sigma/\tau)$. In order to prove the statement claimed in the theorem, and w.l.o.g, assume that there is~$t^\gamma_j \in D^{0,1}_{\bar\gamma}$ as well as~$\Bar{t}>t^\gamma_j$ with~$\Bar{\gamma}(\Bar{t})\in \mathcal{N}$ but
    \begin{align*}
        \Bar{\gamma}(\Bar{t})\not\in \{\Bar{\gamma}(t^\gamma_j)\}+ \operatorname{cone}\{\nu_1, \nu_2\}.
    \end{align*}
    Then, due to continuity of~$\Bar{\gamma}$ as well as
    \begin{align*}
        (1,\sigma/\tau)\in \{\Bar{\gamma}(t^\gamma_j)\}+ \operatorname{cone}\{\nu_1, \nu_2\},
    \end{align*}
    we have that~$\Bar{\gamma}$ necessarily intersects
    \begin{align*}
     \{\Bar{\gamma}(t^\gamma_j)\}+ \operatorname{cone}\{\nu_1\}   \quad \text{or} \quad \{\Bar{\gamma}(t^\gamma_j)\}+ \operatorname{cone}\{ \nu_2\}
    \end{align*}
    for some~$t>\Bar{t}_2$. Repeating the previous construction, i.e. augmenting~$\Bar{\gamma}$ along~$\nu_1$ or~$\nu_2$, respectively, again yields a contradiction. This finishes the proof.
    \end{proof}
    \begin{proof}[Proof of Theorem~\ref{thm:optgeod}]
     Assume that there is~$t^\gamma_j \in D^{0,1}_{\Bar{\gamma}}$ such that~\eqref{eq:necessgeodesic} does not hold. Then we also have
     \begin{align*}
         \Bar{\gamma}(t^\gamma_{j+1})\not\in \{\Bar{\gamma}(t^\gamma_j)\}+ \operatorname{cone}\{\nu_1, \nu_2\}
     \end{align*}
     yielding a contradiction to Proposition~\ref{prop:conegeod}.
    \end{proof}
\bibliographystyle{plain}
\bibliography{gcmcg}

\end{document}